\documentclass{scrartcl}
\usepackage{amsmath,amssymb,amsthm}
\usepackage[utf8]{inputenc}
\usepackage[T1]{fontenc}
\usepackage{comment}
\usepackage{graphicx}
\usepackage{verbatim}
\usepackage{faktor}
\usepackage{calc}
\newcommand{\sigline}[1]{\makebox[\widthof{#1~}]{.\dotfill}\\#1}

\usepackage[sc]{mathpazo}

\usepackage[english]{babel}

\usepackage[shortlabels]{enumitem}
\usepackage{color}
\usepackage{listings} 
\usepackage{ stmaryrd }
\usepackage{commath}

\usepackage{algorithm}
\usepackage{algorithmic}

\usepackage[colorlinks=true, linkcolor=black, urlcolor=blue]{hyperref}

\newcommand\N{\mathbb N}

\newcommand\R{\mathbb R}
\newcommand\C{\mathbb C}

\newcommand\E{\mathbb E}
\newcommand\I{\mathcal I}

\newcommand\Sn{\mathbb S^{n-1}}
\newcommand\sca[1]{\langle #1 \rangle}
\newcommand{\alist}[2][n]{{#2}_1,\ldots,{#2}_{#1}}
\newcommand{\aalist}{{a}_1,\ldots,{a}_{m}}
\newcommand{\mlspace}{\text{Hom}(\R^n,\ldots,\R^n; \R)}
\newcommand{\symlspace}{\text{Sym}(\R^n,\ldots,\R^n; \R)}
\newcommand{\lspace}{\text{Hom}(\R^n; \R)}
\newcommand{\polyspace}[1]{\R[X]_{\leq #1}}
\newcommand{\hompolyspace}[1]{\R[X]_{= #1}}
\newcommand{\sos}[1]{\sum \R[X]_{\leq \frac{#1}{2}}^2}

\newcommand{\moment}{\mathcal{M}}
\newcommand{\tensdecomp}[2][d] {\sum_{i=1}^m \sca{{#2}_i, X}^{#1}}
\newcommand{\weightdecomp}[1][d] {\sum_{i=1}^m \lambda_i \sca{{a}_i, X}^{#1}}
\newcommand{\rezprod}[2] {\langle #1 \mid #2 \rangle_F}
\newcommand{\linfunc}[1] {\phi_{#1}}
\newcommand{\linred}[1] {\phi_{#1 \tens \id}}
\newcommand{\transp}[1] {#1^T}
\newcommand{\supnormI}[1]{\|#1\|_{\infty, \I}}
\newcommand\chebystandard{\mathcal T_{d}}
\newcommand\chebyinterval[1][d]{\mathcal{T}_{\I, #1}}
\newcommand\Chebtest[1][d]{\mathcal C_{#1}}
\newcommand\wtest{W_{\text{test}}}
\newcommand{\Wst} {W^\ast}
\newcommand{\wminmax} {\mathfrak{w}}
\newcommand{\wminmaxtest} {\wminmax}
\newcommand{\speccorrel} {\rho_{\text{spec}}}
\newcommand{\minspeccorrel} {\rho_{\text{minspec}}}
\newcommand{\bestratio} {\mathfrak{r}}
\newcommand{\spec}[1]{\|#1\|_{\specn}}

\newcommand{\Mspec}{\|M\|_{\specn}}
\newcommand{\Mdiffspec}{\|M-a_j\transp{a_j}\|_{\specn}}

\newcommand{\kmin}{\kappa_{\min}}
\newcommand{\kmax}{\kappa_{\max}}

\newcommand{\wlistdec} {\lambda_1 a_1, \ldots , \lambda_m a_m}
\newcommand{\landauO}[1]{\mathcal{O}(#1)}
\newcommand{\landauOmeg}[1]{\Omega(#1)}

\newcommand{\landaueq}[1]{\theta(#1)}

\newcommand{\nummonoms}[1]{\binom{n+#1}{#1}}
\newcommand{\numhommonoms}[1]{\binom{n+#1-1}{#1}}
\newcommand{\monoms}[1]{\mathfrak{m}_{\leq #1}}
\newcommand{\monX}[1]{X^{\small \leq #1}}
\newcommand{\case}[4]{\begin{Bmatrix} #1, & \text{#2} \\#3, & \text{#4} \end{Bmatrix}}

\newcommand{\lfp}[1]{\sca{{#1}, X}}
\newcommand{\lfv}[1]{\sca{a_{#1}, v}}

\newcommand{\dvobfun}{\Delta_{j,i,f}}
\newcommand{\obfun}{f(a_i)}
\renewcommand{\norm}[2][]{\|#2\|_{#1} }
\newcommand{\spheredist}[1]{d_{\Sn}(#1)}
\newcommand{\projdist}[1]{d_{\mathbb{P}^{n-1}}(#1)}

\renewcommand\P{\mathcal P}

\newcommand\al{\alpha}

\newcommand\ph{\varphi}

\newcommand\tens{\otimes}

\theoremstyle{plain}
\newtheorem{sat}{Theorem}[section]
\newtheorem{lem}[sat]{Lemma}
\newtheorem{pro}[sat]{Proposition}
\newtheorem{defi}[sat]{Definition}
\newtheorem{notation}[sat]{Notation}
\newtheorem{remark}[sat]{Remark}
\newtheorem{kor}[sat]{Corollary}
\newtheorem{defprop}[sat]{Definition and Proposition}

\DeclareMathOperator\sgn{sgn}
\DeclareMathOperator\argmax{argmax}

\DeclareMathOperator\id{id}

\DeclareMathOperator\prob{\mathbb{P}}

\DeclareMathOperator\hausdist{d_H}

\DeclareMathOperator\im{im}

\DeclareMathOperator\specn{spec}

\DeclareMathOperator{\maxi}{maximise}

\DeclareMathOperator{\conv}{conv}

\begin{document}

\title{A new Algorithm for Overcomplete Tensor Decomposition based on Sums-of-Squares Optimisation}
\author{\Large Alexander Taveira Blomenhofer\vspace{9pt}\\ \large Master's Thesis in Real Geometry and Algebra \vspace{4pt} \\ \Large Universität Konstanz \vspace{4pt}\\ \large Supervised by: Prof. Dr. Markus Schweighofer\vspace{2pt}\\
\normalsize Submitted: October 5, 2018\\ \normalsize Last Revised: December 13, 2018\vspace{0pt} }
\date{}
\maketitle
\textbf{\centering Abstract\\}
\begingroup
\leftskip3em
\rightskip \leftskip
{Every  symmetric tensor $T$ of degree $d$ may be represented as a linear combination  $T = \sum_{i=1}^m \lambda_i\: a_i\tens\ldots\tens a_i$ of $d $-th tensor powers of vectors $a_i\in\R^n$.  The  task of finding such $a_i $ (when $T$ is given) is called the \emph{tensor decomposition problem}. Tensor decomposition has a broad range of applications:  Symmetric tensors occur  naturally e.g. as moment  tensors of probability  measures  and tensor decomposition techniques can be used to find quadrature rules for them.  However,  tensor decomposition is also a computationally demanding task,  particularly in the so-called \emph{overcomplete setting},  where $ m  > n $.  The approximation algorithms achieving the best known guarantees in this setting are based on  the sums of squares (SOS) programming  hierarchy, using the fact that symmetric tensors correspond to homogeneous polynomials, i.e. $\sum_{i=1}^m \lambda_i\: a_i\tens\ldots\tens a_i \: \longleftrightarrow \: \sum_{i=1}^m \lambda_i \lfp{a_i}^d$. 
	
In this work, a new class of algorithms based on SOS programming is developed. These allow to reduce a degree-$d$ homogeneous polynomial $T = \tensdecomp[d]{a}$ to (something close to) a rank-$1$ quadratic form via a reduction  polynomial $W\in\sum \R[X]^2$. $W$ can be thought of as a ``weight function'' attaining high values on merely one  of the components $a_i$. The component can then be extracted by running an eigenvalue decomposition on the quadratic form $\sum_{i=1}^m W(a_i) \lfp{a_i}^2$.}

\par
\endgroup
\newpage

\tableofcontents
\newpage

\section{Introduction}

\paragraph{What is tensor decomposition?}

\emph{Tensor decomposition} is the problem of finding a preimage\footnote{The $a_i$ can only be recovered up to reordering and double occurences  ~--~ this is suggested by the use of set notation. It is further reasonable to assume generally that no two $a_i$ are multiples of each other. Note that if $d$ is even, then we must further accept that we will not of course be able to distinguish between $a_i$ and $-a_i$. Even taking these effects aside, there will be many preimages in general.} in the assignment
\begin{align}\label{eq:tensdec}
	\{a_1, \ldots, a_m\} \mapsto \tensdecomp[d]{a}
\end{align}
where $a_i \in\R^n$ are (distinct) vectors not being multiples of each other and the image $T = \tensdecomp[d]{a}\in \hompolyspace{d} = \R[\alist[n]{X}]_{=d}$ is a homogeneous polynomial in the variable vector $X = \left( \alist[n]{X} \right)$. The polynomial $T$ is called a \emph{tensor}\footnote{The usual approach from category theory defines tensors as members of an abstract space $\underbrace{\R^n \tens \ldots \tens \R^n}_{d \text{ times}}$ fulfilling a universal property. It can be shown that the space of degree-$d$ homogeneous polynomials in \emph{noncommuting} variables is indeed a model of this property. However, in this whole thesis we will always work with symmetric $d$-tensors. It can be shown that these correspond precisely to the homogeneous polynomials in \emph{commuting} variables, i.e. the space $\hompolyspace{d}$} and the smallest such $m$ is called the \emph{rank} of $T$.  $\sca{\cdot, \cdot}$ denotes the scalar product of two real vectors.

To get a better intuition for the problem, think of a probability measure $\mu =\frac{1}{m} \sum_{i = 1}^m \delta_{a_i}$ which is finitely supported on the (distinct) \emph{nodes} $a_i \in \R^n$. Suppose we do not know the 
$a_i$ and the only thing we are given are samples from a vector $Y\sim \mu$ being a \mbox{$\mu$-distributed} random variable. Can we find out where $\mu$ is supported ~--~ i.e. ~--~ can we compute the $a_i$ from that information? \\  
Note that the degree $k$-moments $\E_\mu[Y^\al]$ of $Y$ (where $\al \in\N_0^n$ is a multi-index of length $|\al| = k$) essentially form a tensor for all $k\in \{0, \ldots, d\}$.
\begin{align}
	\sum_{|\al| = k} \E_{\mu}[Y^\al] X^\al
\end{align}
It can be shown that the following identity holds:
\begin{align}
\sum_{|\al| = k} \E_{\mu}[Y^\al] X^\al = \frac{1}{m} \tensdecomp[k]{a}
\end{align}
by expanding the right hand side and explicitly computing the moments on the left hand side.

Now the key point is the following: If we can draw sufficiently many samples from $Y$, then we may assume that we know the left hand side (by the central limit theorems of probability theory) ~--~ up to a small noise term. If we can then run a tensor decomposition algorithm on the left hand side, we will indeed get the support nodes approximately out of $\mu$. 

At first sight, tensor decomposition might seem like an impossible task, especially to people used to \emph{matrix decompositions}. This is due to the fact that if a quadratic form $Q$ (corresponding to a symmetric matrix $M \in \R^{n\times n} \cong \R^n \tens \R^n$) admits a decomposition $Q = \sum_{i = 1}^m \lfp{a_i}^2$ (corresponding to $M = \sum_{i = 1}^m a_i\transp{a_i}$), then for most of the time this decomposition will not be unique. The reason for this is that we are able to perform an orthogonal transformation: Write $A = (\alist[m]{a}) \in \R^{n\times m}$. Note that $M = A\transp{A}$. 
Then for any orthogonal matrix $R$,
\[
M = \sum_{i = 1}^m a_i\transp{a_i} = A\transp{A} = \left(AR\right) \transp{\left(AR\right)}
\]
Hence the columns of $AR$ will, too, form a matrix decomposition of $M$. If $m, n > 1$ and the $a_i$ are generic vectors, then the columns of $AR$ will always form  another decomposition, unless $R$ was a permutation matrix (in that case only the order of the $a_i$ would change). Thus in general 
the operation $A\mapsto AR$ will yield a decomposition distinct from $\{\aalist\}$.\\

However, there is an obvious reason why unicity fails: The second-order moments are just not enough data. If only these are known, then $\mu$ could as well be, for instance, the Gaussian measure $\mu \sim \mathcal{N}(\frac{1}{m} \sum_{i = 1}^m a_i, \frac{1}{m}M)$ with mean $\frac{1}{m} \sum_{i = 1}^m a_i$ and psd covariance matrix $\frac{1}{m} M$.\footnote{Usually, covariance matrices are required to be positive definite. Yet there is a reasonable extension of the notion of Gaussians for the psd case.}

This points us to the question: Can more data help restoring unicity? There is a natural argument suggesting this: $\mu$ is uniquely determined by the values of \[
\E_{\mu}[f]
\]
for all measurable and integrable functions $f$ on $\R^n$ (since $\mu$ is finitely supported, we can regard all functions $\R^n \to \R$ as measurable and integrable). Actually, we can replace ``all functions $f$'' by 
\[
\E_{\mu}[p]
\]
for all polynomials $p\in \R[X]$. Indeed this is true, since $\mu$ is supported on a compact set, we may approximate any function on $K := \{\alist[m]{a}\}$ by polynomials. Due to polynomial interpolation, we may even restrict to polynomials $p\in \polyspace{d}$ of degree \emph{at most} $d = 2m$: For any $x\in\R^n$, the indicator function $\mathbb{1}_ {x}$ coincides with some nonnegative polynomial $f_x$ on $K\cup \{x\}$, where $f_x$ can be chosen of degree less or equal $d$.\footnote{$2m$ is the degree of the nonnegative multivariate interpolation polynomial $\mathcal{I}_{x} := \prod_{i = 1}^m \frac{\norm{X-a_i}^2 }{\norm{x-a_i}^2 }$  satisfying $\mathcal{I}_{x}(y) = \mathbb{1}_{x}(y) $ for $y\in K\cup \{x\}$. 
We may thus choose $f_x := \mathcal{I}_{x}$.}
At this point, $\mu$ is uniquely determined among all finitely supported measures on $\R^n$ by the values $\E_\mu[f]$ for all polynomials $f$ with $\deg(f) \leq d$: \\
Indeed, if $\nu$ was another finitely supported measure such that $\E_\mu[f] = \E_\nu[f]$ for all polynomials $f$ with $\deg(f) \leq d$, then 
\[
\mu (\{x\}) =\E_\mu[f_x]  = \E_\nu[f_x] \overset{f_x \geq \mathbb{1}_ {x}}{\geq }\E_\nu[\mathbb{1}_ {x}] = \nu (\{x\})  \text{ for } x\in \R^n
\]
Hence $\nu$ is supported on $K$, too (which implies $\E_\nu[\sum_{i = 1}^{m}\mathbb{1}_ {a_i}] = \E_\nu[1]$), and the weight $\nu$ puts on $a_i$ is less than $\mu(\{a_i\})$. Since $\sum_{i = 1}^{m}\nu(\{a_i\}) = \E_\nu[1] = \E_\mu[1] = \sum_{i = 1}^{m}\mu(\{a_i\}) $, this already implies $\mu (\{a_i\}) = \nu (\{a_i\}) $. Thus $\mu = \nu$.\\
Now $\polyspace{d}$ is finite-dimensional, whence we can further reduce to knowing the moments\[
\E_{\mu}[X^\al]
\]
for any multi-index $\al$ of length $|\al| \leq d$. Grouping all these moments together, we see that at this point, $\mu$ is uniquely determined among all measures on $\R^n$ by the degree-$d$ polynomial\[
\mathcal{M}_\mu = \sum_{|\al| \leq d} \E_{\mu}[X^\al] X^\al
\]
Our problem then translates to finding the vectors $\alist[m]{a}$ such that 
\begin{align*}
	\mathcal{M}_\mu = \sum_{k = 0}^d \sum_{i = 1}^m \lfp{a_i}^k
\end{align*}
which can again be shown by expanding the right hand side using the multinomial theorem. To be fair, this is not a tensor decomposition problem as stated in \eqref{eq:tensdec}, but it may be seen as a ``dehomogenised variant'' of \eqref{eq:tensdec}.  From the above considerations, we know that  we can \emph{simultaneously decompose} all of those \emph{moment tensors} $\moment_k = \tensdecomp[k]{a}$ by the same $a_i$ and that this simultaneous decomposition is unique. We call such a problem a \emph{moment decomposition problem}. Moment decomposition and tensor decomposition are closely related problems. 

One might argue that we cheated to get  uniqueness by requiring multiple moments,  and that the issue would look differently if we started with just one tensor. However,  as we will see in \secref{sec:valgotd}, this is essentially not the case:  One  such tensor  of degree $d \geq 2 m$ suffices to  recover $\pm a_i$.  The reason for this is that,  given a  sole tensor  of high degree,  we can generate lower degree tensors $\sum_{i=1}^m \sca{a_i,v}^{d-k} \lfp{a_i}^k$ (where $v\in \Sn$ is some random vector) that  we can use  to cast a tensor decomposition problem into a moment decomposition problem.  These lower degree tensors can  thus be seen as   some sort of ``fake moments''.\footnote{The terminology follows \cite{BS16}, where the authors considered ``fake moments'' of higher degree.}

\paragraph{Undercomplete tensors}
There is another case where it is very easy to see that uniqueness holds: 

\begin{pro}
	Suppose $T = \sum_{i = 1}^m \lfp{a_i}^3$ is a $3$-tensor with $m\leq n$ orthogonal\footnote{This assumption can be relaxed to linear independence of the $a_i$} components $a_i\in \R^n$. Then $\alist[m]{a}$ is the only tensor decomposition of $T$ with $m$ or less components and we can compute it by Jennrich's algorithm (\cite{Har70}, see algorithm \ref{alg:Jennrich}), which is a classical result of tensor decomposition.\footnote{The algorithm was attributed to R. Jennrich in R. Harshman's publication cited above. It seems that R. Jennrich contributed the proof of uniqueness and the basic form of the algorithm, but did not make an own publication as of our knowledge.}
\end{pro}

\begin{algorithm}[h]
	\caption{Jennrich's Algorithm, \cite{Har70}} 
	\label{alg:Jennrich}
	\textbf{Input: } A tensor $T \in \hompolyspace{3}$.\\
	\textbf{Require: } There should exist a tensor decomposition $T = \sum_{i = 1}^m \lfp{a_i}^3$ with pairwise orthogonal components, i.e. $\sca{a_i, a_j} = 0$ for $i\neq j$ and all $a_i \neq 0$. \\
	\textbf{Output: } Pairwise orthogonal vectors $\alist[m]{c}$ satisfying $T = \sum_{i = 1}^m \lfp{c_i}^3$. 
	\\ \vspace{-\baselineskip}
	\begin{algorithmic}[1]
		\STATE{$\mathbf{Compute}$ the symmetric $3$-linear form \[
			\tilde{T}\sca{X, Y, Z} = \sum_{i = 1}^m \sca{a_i, X}\sca{a_i, Y}\sca{a_i, Z}
			\]
			out of $T$ by using polarisation identities. Here $X,Y,Z$ denote vectors of independent unknowns.}
		\STATE{$\mathbf{Choose}$ $v \in \Sn$ uniformly at random.}
		\STATE{$\mathbf{Compute}$ the bilinear form (i.e. matrix) \[
			M := \tilde{T}\sca{X, Y, v} = \sum_{i = 1}^m \sca{a_i, v} \sca{a_i, X}\sca{a_i, Y} = \transp{X} \left(\sum_{i = 1}^m \sca{a_i, v} a_i\transp{a_i} \right)Y
			\]
			by plugging in $v$ for $Z$.} \label{alg:Jennrich:lM}
		\STATE{$\mathbf{Compute}$ the $m$ unit-length eigenvectors $\alist[m]{u} \in \Sn$ corresponding to the nonzero eigenvalues $\alist[m]{\mu}$ of $M$}\label{alg:Jennrich:leigendec}
		\STATE{$\mathbf{Set }$ $c_i := \sqrt[3]{\frac{\mu_i}{\sca{u_i, v}}} u_i$ (where $\sqrt[3]{\:\cdot\:}$ denotes the unique \emph{real} third root). Each of these will be one of the $a_j$. } \label{alg:Jennrich:lcvecs}
		\STATE{$\mathbf{Output }$ $\alist[m]{c}$}
	\end{algorithmic}
\end{algorithm}

\begin{proof}
Let us quickly review why this algorithm will produce correct results. 
Note that with probability $1$ the matrix $M$ in line \ref{alg:Jennrich:lM} will have $m$ distinct eigenvalues. Indeed, due to orthogonality the $a_i$ are eigenvectors of $M$ corresponding to eigenvalues $\mu_i := \sca{a_i, v}\|a_i\|^2$. Since $v$ was uniformly random, hitting the lower-dimensional set $\{v\in \Sn \mid  \sca{a_i, v}\|a_i\|^2 = \sca{a_j, v}\|a_j\|^2 \}$ has probability zero over the choice of $v$ for every two $i \neq j$ (note that the $a_i$ are nonzero). Hence every eigenspace of $M$ is one-dimensional and contains thus precisely two eigenvectors $\pm u$ of unit length. Since $a_i$ was an eigenvector of $M$, we get $u = \pm \frac{a_i}{\|a_i\|}$.\\
This shows that up to reordering, the eigendecomposition computed in line \ref{alg:Jennrich:leigendec} satisfies $u_i = \pm \frac{a_i}{\|a_i\|}$ and $\mu_i = \sca{a_i, v}\|a_i\|^2$. Hence the vector $c_i$ computed in line \ref{alg:Jennrich:lcvecs} satisfies \[
c_i  =  \sqrt[3]{\frac{\mu_i}{\sca{u_i, v}}}u_i = \sqrt[3]{\frac{\sca{a_i, v}\|a_i\|^2}{\sca{\frac{\pm a_i}{\|a_i\|}, v}}}u_i = \sqrt[3]{\pm\|a_i\|^3}u_i = \pm \|a_i\|u_i = a_i
\]
Note that the unknown factor $\pm 1$ of $u_i$ cancels with the one present in $\sca{u_i,v}$ below the root.
Of course the eigenvalue decomposition could yield the eigenvectors in a different order. But this does not matter, since the claim $\:T = \sum_{i = 1}^m \lfp{c_i}^3\:$ we had on the output $\alist[m]{c}$ does not depend at all on the order of the $c_i$. Uniqueness follows from the fact that the resulting $c_i$ are identical to the $a_i$ no matter what representation we chose.
\end{proof}

Note that such a procedure wouldn't be possible if we only had access to the second order input, i.e. $\tensdecomp[2]{a}$. We need that extra degree of freedom a $3$-tensor admits by allowing us to scale the eigenvalues of its matrix reductions.\footnote{The eigenvalues of a random symmetric matrix are distinct with probability $1$, but $m\leq n$ randomly chosen $a_i$ will almost-surely not be orthogonal. Jennrich's Algorithm can though be extended to work for at most $n$ \emph{linearly independent} components: If we get both the $2$nd order moment \emph{and} the degree-$3$ moment of the $a_i$ with, the requirement of orthogonality for uniqueness is, as we will see, without loss of generality, contrary to the matrix case.} Another issue here is the requirement of orthogonality, which allows $T$ to be just of very low rank, i.e. $m\leq n$. In the applications, though, the case $m > n$ is common. Due to the qualitative impact Jennrich's Algorithm has, it is common practice to distinguish the following cases of tensor decomposition:

\begin{enumerate} 
	\item{The \emph{undercomplete} case:} Here $m\leq n$ and thus we can hope for the $a_i$ to be linearly independent (if, for instance, the $a_i$ are  generic vectors, then we'd expect them to be). If we get access to the second-order moment $\moment_2 = \sum_{i = 1}^m a_i\transp{a_i}$ as well (this is (approximately) the case e.g. if our tensor stems from a probability distribution from which we can take samples) we may even
	reduce to the orthogonal case: For simplicity, assume $m = n$.\footnote{The case $m  <  n$ can be dealt with by learning a projection to a  lower-dimensional space (for example out of the  eigendecomposition of $\moment_2$) before employing this argument.} Then the matrix  $\moment_2$ will be positive definite and the $a_i$ will be orthogonal w.r.t. the scalar product defined by \[
	\sca{x, y}_{\moment_2} := \transp{x}\moment_2^{-1}y
	\] 
	Indeed, note that $A = (\alist[m]{a})$ factorises $\moment_2$, that is, $A\transp{A} = \moment_2$. Thus
	$\moment_2^{-1} = \transp{(A^{-1})}A^{-1}$ and $a_i = Ae_i$ satisfy \[
	\sca{Ae_i,Ae_j}_{\moment_2} = \transp{\left(Ae_i\right)}\transp{(A^{-1})}A^{-1} \left(Ae_j\right) = \sca{e_i, e_j}
	\]
	\item{The \emph{overcomplete} case:} Here $n < m \leq \numhommonoms{d}$ holds\footnote{In the applications, the components typically correspond to features of a dataset. To ensure that these features are actually meaningful, it is therefore recommendable to search for a decomposition with a small number $m$ of components. From that point of view, it becomes clear that the case $m \geq \numhommonoms{d} = \dim(\hompolyspace{d})$ isn't of much interest: Then the $\lfp{a_i}^d$ will be linearly dependent in the space $\hompolyspace{d}$. By Caratheodory's theorem from convex geometry, the cone spanned by the forms $\lfp{a_i}^d$ would then also be spanned by $\dim(\hompolyspace{d})=\numhommonoms{d}$-many forms (Cor. 7.4.21 in \cite{RAG}). Thus it could be possible to change the size of some of the components or even make some of them vanish or reappear, which wouldn't fit well with the intuition of them being ``relevant features''. } and therefore the $a_i$ will always be linearly dependent, even if the $\lfp{a_i}^d$ are not. In particular, we can't obtain the $a_i$ as eigenvectors of the same matrix. Overcomplete tensor decomposition is usually the more interesting case, though also way more difficult.
\end{enumerate} 

\paragraph{Overcomplete tensors and the variance problem}

As of now, people are trying to find time-efficient, noise stable and easily implementable algorithms for the overcomplete case.  
Since we have seen that the solution will get unique once the input degree $d$ is high in relation to the number of components (e.g. $d\geq 2m$ will do), 
it is natural to conjecture that the overcomplete case will become easier if we have the possibility to ``increase'' the value of $d$ (while fixing $m$). This is possible e.g. in the \emph{empirical case} where we can estimate any moment from given samples. However, there is at least one problem with that:

Suppose $Y \sim \mu$ is a $\mu$-distributed random vector whose variances are bounded by some $\sigma^2\in \R_{> 0}$. To generate the moment $\E_\mu[Y^\al]$ for some multi-index $\al$ of length $d$ by averaging, we need a number of samples which is, in general, exponential in $d$. Indeed, by Chebyshev's law of large numbers, we have
\[
\prob_\mu\left[|\overline{Y^\al} - \E_\mu[Y^\al]| \geq \tau\right] \leq \frac{\sigma^{2d}}{N\tau^2}
\]
Here we denote by $\overline{Y^\al}$ the average over $N$ i.i.d. copies of the random variable $Y^\al$. If we want to get the moment $\E_\mu[Y^\al]$ up to noise of the magnitude of $\tau$ with at least $0.99$ certainty, then we need (by setting the right hand side equal to $0.01$) 
\[
N \geq \frac{100\sigma^{2d}}{\tau^2} \in \landauO{\frac{\sigma^{2d}}{\tau^2}}
\]
samples.\footnote{The claim of exponentiality follows from the fact that Chebyshev's bound is tight for some random variables with variance $\sigma > 1$.}
If the data is obtained e.g. by pricey physical measurements, generating higher-order moments may literally get expensive.

\paragraph{History of moment generation procedures} \label{sec:history}

The research of B. Barak, D. Steurer and J. Kelner \cite{BKS15} together with R. Ge and T. Ma's \cite{GM15} suggests a way around that. The authors of \cite{BKS15} proposed that, given a degree-$d$ tensor $T = \tensdecomp[d]{a}$, it might be possible to generate higher-order ``fake moments'' of the solution vectors $a_i$ by optimising over the cone of \emph{pseudo-expectations}.

A pseudo-expectation of degree $d$ is a linear functional $\E$ on the space $\polyspace{d}$ satisfying 
\begin{align*}
&(1) \quad \E[1] = 1\\
&(2) \quad \E[P^2] \geq 0 \text{ for all square polynomials }P^2\in \polyspace{d}
\end{align*}
This optimisation can be done by a powerful tool known as \emph{sums of squares (SOS) programming},  which we will discuss in \secref{sec:sos_opt}.
Now, if such a pseudo-expectation would have anything to do with $\mu = \sum_{i = 1}^m \delta_{a_i}$, then we could hope that \[
\E[X^{\tens k}] := \sum_{|\al| = k} \E[X^\al] X^\al \approx \tensdecomp[k]{a}
\]
and thus we could try to run a noise-stable tensor decomposition algorithm on $\E[X^{\tens k}]$. \\

In \cite{BKS15}, the authors demonstrated this technique: To get a $\landaueq{\varepsilon}$-approximate decomposition of a measure supported on $m$ points, they take a tensor $T$ of degree $d \in \landauOmeg{\frac{\log(\sigma) + \tau}{\varepsilon}}$  
(where $\tau$ is some noise parameter. Note that this requirement still assumes that we have a sufficiently high-order tensor $T$ to begin with ~--~ for arbitrary accuracy $\varepsilon$ we'd need to be able to generate actual moments of arbitrary high degree). From that, they compute a degree-$k$ pseudo-expectation $\E$ where $k \approx  \max(4d, \frac{12\log(m)}{\varepsilon})$. 
Then the authors applied a brute-force yet noise stable decomposition algorithm to get one of the $a_i$ (approximately) out of their pseudo-distribution $\E$. However, we encountered an issue with their decomposition algorithm (``sampling from pseudo-distributions'' ~--~ Lemma 5.1ff in \cite{BKS15}), which we will address in \secref{sec:rel_work}. An error in one of the key lemmas broke their decomposition algorithm for the general case.  
Fortunately, they gave  a second algorithm working particularly for the Dictionary Learning problem (``refined sampling from  pseudo-distributions'' – see §7 in \cite{BKS15}. It assumes that samples from a $\mu$-distributed vector $Y$ are given where the distribution should satisfy certain assumptions). 

Such an approach can, as the authors pointed out, only work if  we assume that we can generate the  higher order moments by a reasonable amount of samples. This is why  the authors assumed ``niceness conditions'' on the distribution which imply that moments of high order  $d $   are known up to  a noise constant $\tau $ independent of $d $. Recall that in the worst case $\tau $ would grow exponentially with  $d $. Still, this shows how much can be done with high order moments  and therefore  the question remained whether it would  still be possible to  efficiently generate such higher degree fake moments even if only a  3-tensor is given. 
In particular, consider the case where we are left with a slightly overcomplete degree-$3$ tensor $T$ of rank $m$ in between $n$ and $n^{1.5}$
\[
T = \tensdecomp[3]{a}
\]
Ge and Ma then showed that a quasi-polynomial time procedure can be realised for ``average'' degree-$3$ input tensors $T$, restricting to the case where  the components are chosen randomly from an $n$-dimensional hypercube, precisely  $a_i\in\{\pm \frac{1}{\sqrt{n}} \}^n$. 
Alas, to this end,  they used the sampling procedure of \cite{BKS15} as a key part of their algorithm without giving a proof on their own. Therefore, the only  publication we know of where  higher degree fake moments have  successfully be used to help with the decomposition of 3-tensors is \cite{HSS16},  where the tensor
\[
T^2 = \sum_{i,j= 1}^m \lfp{a_i}^3 \lfp{a_j}^3
\]
is reshaped  and reweighed in a sophisticated manner to generate a proxy for the moment of order 4. But this method does not even use pseudo-expectations anymore.

Despite these issues in the current meta of research,  the reader should have got an idea how  valuable high order moments are and  that sums of squares programming can be a valuable tool in generating them.  In this thesis, we will show that sums of squares programming can also be used to \emph{decompose} high order tensors directly.

This is interesting because it indicates that  in situations very similar to the setting of \cite{BKS15}, we can work directly  on the given actual moments instead of generating higher degree fake moments. This does not bridge the gap in \cite{BKS15} (since  to this end,  one would need to verify that all of the arguments used work  (at least qualitatively) for the fake moments  as well,  which is likely not the case in the current formulation),  but it essentially allows  to solve moment  decomposition problems with the components  lying on the unit sphere when sufficient data is given.

What's even more interesting is that if it would be possible to design two  compatible sums of squares  based procedures,  one for  the generation  part and  another one  for the decomposition part,   then this could  open a whole lot of possibilities. But we are not quite there yet (and we do not even know if or for which tensors this can possibly work) and this is an interesting subject for future research.

\paragraph{Acknowledgements} 
This thesis would not have been possible without the many valuable discussions and the support of my supervisor Markus Schweighofer. Also, I wish to thank Adam Kurpisz for telling me a lot about sums of squares programming  and my parents for ``funding my research''.

\paragraph{Disclaimer} This thesis was written just by myself. When the term ``we'' occurs,  it is either supposed to include the reader or a matter of habit.

\newpage
\section{Overview}

\subsection{Outline of this Thesis}\label{sec:overview}
We will start by introducing  the basic notions and notations of tensor decomposition in \secref{sec:momandtdecomp}. Section \secref{sec:tensor_notation} will cover the connection between homogeneous polynomials and symmetric tensors. In \secref{sec:sos_opt} we give a very brief and rudimentary introduction to sums of squares programming. \\

A common technique in tensor decomposition is to use linear shrinking maps that take a high order tensor and reduce  it to a lower order tensor. In our case,  we work a lot with ``matrix reductions'',  where the  initial tensor $T = \tensdecomp[d]{a}$ is reduced to a weighted quadratic form $\sum_{i=1}^m W(a_i) \lfp{a_i}^2$. In \secref{sec:rezstuff} we introduce several  such linear ``shrinking maps'' connected to polynomial evaluation that will allow us to write down the linear constraints of our sums of squares programmes.\\

Done with  the preliminaries,  we will  start \secref{sec:valgo} by giving an algorithm that can compute the exact solution of a moment decomposition problem in exponential time from the first $ d\geq 2 m $ moments. In \secref{sec:valgotd},  we will then show how this algorithm can be adapted to work for tensor decomposition by feeding the original algorithm with ``fake moments'' of lower degree.\\

The algorithms of Section  \secref{sec:valgo} aren't very efficient neither with respect to computation time nor with respect to the amount of data needed,  but they  illustrate  some of the main ideas and techniques that we are going to use in  the approximate setting of Section \secref{sec:effdecomp}: 
Here we investigate what  approximation results the same kind of algorithms can achieve when we restrict to polynomially sized SOS programmes.

\subsection{Results and Related Work} \label{sec:rel_work} Based on ideas and similar concepts present in \cite{BKS15}, we develop a new class of algorithms for tensor decomposition that can be seen as a generalisation of Jennrich's Algorithm to the case of overcomplete tensor decomposition. We show exact recovery guarantees in the case that the input tensor $T$ is of sufficiently high order ($d \geq 2m$ with $m$ being the number of components).  

These algorithms are in some sense ``matrix reduction algorithms'',  since they reduce high order tensors to quadratic forms   which correspond to symmetric matrices  and then recover the components by running eigenvalue decompositions on the reduced matrices. The broad concept of such matrix reductions  is very old –  note that Jennrich's classical result Alg. \ref{alg:Jennrich} from the introductory  section can actually be seen an example of such a matrix reduction algorithm -- at least in the broad sense:  In Jennrich's case we reduced the input tensor via the polynomial $W := \lfp{v}$ and then performed an eigenvalue decomposition.\\
The main conceptual novelty is in showing that it's possible to use sums of squares programming in order to find such  $W $ which attains high values only on one of the components $a_i $ (and acts thus as some sort of a weight function on the $a_i $),  which  yields a quadratic form being approximately of rank 1. The sums of squares condition is needed to ensure that the weights are nonnegative. It will turn out that this nonnegativity condition for the values $W(a_i)$ is essentially what enables us to find such $W$ via optimisation, since it allows to cap the maximum weight by a simple linear constraint, e.g. $\sum_{i=1}^mW(a_i) = 1$. This may sound quite different from Jennrich's Algorithm, since there we needed no such thing as an SOS constraint on the reduction polynomial. However, we will see  in \secref{sec:qualitatives} that Jennrich's classical Algorithm admits an equivalent SOS based formulation.

The authors of \cite{BKS15} were already using techniques which can,  in our terminology,  be described as  matrix reduction via SOS polynomials:  After computing some degree $k\geq d$ pseudo-distribution $\E $ satisfying certain constraints (in particular that $\E[\lfp{a_j}^k]$ is not too small for some $j\in[m] $),  they wanted to reduce $\E $ to a degree 2 pseudo-distribution (which corresponds to a psd matrix)  via conditioning with a sum of squares polynomial $W$.  The authors hoped to show that  with sufficiently high probability the polynomial\[
W_0 := \prod_{l = 1}^d \frac{1}{M} \sca{G^{(l)}, X}^2
\]
which is a product of  $d $ squares of independent Gaussian  linear forms $\sca{G^{(l)}, X}$ given by Gaussian vectors $G^{(l)} \sim  \mathcal{N}(0,I_n)$,  would satisfy 
\begin{align}
\E[W_0\lfp{a_j}^2 ] \geq (1-\landauO{\varepsilon}) \E[W_0] \label{eq:guessed_W}
\end{align}
where $\varepsilon$ is an (unknown) approximation constant depending on $d,m$  and the condition of the problem. This was Lemma 5.2 in \cite{BKS15}.\\

Unfortunately though,  there was an error in the proof of Lemma 5.2: In lines 2~-~3 on page 19,  they choose pairs $(\tau_M, M) \in\R^2$ of real numbers such that a standard (expected value 0 and standard deviation 1) Gaussian scalar variable $\xi$ satisfies \[
\E_{\xi\sim \mathcal{N}(0,1)}[\: \xi^2 \: \mathbb{1}_{\{\xi \geq \tau_M\}}\:] = M
\]
Of course this is feasible precisely for any $M\in (0,1)$.

On the same page,  in line 21,  they chose $M = (1/\varepsilon) \cdot \log(1/\varepsilon) \gg 1$ for some very small $\varepsilon >  0$.  Therefore $1 > M \gg 1$, rendering the choice of  $M $ infeasible. This essentially breaks the proof.  The statement also has a  surprising qualitative aspect,   since  we would expect  that the  choice of the scaling factor $M$  would not make any difference on the quality of the estimation \eqref{eq:guessed_W}.

Even more unfortunate is that Ge and Ma cited particularly this Lemma without repeating the proof in \cite{GM15}. Using Lemma 5.2 as  an integral part of their decomposition algorithm,  they designed the procedure which we described in the introductory chapter (\ref{sec:history}) and which achieves the currently best known recovery guarantees for average case tensors (in the sense that they achieve quasipolynomial time while allowing  $m $ to be quite large,  almost $m =  n^{1.5} $).

\newpage
\section{Preliminaries} \label{sec:prelims}

\begin{notation} \label{not:init_nots}
	Vectors of polynomial unknowns and random vectors will always be denoted by capital letters such as $X, Y$, whereas variables with numerical values such as $x,y\in \R^n$ will be denoted by lower case letters. In the case of polynomial unknowns, we implicitly understand that $X = \left( \alist[n]{X} \right)$, where the $\alist[n]{X}$ are algebraically independent (scalar) unknowns. A single scalar unknown (for univariate polynomials) will be denoted by $\Lambda$. We write $\N = \{1, 2, 3, \ldots\}$ and $[m] := \{1,\ldots, m\}$ for $m\in\N$. $e_i \in \R^n$ will always denote the $i$-th unit vector, such that $(e_i)_j = \case{1} {if i = j} {0} {otherwise} = \delta_{ij}$. $\spec{A} := \sup_{x\in \Sn} \sca{Ax,Ax}^{1/2}$  will denote the spectral norm of the matrix $A $. For $d\in \N_0$, let $\polyspace{d}$ denote the space of all polynomials in $X = (\alist[n]{X})$ up to degree $d$ and $\hompolyspace{d}$ the subspace of all homogeneous polynomials of degree precisely $d$. 
	For $\al\in\N_0^n$, we denote the multinomial coefficients \[
	\binom{d}{\al} := \frac{d!}{\al_1! \cdots \al_n!}
	\]
	occurring in the the important \emph{multinomial theorem}: For $a\in\R^n$ we have: \[
	\lfp{a}^d = \sum_{|\al| =  d} \binom{d}{\al} a^\al X^\al
	\] 
\end{notation}

\begin{notation} \label{not:coeff_convention}
	Each polynomial may be represented by its coefficients w.r.t. the canonical basis $(X^\al)_{|\al| \leq d}$, where $\al \in \N_0^n$ denotes a multi-index. We, however, use the convention to write a polynomial $P$ as \[
	P = \sum_{|\al| \leq d} \binom{|\al|}{\al} P_\al  X^\al.
	\]
	Hence, we represent a polynomial w.r.t. the basis consisting of all scaled monomials $\binom{|\al|}{\al} X^\al$. This convention is technical convenience ~--~  mainly for compatibility  with the multinomial theorem and polynomial evaluation. Details will become clear from \secref{sec:momandtdecomp}, \secref{sec:tensor_notation} and 
	\secref{sec:rezstuff}.
\end{notation}

\begin{remark}
	During this whole thesis, we will use the term ``tensor'' both to describe a multilinear map $(\alist[d]{v})\mapsto T\sca{\alist[d]{v}} \in \mlspace$ and a homogeneous polynomial $T = \sum_{|\al| =  d} \binom{d}{\al} T_\al X^\al \in \hompolyspace{d}$.  For most of the time, we will work with polynomials instead of multilinear maps. We will justify this and explain the connection in \secref{sec:tensor_notation}.
\end{remark}

\newpage
\subsection{Moment and Tensor Decompositions} \label{sec:momandtdecomp}

\begin{defi} (Tensor decomposition)\:
	Let $T\in\hompolyspace{d}$ for some $d\in\N$. $\alist[m]{a}$ is called an (unweighted symmetric) \emph{tensor decomposition} of $T$, if 
	\[
	T = \sum_{i = 1}^m \lfp{a_i}^{d} 
	\]
	The smallest $m\in \N \cup \{\infty\}$ for which such a tensor decomposition exists, is called the \emph{rank} of $T$ (where the rank is $\infty$ iff no tensor decomposition exists) and the $a_i$ are called \emph{components} of the decomposition. $\|a_i\|$ is called the \emph{magnitude} of $a_i$ in the decomposition.
\end{defi}

\begin{remark} 
	When talking about tensor decompositions, we will implicitly assume that the $a_i$ are all nonzero and no nonnegative multiples of each other. This makes sense for otherwise we could group $a_i$ and $\lambda a_i$ together to one summand $\sqrt[d]{(1+\lambda)}a_i$. If $d$ is odd, then we can extend this assumption to the $a_i$ being no real multiples of each other. Slightly imprecise, we will often call the $a_i$ \emph{components of $T$}. This practice is justified if there exists only one tensor decomposition (which is the usually the case for low-rank decompositions of high degree tensors). However, we will also use this practice when we merely hope that there could be uniqueness. Note that for even $d$ it's not possible to distinguish between components $a_i$ and $-a_i$.
\end{remark}

\begin{notation}
	Let $x\in \R^n$. Then by $\delta_x$ we denote the \emph{Dirac measure} of $x$, that is 
	\begin{align*}
	\delta_x : \P(\R^n) \to \R_{\geq 0},\:\: A \mapsto \case{1}{$x\in A$}{0}{otherwise}
	\end{align*}
	This is not to be confused with the \emph{Kronecker delta}: For $i,j\in\N$ we denote \[
	\delta_{ij} = \case{1}{if $i = j$}{0}{otherwise}
	\]
\end{notation}

\begin{defi}
	Let $\mu$ be a measure on some sub-sigma algebra $\mathcal{A} \subseteq \P(\R^n)$ such that $\E_\mu[P]$ exists for every homogeneous polynomial of degree $k$. The polynomial \[
	\E_\mu[X^{\tens k}] := \sum_{|\al| = k} \binom{k}{\al} \E_\mu[X^\al] X^\al 
	\]
	is called the $k$-th \emph{moment tensor} of $\mu$ (or, for brevity, the $k$-th \emph{moment}).\footnote{For technical reasons that will become apparent from the following proposition, our definition of moments is slightly different than in the introduction: The summands are now rescaled by some multinomial coefficients.}
\end{defi}

\begin{pro}\label{pro:asaresimwdec}
	Let $\mu = \sum_{i = 1}^m \lambda_i \delta_{a_i}$ be a measure finitely supported on the set $\{\alist[m]{a}\}$ with weights $\lambda_i > 0$. Then for any $k\in\N$ the $k$-th moment exists and satisfies \begin{align}
	\moment_k = \sum_{i = 1}^m \lambda_i \lfp{a_i}^k \label{eq:moment_quadrature}
	\end{align} 
\end{pro}

\begin{proof}
	Note that\begin{align}
	\E_\mu[X^\al] =  \sum_{i = 1}^m \lambda_i \delta_{a_i}[X^\al] = \sum_{i = 1}^m \lambda_i a_i^\al \label{eq:qf_onemonom}
	\end{align}
	Furthermore, according to the multinomial theorem we know for each $a \in \{\aalist \}$ that \[
	\lfp{a}^k = \sum_{|\al| = k} \binom{k}{\al} a^\al X^\al
	\]
	By expanding the right hand side of \eqref{eq:moment_quadrature}, we get
	\begin{align}\label{eq:qf_calc}
	\sum_{i = 1}^m \lambda_i \lfp{a_i}^k = 	\sum_{i = 1}^m \lambda_i \sum_{|\al| = k} \binom{k}{\al} a_i^\al X^\al = \sum_{|\al| = k} \binom{k}{\al} \sum_{i = 1}^m \lambda_i a_i^\al X^\al = \sum_{|\al| = k} \binom{k}{\al} \E_\mu[X^\al]  X^\al 
	\end{align}
	where the last step used \eqref{eq:qf_onemonom}. The right hand side of equation \eqref{eq:qf_calc} is precisely $\moment_k$.
\end{proof}
Hence $\sqrt[k]{\lambda_1} a_1, \ldots, \sqrt[k]{\lambda_m} a_m$ is a tensor decomposition of $\moment_k$ for each $k\in \{0, \ldots, d\}$. This motivates the following definition.

\begin{defi} (Moment decomposition)\:
	Let $T_0,\ldots, T_d$ be given Tensors of orders $k_0 < \ldots < k_d$. We call $\alist[m]{\lambda}, \alist[m]{a}$ a 
	\emph{moment decomposition} of $T_0,\ldots, T_d$, if for $j\in\{0, \ldots, d\}$
	\[
	T_j = \sum_{i = 1}^m \lambda_i \lfp{a_i}^{k_j} 
	\]
	For simplicity, we will usually write $\wlistdec$ instead of $\alist[m]{\lambda}, \alist[m]{a}$. Note that this is slight abuse of notation, since the $a_i$ and $\lambda_i$ are interpreted as distinct variables. 
\end{defi}

Moment decompositions may be seen as simultaneous (weighted) decompositions of several tensors at once.  As the name suggests, the most important special  case is when the first $ d $ moment  tensors of a measure are given: 

\begin{pro} 
	Let $\mu = \sum_{i = 1}^m \lambda_i \delta_{a_i}$ be a measure finitely supported on the set $\{\alist[m]{a}\}$ with weights $\lambda_i > 0$. Then $\wlistdec$ is a moment decomposition of $\moment_0,\ldots, \moment_d$ for any $d$. 
\end{pro}

\begin{proof}
	This was shown already in Prop. \ref{pro:asaresimwdec}. 
\end{proof}

\begin{defi}
	Let $\mu$ a measure on some sub-sigma algebra $\mathcal{A} \subseteq \P(\R^n)$ such that its first $d$ moments $\moment_0,\ldots, \moment_d$ exist. Using terminology from the theory of numerical integration, a moment decomposition of $\moment_0,\ldots, \moment_d$ is called a \emph{quadrature formula} for $\mu$ up to degree $d$. This is due to the fact that for any such decomposition $\wlistdec$ we have \[
	\E_\mu[P] = \E_{\sum_{i = 1}^m \lambda_i \delta_{a_i}} [P] = \sum_{i = 1}^m \lambda_i P(a_i)
	\] 
	for all polynomials of degree up to  $d$.  The largest $ d \in\N\cup\{\infty\}$ for which  the actual integral with respect to $\mu $ can be replaced by  this quadrature rule is called the \emph{exactness degree} of the quadrature formula $\wlistdec$ for $\mu$.  
\end{defi}

If $\mu$ is a finitely supported measure\footnote{In that case we can always assume $\mathcal{A} = \P(\R^n)$ and suppress the sigma-algebra in the notation.},  then there exists a quadrature formula of infinite exactness degree and our objective is to find it. This case will be the focus of this thesis.

\begin{remark} \label{rem:input_ambiguity} An obvious difference between tensor decomposition and moment decomposition is that the first problem does not ask for weights. This is due to the fact that for $\lambda_i \in \R_{\geq 0}$, $a_i \in \R^n$ we have that  \[
	\lambda_i \sca{a_i, X}^d = \sca{\sqrt[d]{\lambda_i} a_i, X}^d
	\]
	Thus, there is an ambiguity in the input data: We will never know if the component we search for is $a_i$ and the corresponding weight is $\lambda_i$ or if the component is $\sqrt[d]{\lambda_i}a_i$ with weight $1$. From that it becomes clear that we need more information if we want to distinguish between these cases. One way is the above method to require multiple moments of our solution vectors. Another way would be to impose constraints on the $a_i$ that restrict where the $a_i$ can be located (such as $\|a_i\| = 1$).\footnote{For even $d$, we will see in \secref{sec:qualitatives} that this particular constraint is equivalent to knowing all even-degree moments of degree lower than $d$} 
\end{remark}

Let us define the corresponding notion of magnitude for quadrature formulae. 

\begin{defi}
	Let $\wlistdec$ a quadrature formula for the measure $\mu$. Then we call $\lambda_i a_i$ the $i$-th \emph{weighted component} and $\lambda_i \|a_i\|$ its \emph{magnitude}. 
\end{defi}

In some sense, moment decomposition seems to be the easier task,  since the additional information we have breaks for instance the input ambiguity described above. On the other hand,  tensor decomposition seems to be the more general task,  since it relies on less information. 
Indeed, once we solved the tensor decomposition problem for the moment tensor $\moment_d$ of even degree  $d $,  then we can search for each $ k$ for some weights $\lambda_{k,i,\pm} $ minimising \[
\norm[F]{\moment_k - \sum_{i = 1}^m \lambda_{k,i,+} \lfp{b_i}^k +  \lambda_{k,i,-} \lfp{-b_i}^k}^2
\] by solving a linear least squares problem.\footnote{Here, $\norm[F]{P}$  denotes the Frobenius norm of  $P$, see Def. \ref{def:frobnorm}. Note there is a problem when starting with some odd degree $d$: If the components $a_i $ and $ -a_i $ occur  both with the same weight in $\mu$, they cancel in the odd-degree tensors,  whence they are lost.}  It is not too hard to see that you can get the true weights  and the true component lengths by comparison of the  $\lambda_{k,i,\pm}$.

Surprisingly, it works also the other way around:  Once we have an algorithm that can solve the moment decomposition problem,  but are given only one tensor of degree $d $,  then  we can generate some fake moments of lower degree and feed them to the moment decomposition algorithm. We'll discuss methods to generate artificial lower degree moments, given a sole tensor $T$, in Section \secref{sec:valgotd}.

We will go with this approach and thus first design an algorithm for moment decomposition,  that we will later  adapt to work for tensor decomposition as well. But first, here are some more basic facts about tensor decompositions.

\begin{remark} 
	\begin{enumerate}[(a)]
		\item For even $d$, if $T\in\hompolyspace{d}$ admits a tensor decomposition, then this implies that $T\geq 0$ on all of $\R^n$.  
		\item There is no efficient way (i.e. no procedure running in polynomial time) to decide the rank of a general $3$-tensor (assuming $P\neq NP$). 
		\item In fact, many problems that admit efficient algorithms in the matrix case turn out to be NP-hard when generalised to the tensor case. In \cite{HL13}, C. Hillar and L.-H. Lim pointed out that ``Most tensor problems are NP-hard'' (which was also the title of their paper) 
	\end{enumerate}
\end{remark}

The results of \cite{HL13} do not tell us whether or not there could be  efficient approximation algorithms for (special cases of) the tensor decomposition problem.   To be able to analyse this, let us first define how we measure the quality of an approximate solution as well as all of the  related notions.

\begin{defi} (Hausdorff distance)
	The finite sets $M,N$ (w.l.o.g. with $\#M \geq \#N$) are called \emph{$\varepsilon$-close}, if there exists a surjective map $\sigma \:: M \to N$ such that \begin{align}
	\norm{a - \sigma(a)} < \varepsilon
	\end{align}
	for all $a\in M$. The smallest such $\varepsilon \in \R_{\geq 0}$ 
	is called the \emph{Hausdorff distance} of $M$ and $N$. We denote it as\[
	\hausdist(M,N)
	\] 
\end{defi}
For the special case that the sets $M = \{\aalist\}$, $N = \{\alist[m]{b}\}$ have the same cardinality $m = \#M = \#N$, $\varepsilon$-closeness is equivalent to
\begin{align}
\min_{\sigma \in S_n} \norm{a_{i} - b_{\sigma(i)}} < \varepsilon
\end{align}
That is, such sets are $\varepsilon$-close if we can match them to pairs $(a_i, b_{\sigma(i)})$ of ``neighbours'' such that $\norm{a_{i} - b_{\sigma(i)}} < \varepsilon$.

\begin{defi} (Frobenius Norm)  \label{def:frobnorm}
	Let $P = \sum_{|\al| \leq d} \binom{|\al|}{\al} P_\al X^\al \in\polyspace{d}$. Then 
	\begin{align*}
	\norm[F]{P} = \sqrt{\sum_{|\al| \leq d} \binom{|\al|}{\al} P_\al^2}
	\end{align*}
	is called the \emph{Frobenius norm} of $P$.  It  corresponds to the 2-norm  of $P $ when $P$ is seen as a vector of  its coefficients with respect to the basis  $(\sqrt{\binom{|\al|}{\al}} X^\al)_{|\al| \leq d}$. 
\end{defi}

\begin{defi} (Forward and Backward Error)  
	Let $T\in\hompolyspace{d}$ a tensor  and $\alist[k]{b}\in\R^n$ an approximate solution to the tensor decomposition problem for $T$. 
	Then we call
	\begin{align*}
	\norm[F]{T - \sum_{i = 1}^m \lfp{b_i}^d} 
	\end{align*}
	the \emph{backward  error} of the  approximate solution $\alist[k]{b}$.  Furthermore,  with respect to an exact solution $\aalist$ (such that $T = \tensdecomp[d]{a}$),  we define
	\begin{align*}
	\hausdist(\{\aalist\}, \{\alist[k]{b}\})
	\end{align*}
	as the \emph{forward  error} of $\alist[k]{b}$ at $\aalist$. The nomenclature is historical convention.  Note that the latter notion will usually depend on the choice of $\aalist$.  However,  if the  true decomposition happens to be unique,  then the forward error is a property of $\alist[k]{b}$ and $T$, too. 
\end{defi}
Our approximation algorithms in \secref{sec:effdecomp} will focus on minimising the  forward  error of the  approximate solution with respect to all exact solutions which works  of course best  if the solution is unique.

\paragraph{Condition}

A widespread concept in  numerics  is condition.  This notion captures the fact  that the quality of the results that an approximation algorithm produces can depend gravely on some properties of the input data / the exact solution. In fact,  condition can be seen as a notion of ``the exact solution's quality''. Specifically in  tensor decomposition,  problematic instances could be e.g.  such where some components sit very closely together.  To deal with this, we will briefly introduce some parameters that will naturally appear in our following computations.  Notice  that all of these notions are defined with respect to  a fixed set of solution vectors.  In the case where the decomposition is not unique,  we do not  know of a canonical definition  of these parameters with respect to just $T$. 
\begin{notation} \label{not:cond_params} Let $(\aalist)$ a list of $m$ distinct vectors in $\R^n$. We define:
	\begin{align}
	\kmin &:= \min_{\substack{i,j = 1,\ldots,m\\ i\neq j}} \|a_i - a_j\|^2 \\
	\kmax &:= \max_{\substack{i,j = 1,\ldots,m\\ i\neq j}} \|a_i - a_j\|^2 \\
	\speccorrel &:= \max_{\substack{i,j = 1,\ldots,m\\ i\neq j}}  \spec{a_i\transp{a_i} - a_j\transp{a_j}}	\\
	\rho_{\text{lin}} &:= \max_{\substack{i,j = 1,\ldots,m\\ i\neq j}}  1 - \frac{\sca{a_i, a_j}}{\norm{a_i}\norm{a_j}}	
	\end{align}
	We'll hide their dependency  of  $(\aalist)$ in the notation,  but they will always be defined with respect to  $(\aalist)$. 
\end{notation}
In the case that all the components have unit length, \[
2\speccorrel^2 =  1 - \frac{\sca{a_i, a_j}^2}{\norm{a_i}^2\norm{a_j}^2}	
\] is one minus the maximum squared correlation between two components,  which can be seen by explicitly computing the characteristic polynomial of a rank 2 matrix (the  characteristic polynomial of a rank 2 matrix can be computed by looking at its trace and at the sum of its $2\times 2$ principal minors).
These parameters all have in common that they measure how well-separated the components are: This can either be done by looking at the length of their differences $\|a_i - a_j\|^2$ or by looking at their correlation $1 - \frac{\sca{a_i, a_j}}{\norm{a_i}\norm{a_j}}	$ together with their length differences $\norm{a_j}^2 - \norm{a_i}^2$.  Both approaches are connected via the identity \[
\|a_i - a_j\|^2 = \norm{a_i}^2 + \norm{a_j}^2 - 2\sca{a_i, a_j}
\]
For convenience, we  defined all of these parameters since they will naturally appear in our estimations in \secref{sec:effdecomp}.  The correlation metric can be defined slightly more generally:

\begin{defprop} (Correlation Metric)  \label{def:corr_metric}
	Let $ x, y \in\R^n$. We define
	\[
	d_{\Sn}(x,y) := \left(1 - \frac{\sca{x,y}}{\norm{x}\norm{y}} \right)^{1/2}
	\]	
	as the \emph{correlation metric} which is a metric on the unit sphere, and
	\[
	d_{\mathbb{P}^{n-1}}(x,y) := \left(1 - \frac{\sca{x,y}^2}{\norm{x}^2\norm{y}^2} \right)^{1/2}
	\]
	as the \emph{square correlation metric} which is a metric on the unit sphere modulo $\pm  1$,  that is,  on the real projective space $\mathbb{P}_\R^{n-1}$.
\end{defprop}

\begin{proof}
	For $x,y \in\Sn$, note that $\norm{x-y}^2 = \norm{x}^2 + \norm{y}^2 - 2\sca{x,y} = 2(1-\sca{x,y}) = 2\cdot d_{\Sn}(x,y)^2$. This  shows that $d_{\Sn}$ is a metric and hence the first claim. For the second claim,  recall that for unit vectors we have \[
	1 - \sca{x,y}^2  = \frac{1}{2}\spec{x\transp{x} - y\transp{y}}^2
	\]
	showing the triangular inequality and whenever $1 - \sca{x,y}^2 = 0$ then by Cauchy-Schwarz $x = \pm y$.
\end{proof}

\newpage
\subsection{Tensor Notation and Operations} \label{sec:tensor_notation}

In the introductory chapter,  we defined tensors in a way that the reader might not be familiar with. In fact,  the much more common approach to tensors is to see them as multilinear maps. In this section,  we will  state this second,  more standard definition of tensors and show that both viewpoints are equivalent for symmetric tensors.

\begin{defi}
	An \emph{(order-$d$) tensor} is a multilinear map \[
	T\::\: \underbrace{\R^n \times \ldots \times \R^n}_{\text{$d$ times}} \to \R
	\] 
	We denote $T\sca{\alist[d]{v}}$ instead of $T(\alist[d]{v})$ to stress the linearity in each of the arguments surrounded by angular brackets. \\
	$T$ is called (totally) symmetric, if $T\sca{\alist[d]{v}} = T\sca{v_{\sigma(1)}, \ldots, v_{\sigma(d)}}$ for all permutations $\sigma \in S_n$. 
\end{defi}

\begin{defi}
	An order-$d$ tensor $T$ is called \emph{simple}, if it is the product of $d$ linear forms, that is, if there exist linear maps $l_i\::\: \R^n \to \R$ such that
	\[
	T\sca{\alist[d]{v}} = l_1\sca{v_1}\cdots l_d\sca{v_d}
	\]
	for all $\alist[d]{v} \in \R^n$. 
	We  denote such a simple tensor by $l_1 \tens \cdots \tens l_d$ 
\end{defi}

Every Tensor can be written as a linear combination of simple tensors, since the $n^d$ tensors given by  \[
E_{(i_1, \ldots, i_d)}\sca{\alist[d]{v}} := e_{i_1}\sca{v_1}\cdots e_{i_d}\sca{v_d}
\] 
form a basis of the space of all linear maps $\R^n \times \ldots \times \R^n \to \R$, which we denote by $\mlspace$ (while the subspace of all symmetric tensors is denoted by $\symlspace$). We write $w = (i_1, \ldots, i_d) \in [n]^d$ and call $w$ a \emph{word} made of the ``letters'' $i_1, \ldots, i_d$. With respect to this basis,  we can write each tensor as a $d$-fold indexed array,  by looking at the coordinate representation\[
T \mapsto (T_w)_{w \in [n]^d}
\]
Also, we can define $X^w := X_{i_1} \cdot \ldots\cdot X_{i_d}$ for each word $w = (i_1, \ldots, i_d) \in [n]^d$.

\begin{pro} \label{pro:simple_sym}
	Let $T$ be a simple symmetric tensor. Then $T$ or $-T$ is the $d$-th power of some linear form, i.e.\[
	\exists a\in \lspace : \pm T =  a \tens \cdots \tens a =: a^{\tens d}
	\]
\end{pro}
\begin{proof} Let $T = l_1 \tens \cdots \tens l_d$. Fix some $v$ and  consider the linear form 
	\[
	w \mapsto T\sca{v,\ldots,v,w} = l_1\sca{v}\cdots l_{d-1}\sca{v}\cdot l_d\sca{w}
	\]
	which is a multiple of the form $l_d$. We can clearly choose $v$ in such a way that $l_i\sca{v} \neq 0$ for $i\in\{1, \ldots, d\}$, since these linear forms will be nonzero and thus vanish on lower-dimensional subspaces.\footnote{It is well known that $\R^n$ is not a union of finitely many lower-dimensional subspaces.} Then the above form is even a nonzero multiple of $l_d$. Now by symmetry we may exchange the arguments of $T$ and get that $\ph$  is exactly the same as
	\[
	w \mapsto T\sca{v,\ldots,v,w,v} = l_1\sca{v}\cdots l_{d-2}\sca{v}\cdot l_{d-1}\sca{w}\cdot l_d\sca{v}
	\]
	whence $\ph$ is a nonzero multiple of $l_{d-1} $, too. Therefore, $ l_{ d - 1} $  must be a nonzero multiple of $l_d$.  By repeating this argument,  we may see that actually all of  the $l_i$ must be nonzero multiples of each other. We can therefore choose scalar multiples $\alist[d-1]{\lambda}$ such that $ l_i = \lambda_i l_d$ for $i\neq d$. Set $a:= \frac{l_d}{\sqrt[d]{|\prod_{i = 1}^{d-1} \lambda_i|}} $. Then\[
	T = \big(\prod_{i = 1}^{d-1} \lambda_i\big) l_d \tens \cdots \tens l_d = \pm a \tens \cdots \tens a
	\]	
	just as we  claimed.
\end{proof}

\begin{defi}
	The smallest number  $m \in\N\cup\{\infty\}$ for which  a symmetric tensor $T$ admits a representation \[
	T = \sum_{i= 1}^m a_i^{\tens d}
	\]as a sum of $ m $ $d$-th (tensor) powers of linear forms is called the \emph{(symmetric) rank} of $T $.	
\end{defi}

Please be aware that the argument of Proposition \ref{pro:simple_sym} works only if the tensor really is simple. In a general symmetric tensor,  we can usually not exchange a non-symmetric simple summand by $\pm 1$ times a power of a linear form.  Requiring the factors of  $\pm 1 $ to be  $ 1 $ is a serious restriction, too, in particular since for even $d$ it implies  that the tensor is \emph{positive semidefinite},  that is, $T\sca{\alist[d]{v}} \geq 0$ for all $(\alist[d]{v})\in \R^n \times \ldots \times \R^n$. Therefore the concept of rank we use in this thesis might or might not be different from what  the reader could have read before in other works on tensor decomposition:  Some people  define the rank of a tensor $T$ as the smallest number of simple tensors such that $T $ can be written as a linear combination of them.

Due to these restrictions, a tensor decomposition does not need to exist in general. If no tensor decomposition exists,  we understand the rank to be infinite. Still, these particular  decompositions of symmetric tensors are an interesting object of study:  As we saw in the introductory chapter,  moment tensors of  finitely supported measures will always admit a tensor decomposition.

Now, let us just briefly point out the connection between symmetric multilinear forms and homogeneous polynomials that we both called tensors.

A multi-index $\al$ with $|\al| = d$ corresponds to an equivalence class of $d$ indices $(\alist[d]{i})$ modulo permutation. In particular, when writing down a polynomial in its coefficient representation, the coefficients of the monomials $X^\al$ correspond to the ``upper triagonal'' entries $T_{(\alist[d]{i})}$ of a traditionally indexed tensor (where by upper triagonal we mean the entries corresponding to words $(\alist[d]{i})$ such that $i_1 \leq \ldots \leq i_d$ ~--~ note that a symmetric tensor is determined by its upper triagonal). 
Thus there is a linear map 
\begin{align}
\varphi &\::\: \mlspace \to \hompolyspace{d},\: \nonumber \\
& T  \mapsto \sum_{|\al| =  d} X^\al \:\: \big(\sum_{\stackrel{w \in \N^d \text{ s.t.}}{X^w = X^\al}} T_w \big)
\end{align} 
which gets bijective when restricted to $\symlspace$.
Indeed, if $T$ is symmetric, the coefficients corresponding to $w$ and $w'$ are the same if the words $w$ and $w'$ are permutations of each other, which is the case if and only if $X^w = X^{w'}$. From that we see that 
\begin{align}
\sum_{\stackrel{w \in \N^d \text{ s.t.}}{X^w = X^\al}} T_w = \binom{d}{\al} T_w \label{eq:symtensorsid}
\end{align}
for any word $w'$ satisfying $X^{w'} = X^{\al}$. Here we denote by $\binom{d}{\al}$ the number of words $w$ satisfying $X^w = X^{\al}$. It can be shown that \[
\binom{d}{\al} = \frac{d!}{\al_1! \cdots \al_n!}
\]
(thus this is consistent with our definition in \ref{not:init_nots}). For that reason, we call $\binom{d}{\al}$ the \emph{multinomial coefficients}. Since $\binom{d}{\al} > 0$, the identity \eqref{eq:symtensorsid} shows injectivity as well as surjectivity, whence the two spaces are isomorphic. \\
Under the map $\varphi$,  simple tensors are mapped in the following way
\[
a_1\tens\ldots\tens a_d \: \mapsto \:  \lfp{a_1} \cdot \ldots \cdot \lfp{a_d}
\]
This implies that the notions of this  section and \secref{sec:momandtdecomp}  correspond.\\

Viewing tensors  as polynomials  has several advantages over the traditional representation as multilinear forms or arrays with $d$ indices, respectively.  One is that we can add  two tensors  of different orders,  since they're both members of the polynomial ring.  Another one is that the tensor product  of two simple  tensors  $\sca {a, X }^k $  and $\lfp{b } ^l $ becomes the regular product $\sca {a, X }^k\lfp{b } ^l $ of two polynomials. However,  the most important advantage is for sure that polynomials  are the native environment  for SOS  optimisation (see next section).\\

\newpage
\subsection{SOS Optimisation} \label{sec:sos_opt}
Sums of squares optimisation is a powerful tool that allows us to check whether or not a polynomial $W$  can be written as a sum of squares of polynomials.  Furthermore,  it allows us to optimise over all  such sums of squares (SOS) polynomials  $W$  satisfying certain linear constraints.

As such,  it is a generalisation of linear programming  (LP) and it serves,  in some sense, as an approximation to the condition that a polynomial is  nonnegative. The main novelty that SOS programming brings is that it allows to solve optimisation problems of the form:
\begin{align}
\max f(W) \quad \text{over all $W \in \sos{d} \cap \mathcal{L}$}
\end{align}
for some function $ f $ being linear in the coefficients of the polynomial  $W$ and some  linear subspace $\mathcal{L} \subseteq \polyspace{d}$.  Here,  $\sos{d}$ denotes  the convex cone of all SOS polynomials of degree $\leq d$.  Here, $d $ is called the \emph{degree} of the programme.

The motivation behind this is that  we want to optimise over  polynomials which satisfy a  global nonnegativity constraint $W\geq 0$. 
Alas, nonnegative polynomials are hard to optimise over,  so we need to replace the nonnegativity condition $\geq 0$ by something which is easier to verify:
\begin{defi} Let $P, Q \in \polyspace{d}$. We write \[
	P\preceq Q :\Longleftrightarrow Q-P\in\sos{d}	
	\]
\end{defi}
Of course each sums of squares polynomial will be globally nonnegative. The converse direction is in general not true,  however,  there are some special cases where we know by classical results dating back to D. Hilbert that equivalence holds between ``$\geq 0$'' and ``$\succeq 0$'' (e.g. 2.1.1, 2.3.2 and 2.3.5 in \cite{RAG}).

The first case  are univariate polynomials,  which is really easy to see by computing the polynomial factorisation over $\C$ and then grouping the non-real factors by pairs of complex conjugates. The remaining real roots  then have to be two-fold due to nonnegativity. The second case is the one of degree two polynomials in an arbitrary number of variables. Here the equivalence  can be proven by diagonalisation of symmetric matrices (after homogenising). The reason why the sums of squares condition is so much easier to  check is that it can be reduced to the task of finding some psd matrix:

\begin{notation}
Let $ x\in\R^n $. We denote
\begin{align}
	X^{\leq d} &= (X^\al)_{|\al| \leq d} \:\:\text{and}\\
	x^{\leq d} &= (x^\al)_{|\al| \leq d} 
\end{align}
Furthermore, let $\monoms{d} := \{X^\al \mid \:{|\al| \leq d} \}$. With respect to any fixed ordering of this set of cardinality $\nummonoms{d}$,  we may identify $\monoms{d} \equiv [\nummonoms{d}]$ and treat elements of  $\R^{\monoms{d} \times \monoms{d}} $ as matrices.

\end{notation}

\begin{defi}
Let $W \in\polyspace{d}$. 
We call $G = (G_{\al, \beta}) \in \R^{\monoms{d/2}\times \monoms{d/2}}$  a \emph{Gram matrix}  of $W $,  if 
\begin{align}
	W = \transp{\left(\monX{d/2}\right)} G \left.\monX{d/2}\right.
\end{align}	
\end{defi}

Since any monomial of  degree less or equal to $ d $ is a product of two monomials of degree less or equal to $\frac{d}{2} $,  each polynomial has a Gram matrix representation. Gram matrices are  in general highly non-unique since for instance any monomial of degree greater  or equal than 2 can be written in at least two different ways.  Indeed,  we can write  the monomial as $X_i\cdot X^\al$ for some $i\in[n]$ and some $\al \neq (0,\ldots,0) $  but also as   $ 1\cdot \left(X_iX^\al\right) = X^{(0,\ldots,0)}\cdot X^\beta$ for $\beta := \al + e_i$.  Hence the coefficient of this monomial in  $W $  may be distributed over several entries of $G $. 

It can be shown that a polynomial is a sum of squares if and only if it has a  positive semidefinite Gram matrix (see for instance §2.6 in \cite{RAG}).  This condition can be encoded within a  \emph{semidefinite programme} (SDP)  and then be  checked by an SDP  solver.  The condition that $G $ is a Gram matrix of $W $ can likewise be encoded by adding linear constraints between the coefficients of $G $ and the coefficients of $W $.

Sums of squares programmes can be,  in some sense,  efficiently solved numerically by SDP  solvers.  However,  there are some  caveats  related to feasibility and the coefficients which one would need to take care of  in order to get precise formulations  of the guarantees  SOS solvers  give. For simplicity,  we  will not do this but instead formulate what we call the  \emph{``magical black box rule'' of SOS programming}:

\paragraph*{Magical Black Box Rule of SOS Programming:} An $n$ variate SOS programme  of degree $d $ can be solved in time $n^{\landaueq{d}}$. \footnote{By ``solved'' we mean that if $f$ is the linear objective function and $W^\ast$ is some optimal solution then to any fixed numerical accuracy $ \gamma > 0$ we can find in time $n^{\landaueq{d}}$ a feasible polynomial $W$ which achieves $|f(W^\ast) - f(W)| < \gamma$, where the hidden constants in the Landau notation may depend on $\gamma$.  For simplicity, we will pretend that this would work for $\gamma = 0$ as  well.}  \\
\\
Now, this rule is wrong: There exist at least the following caveats:
\begin{enumerate}[(I)]
	\item  There is an issue  with the exact complexity of  general feasibility problems,  that we will not get into.
	\item We  ignored the encoding length of the coefficients we used to write down the linear constraints. Also, it is recommendable to have a bound on the diameter of the feasible space.
	\item The  number of linear constraints has an impact on the running time,  but following P. Parillo \cite{Par00}, it is usually ignored since  the cost of the linear constraints is dominated by the cost of the SOS constraints.
\end{enumerate}
  
However,  for most cases these will not be an issue and therefore people usually pretend that you can  solve an SOS programme of fixed degree $d$ exactly in polynomial time. 

In order to not go beyond  the scope of this thesis,  we will leave it with this very coarse introduction to SOS programming.  It's needless to say that we  merely scratched the surface of a huge topic. There is a long and detailed survey of M. Laurent \cite{Lau09} that we recommend to anyone interested in the details. Since we did not give a theorem with rigid runtime guarantees,  we will refrain from claiming any such runtime guarantees (in the sense of the classical complexity classes of theoretical computer science) on the algorithms we develop. Instead, we will use  a simpler notion of efficiency:

We say that a problem can be solved efficiently, if there exists  a \emph{polynomial time SOS algorithm} for it: 
\begin{defi}
	Suppose $\mathcal{B}$ is a black box that can solve all $n$-variate SOS optimisation programmes of degree $d$ in time $n^d $. 
	An \emph{SOS algorithm} is an algorithm which may  make arbitrarily many calls to the black box $\mathcal{B}$ in addition to the usual algorithmic operations (for each call,  the time it takes the black box to process the call is of course counted towards the running time of the SOS algorithm). 
\end{defi}
In the next section,  we will develop an SOS algorithm for the moment decomposition problem  which calls the black box $m $ times with $d = 2m $. This algorithm is not efficient with respect to the notion we just defined. This is due to the fact that the black box calls take time exponential in $m$. 

In Section \secref{sec:effdecomp},  we will then show that basically the same algorithm can be used with a lower value of $d$ (e.g. $d$ so small that we get a polynomial time SOS algorithm), as long as we accept that the algorithm will return only an approximate solution. In practice, the value of  $d $ will usually be determined by the amount of data that is given (at least if we are in the empirical case). An increase in the value of  $d $ will therefore not just mean an increase in computation time,  but also require more data, as outlined in the introductory chapter.\footnote{Note that  increasing the accuracy at will is thus usually not possible with this algorithm. However,  if we get sufficiently close to the actual solution,  then we could try to find the latter by  running a local searching procedure afterwards. } Therefore it makes sense to analyse what guarantees we get for a fixed value of $d$ instead of going the other way around and say what  value of $d $  is needed to get a guarantee  that the approximation error  is smaller than a fixed constant $\varepsilon$. Our approach is also the one which is technically more convenient.\\

As a last note on SOS programming,  recall that in the introductory chapter we defined \emph{pseudo-expectation operators} $\E$ of degree $d$ as linear functionals  on the space $\polyspace{d}$ satisfying 
\begin{align*}
&(1) \quad \E[1] = 1\\
&(2) \quad \E[P^2] \geq 0 \text{ for all polynomials }P \text{ with } P^2\in \polyspace{d}
\end{align*}
and claimed that we could efficiently optimise over them. This is due to the fact that the cone of  pseudo expectation operators of degree $d$ is the dual cone of $\sos{d}$. The literature on SOS programming seems to have  a bias towards the dual point of view,  but we will prefer to work directly with optimisation over SOS polynomials.

\newpage
\subsection{Multilinear Algebra on the Space of Polynomials}\label{sec:rezstuff}
In this chapter, we will introduce some basic notions and notations that will facilitate the operations we have to employ when dealing with (symmetric) tensors. 

We will have to redefine the usual methods  for tensor manipulation  in terms of polynomials. For those already used to tensor decomposition, but unused to see tensors as members of a polynomial ring: We will, in particular, define the \emph{multilinear multiplication} in such a way that it can be computed in terms of the monomial coefficients (i.e. the upper triagonal) only. It turns out that this was the reason why we had to scale the standard basis of monomials by a factor $\binom{d}{\al}$, which corresponds to the size of one of the equivalence classes described in \secref{sec:tensor_notation}.

\begin{notation}
	For $P\in \polyspace{d}$ and $k \leq d$, we denote by $P_{=k}$ the $k$-th \emph{homogeneous part} of $P$, that is \[
	P_{=k} = \sum_{|\al| = k} \binom{k}{\al} P_\al  X^\al.
	\]
	We have of course $P_{=k} \in \hompolyspace{k}$ and \[
	P = \sum_{k = 0}^d P_{=k}
	\]
\end{notation}

The reason why this representation is more suitable for our purposes is that it fits well with the following inner product on $\polyspace{d}$.

\begin{defi} (Reznick's Scalar Product, \cite{Rez92})
	Let $P, Q \in \polyspace{d}$. Consider their representations \[
	P = \sum_{|\al| \leq d} \binom{|\al|}{\al} P_\al  X^\al, \quad Q = \sum_{|\al| \leq d} \binom{|\al|}{\al} Q_\al  X^\al
	\]
	We define \[
	\rezprod{P}{Q} := \sum_{|\al| \leq d} \binom{|\al|}{\al} P_\al  Q_\al
	\] 
	which is bilinear in $P$ and $Q$. Fixing the left hand side, we get a linear form \[
	\phi_{Q}: \polyspace{d} \to \R,\: P \mapsto \rezprod{Q}{P}
	\]
\end{defi}

It is immediate that $\rezprod{\cdot}{\cdot}$ defines an inner product on $\polyspace{d}$. In fact, it is actually just a rescaled version of the standard scalar product on $\polyspace{d}$ w.r.t. the monomial basis. The following proposition shows that the Reznick inner product allows us to express polynomial evaluation. 

\begin{pro}\label{rezevalhom}
	Let $P\in \hompolyspace{k}$ homogeneous and $a\in\R^n$ (where $k\in\N_0$). Then \[
	\rezprod{P}{\sca{a, X}^k} = P(a)
	\]
	More general, if $P\in \polyspace{d}$ is not necessarily homogeneous, then for $k\leq d$\[
	\linfunc{P}(\lfp{a}^k) = \rezprod{P}{\sca{a, X}^k} = P_{=k}(a)
	\]
\end{pro}

\begin{proof} Let first $P\in \hompolyspace{k}$. By the multinomial theorem we have \[
	\sca{a, X}^k = \sum_{|\al| =  k}  \binom{k}{\al} a^\al  X^\al
	\]
	Using this representation, we can write down both sides of the equation as:
	\begin{align*}
	\rezprod{P}{\sca{a, X}^k} &= \sum_{|\al| =  k}  \binom{k}{\al} P_\al a^\al  \\
	P(a) = \sum_{|\al| \leq d} \binom{k}{\al} P_\al a^\al &= \sum_{|\al| =  k}  \binom{k}{\al} P_\al a^\al
	\end{align*}
	The last step in the second line uses that all non-degree-$k$ coefficients of $P$ vanish due to homogeneity. The claim for general $P$ is an immediate consequence. 
\end{proof}

We'll need a slightly more general version of this fact:

\begin{pro}\label{pro:rezeval}
	Let $P\in \polyspace{d}$ and $a\in\R^n$. Then \[
	\rezprod{P}{\sum_{k = 0}^d \sca{a, X}^k} = \sum_{k = 0}^d \rezprod{P_{=k}}{\sca{a, X}^k} = P(a)
	\]
\end{pro}

\begin{proof} Applying the previous proposition on all homogeneous parts $P_{=k}$, we get 
	\[
	P(a) = \sum_{k = 0}^d P_{=k}(a) = \sum_{k = 0}^d \rezprod{P_{= k}}{\sca{a, X}^k}
	\]
\end{proof}
Sometimes we want to apply the Reznick scalar product only partially to a given Polynomial $P$. This is achieved by the following ``tensoring''.\footnote{For homogeneous $Q$, this corresponds to the notion of the tensor product of linear maps. We will use (or, more accurately, abuse) this notation even for the case when $Q$ is not homogeneous. }

\begin{defi} Fix an integer $k\in \{0,\ldots, d\}$, which we will suppress in the notation along with $d$. Let $Q \in \polyspace{k}$.  Denote by $V = (\alist[n]{V})$ a new vector of algebraic unknowns. Define 
	\begin{align*}
	\rezprod{Q \tens \id}{P} :=  \rezprod{Q \sca{X, V}^{d-k}}{P}
	\end{align*}
	for each $P\in\polyspace{d}$. 
	Here,  the algebraic vector $V $ is plugged in formally as if it was an ordinary real vector $v\in\R^n$.
	This operation yields a homogeneous polynomial in $\R[V]_{= d-k}$. 
	Note that the map \[
	\phi_{Q\tens id}:  \polyspace{d}\to \R[V]_{=d-k}, P \mapsto \rezprod{Q \tens \id}{P}
	\] is linear in both $P$ and $Q$.
\end{defi}

The following proposition shows that $\phi_{Q\tens id}$ can be seen as a linear shrinking map that replaces powers of $\lfp{a}$ by ``evaluation terms'' $Q(a)$.

\begin{pro}\label{rezevalproof} For any $a\in\R^n$ and any polynomial $Q$ of degree at most $k\leq d$ we have\[
	\linred{Q}(\sum_{l = 0}^d \sca{a, X}^l) = \rezprod{Q \tens \id}{\sum_{l = 0}^d \sca{a, X}^l} = Q(a) \sca{a, V}^{d-k}
	\] 
\end{pro}

\begin{proof} From Prop. \ref{pro:rezeval} it follows that $\rezprod{Q \sca{X, v}^{d-k}}{\sum_{l = 0}^d \sca{a, X}^l} = Q(a) \sca{a, v}^{d-k}$  for any concrete value $v \in\R^n$ of $V$. Since $\R^n$ is Zariski-open, this identity must hold on the level of variables, too.
\end{proof}
Note that all terms $\sca{a, X}^l$ of degree less than $d-k$ are redundant: Their contribution will vanish, since $Q \sca{X, v}^{d-k}$ has no monomials of degree less than $d-k$.  After the reduction by the polynomial $Q $ is done,  there is no need to keep the additional variable $V $. Therefore, we use the convention to silently replace $V $ again by $X $.

\begin{kor} \label{kor:moment_reduction}
	Let $T_0,\ldots, T_d$ be given Tensors of orders $0,\ldots, d$ having a moment decomposition $\wlistdec$, that is,
	\[
	T_l = \sum_{i = 1}^m \lambda_i \lfp{a_i}^{l} , \quad l \in\{0,\ldots, d\}
	\]
	Let $Q \in \polyspace{d}$ be of degree $k$.
	Then we have\[
	\linred{Q}(\sum_{l = 0}^d T_l) = \rezprod{Q \tens \id}{\sum_{l = 0}^d T_l} = \sum_{i=1}^m \lambda_i Q(a_i) \sca{a_i, X}^{d-k}
	\] 
\end{kor}

\begin{proof}\begin{align*}
	&\rezprod{Q \tens \id}{\sum_{l = 0}^d T_l} =  \rezprod{Q \tens \id}{\sum_{l = 0}^d \sum_{i = 1}^m \lambda_i \lfp{a_i}^{l}} \\
	= &\sum_{l = 0}^d \sum_{i = 1}^m \lambda_i \rezprod{Q \tens \id}{ \lfp{a_i}^{l}}\\
	= &\sum_{i = 1}^m \lambda_i \rezprod{Q \tens \id}{\sum_{l = 0}^d  \lfp{a_i}^{l}} \\
	= &\sum_{i = 1}^m \lambda_i Q(a_i)\sca{a_i, V}^{d-k}
	\end{align*}
\end{proof}

\newpage
\section{The v-Algorithm for exact Moment and Tensor Decomposition}\label{sec:valgo}

\subsection{v-Algorithm for Moment Decomposition}\label{sec:valgoqf}

In this section, we  consider  a finitely supported measure  $\mu = \sum_{i = 1}^m \lambda_i \delta_{a_i}$ and we want to recover  the $\lambda_i$ and the distinct vectors $a_i$. Our main result in this section is that if we are given sufficiently many moments $\moment_0,\ldots, \moment_d$ of  $\mu $ (where $d\geq 2m$ will suffice),  we can compute them exactly by SOS programming as a the unique moment decomposition of $\moment_0,\ldots, \moment_d$. Precisely:

\begin{sat}\label{thm:valgoex} (Exact decomposition from moments up to $2m$) \\
	Let $m > 1$, $\alist[m]{a} \in \R^n$ be distinct vectors, $\alist[m]{\lambda}\in \R_{> 0}$. Suppose we are given the $k-$th moment \[
	\moment_k = \sum_{i = 1}^m \lambda_i \lfp{a_i}^{k}
	\]
	of the quadrature formula $\wlistdec$ for each $k \in \{1,\ldots, 2m\}$. Then $\wlistdec$ is the only moment decomposition of $\alist[2m]{M}$ and there exists an SOS algorithm to compute $\wlistdec$ from the input $\alist[2m]{M}$ in time $n^{\landaueq{m}}$.
\end{sat}

We've seen in section \secref{sec:rezstuff},  in particular from \ref{kor:moment_reduction},  that  with any polynomial $W$ of degree $ d - 2 $ we can  perform some sort of a ``reduction to a matrix'' \[
\sum_{k \in\{2,\ldots, d\}}\tensdecomp[k]{a} \mapsto \sum_{i=1}^m W(a_i) \lfp{a_i}^2
\]

The basic step in the following algorithm is to optimise for a nonnegative ``weight function'' $W$ which will allow us to compute the matrix\footnote{Technically, this is a quadratic form. As we did with tensors before, we aren't going to distinguish between symmetric matrices and quadratic forms for simplicity.} \[M := \sum_{i = 1}^m \lambda_i W(a_i) \lfp{a_i}^2\] We want to find a weight function which concentrates on one of the $a_i$: If the weight function (approximately) satisfies $W(a_i) = \case{\frac{1}{\lambda_i}}{if $i = j$}{0}{otherwise}$ for some $j\in [m]$, then $M$ will be (close to) the rank-$1$-matrix $\lfp{a_j}^2$ and we will be able to get $a_j$ out of $M$. In order to decide which component $a_j$ we are going to recover in a single step, we choose $v\in \Sn$ at random and then solve for the $a_j$ which maximises $\sca{a_i, v}$. Let's put this all together:\\

\begin{algorithm}[H]
	\caption{$v$-algorithm, first-read version, single component}
	\label{alg:exactDecomp_firstread}
	\begin{algorithmic}
		\STATE{$\mathbf{Choose}$ $v \in \Sn$ uniformly at random.}
		\STATE{$\mathbf{Solve}$ the optimisation problem \[
			\max {\sum_{i = 1}^m \lambda_i W(a_i) \sca{a_i, v}}
			\]}
		over all SOS polynomials $W \succeq 0$ satisfying $\deg(W) \leq d-2$ and $\sum_{i = 1}^m \lambda_i W(a_i) = 1$. (The starting point for the SOS solver may be chosen as the constant polynomial $W_0 = 1/\sum_{i = 1}^{m} \lambda_{i}$).\\
		\STATE{$\mathbf{Let }$ $W^\ast $ denote the output of this optimisation problem.}
		\STATE{$\mathbf{Compute}$ the matrix $M := \sum_{i = 1}^m \lambda_i W^\ast(a_i) a_i\transp{a_i}$ }
		\STATE{$\mathbf{Compute}$ an eigenvalue decomposition of $M$ with corresponding eigenvectors of \emph{unit length}.}
		\STATE{$\mathbf{Let }$ $u$ denote the computed eigenvector corresponding to the largest eigenvalue $\mu$.}
		\STATE{$\mathbf{Decide }$} the sign of $u$. To this end, compute the vector \[
			w := \sum_{i = 1}^m \lambda_i W^\ast(a_i) a_i \quad (\approx a_j)
		\]
		\STATE{$\mathbf{If }$  $\sgn(\sca{u, w}) < 0$ $\mathbf{replace}$ $u$ by $-u$.}
		\STATE{$\mathbf{Output }$ $\sqrt{\mu} u$. This vector will be precisely the $a_j$ maximizing $\sca{a_i, v}$}
		\STATE{$\mathbf{Output}$ $\frac{1}{W^\ast(\sqrt{\mu}u)}$. This scalar will be precisely  $\lambda_j$}
	\end{algorithmic}
\end{algorithm}
Now, this is in some sense an ``illicit'' formulation of an algorithm, since we wrote down the operations, functions and constraints involved in terms of the solution $\wlistdec$. To get an actually implementable algorithm (and thus get a proof of \thmref{thm:valgoex}), we need to show that all of these can be rewritten in terms of the input data $\alist[2m]{M}$. 

Let us first, though, review why this algorithm produces the correct result. To this end, we formulate correctness in the following theorem and prove it.

\begin{sat} \label{thm:correctness_ex}
	Let $W^\ast$ be an optimiser of 
	\begin{align}
		\max {\sum_{i = 1}^m \lambda_i W(a_i) \sca{a_i, v}}\quad (\ast)
	\end{align}
	over all SOS polynomials $W\in \polyspace{2m}$ satisfying 
	\begin{align}
		\sum_{i = 1}^m \lambda_i W(a_i) = 1
	\end{align}
	Then for $j = \argmax_{i\in [m]} \sca{a_i, v}$:
	\begin{align}
	\Wst(a_i) = \frac{1}{\lambda_i} \delta_{ij} \label{eq:deltapropertyforW}
	\end{align}
	and therefore $\sum_{i = 1}^m \lambda_i W^\ast(a_i) a_i\transp{a_i} = a_j\transp{a_j}$. From that we can compute $\pm a_j$ and $\lambda_j$.
\end{sat}  

	\begin{proof}
	Note that since we are optimizing over SOS polynomials of degree  $d-2\geq 2m-2$, the feasible space contains the \emph{interpolation polynomials} \[
	\mathcal{I}_j = \frac{1}{\lambda_i} \prod_{i\neq j} \frac{\|X-a_i\|^2}{\|a_i-a_j\|^2} 
	\] 
	which satisfy \begin{align}\label{eq:deltaproperty}
	\mathcal{I}_j(a_i) = \case{\frac{1}{\lambda_i}}{if $i = j$}{0}{otherwise}
	\end{align}
	Our claim is that any optimal solution $W^\ast$ must satisfy the property in (\ref{eq:deltaproperty}) as well for $j = \argmax_{i\in [m]} \sca{a_i, v}$.
	(Note that with probability one over the choice of $v$, the values $\sca{a_i, v}$ will all be distinct. Hence there will be only one maximiser $j$ which justifies the usage of the $\argmax$ function.) Indeed, we see that for any feasible $W$, since $W$ is a sum of squares, we have $\lambda_i W(a_i) \geq 0$ and thus
	\begin{align}
	\sum_{i = 1}^m \lambda_i W(a_i) \sca{a_i, v} 
	\leq  \big(\underbrace{\sum_{i = 1}^m \lambda_i W(a_i)}_{= 1}\big) \sca{a_j, v}
	= \sca{a_j, v} \label{eq:vupperbound}
	\end{align}
	where strict inequality holds if the term $\lambda_j W(a_j)$ is strictly smaller than $1$.\footnote{Here we use that the values $\sca{a_i, v}$ are all distinct, whence there is precisely one maximum value.} Since $I_j$ attains the upper bound (\ref{eq:vupperbound}), $W^\ast$ must, too, and by the above we get $\lambda_i W^\ast(a_i) = \delta_{ij}$ as claimed. 
	Hence the matrix $M$ is precisely \[
	M = a_j\transp{a_j}
	\]
	Any eigenvector of $M$ with nonzero eigenvalue is contained in $\im(M)$ and therefore a multiple of $a_j$. Since $u$ has unit length, $u = \pm \frac{a_i}{\|a_i\|}$. It remains to recover $\|a_j\|$ and the sign. Note that for the corresponding eigenvalue $\mu$ we know \[
	\mu u = Mu = a_j\transp{a_j}\left(\pm \frac{a_j}{\|a_j\|}\right) = \pm\|a_j\|a_j  = \|a_j\|^2 u 
	\]
	Hence $\mu = \|a_j\|^2$ and thus $\sqrt{\mu}u$ will be either $a_j$ or $-a_j$. It remains to show that the sign is correct. Note that \begin{align}
		w = \sum_{i = 1}^m \lambda_i W^\ast(a_i) a_i = a_j \label{eq:vector_reduction}
	\end{align}
	will also be equal to $a_j$. Hence $\sgn(\sca{u, w}) = 1$ if and only if $\sqrt{\mu}u = a_j$ and $\sgn(\sca{u, w}) = \sgn(-\|a_j\|^2) = -1$ otherwise. Now since we recovered $a_j = \sqrt{\mu}u$, we might plug it into the identity $\lambda_j W^\ast(a_j) = 1$ to recover $\lambda_j = \frac{1}{W^\ast(\sqrt{\mu}u)}$ as well.
\end{proof}

\begin{remark}
	At this point the observant reader is probably questioning why we didn't just output $w$, which would have been much simpler. There is no actual reason for this except for the fact that reducing to a matrix will be slightly more natural for the case of even degree tensor decomposition in the next section: When we can not distinguish between $\pm a_i$ anyway, computing the matrix $a_i\transp{a_i}$ will save us from having to make a sign-decision.
	This decision is instead transferred to the computation of the top eigenvectors.\footnote{By that we mean the eigenvectors $u\in \Sn$ of unit length that correspond to the largest eigenvalue. For some matrices, e.g. for (real) symmetric matrices with $n$ distinct eigenvalues, there are precisely two top eigenvectors $\pm u$.} 
\end{remark}

\begin{remark} On invalid input (that is, input which does not fulfill the requirements),  weird things might happen:  If for instance the degree of  $W $ is too low,  then $M$ will likely not be a rank 1 matrix (but it might still be very close to one,  which is essentially  what we show  in \secref{sec:effdecomp}). And if for instance $\mu = \sum_{i= 1}^m\lambda_{i}\delta_{a_i}$ is a \emph{signed} measure (where the weights $\lambda_{i}\in\R\setminus\{0\}$ may be negative) then we can show by the same interpolation argument that the decomposition is unique for $d\geq 2m$ (assuming the $m$ components are distinct). However, we can in general \emph{not} compute this decomposition by the $v$-Algorithm: If one of the components has a negative weight $\lambda_{i} < 0$,   then the optimisation problem will immediately become unbounded and thus there will be no optimiser.  This  result will then also certify that $\mu$ is not a measure of finite support size (for otherwise we would have a bound on the optimal value). \\
Still, this property of the algorithm reveals an issue problematic for noise stability: If the moments of $\mu$ are disturbed by some small noise tensors, this might result in an unbounded problem, since e.g. an arbitrarily small but negative $0 \approx \lambda_{i} < 0$ might have an arbitrarily high impact on the optimal value of the SOS programme. 
This can be dealt with by introducing a ``complexity bound'' on $W$ implying e.g. an upper bound on $\norm[\textrm{F}]{W}$. Bounding the complexity of $W$ comes with the usual issues known from Machine Learning problems: Having the bound too small results in a sub-optimal solution (``underfitting''), while leaving it too high makes the algorithm more vulnerable to noise (``overfitting''). 
\end{remark}

\noindent
Let us now summarise what is left to do. We need to show that\ldots
\begin{enumerate}[(a)]
	\item The algorithm can be written down in a way that relies only on the input data.
	\item We can repeat the algorithm $m$ times in order to recover all components. In particular, we need a constraint which ensures that we will not recover the same component twice.
\end{enumerate}

We'll start by writing down the algorithm in rigid manner, showing (a) and (b) simultaneously (see Algorithm \ref{alg:exactDecomp}). 
Now we're ready to prove \thmref{thm:valgoex}, in particular that Algorithm \ref{alg:exactDecomp} produces  the correct results.

\begin{algorithm}[h]
	\caption{ $v$-algorithm, implementable version, all components} 
	\label{alg:exactDecomp}
	\textbf{Input: }Tensors $T_0 \in \hompolyspace{0}, \ldots, T_d\in \hompolyspace{d}$ \\
	\textbf{Output: } Vectors $\alist[m]{c}$ and weights $\alist[m]{\rho}\in\R_{> 0}$ satisfying $T_k = \sum_{i = 1}^m \rho_i \lfp{c_i}^k$ \\
	\textbf{Require: } $d$ even and there should exist a moment decomposition of the input with $m$ weighted components for some $m$ satisfying $d\geq 2m$.\\
	\textbf{Procedure: }
	\begin{algorithmic}[1]
		\REPEAT
			\STATE{$\mathbf{Choose}$ $v \in \Sn$ uniformly at random.}
			\STATE{$\mathbf{Solve}$ the SOS optimisation problem \[
				\max \rezprod{W \lfp{v}}{\sum_{k = 1}^d T_k} \quad (\ast)
				\]}
			over all polynomials $W \succeq 0$ satisfying $\deg(W) \leq d-2$ and $\rezprod{W }{\sum_{k = 0}^{d-2} T_k} = 1$. 
			\STATE{$\mathbf{Let }$ $W^\ast $ denote the output of this optimisation problem.} 
			\STATE{$\mathbf{Compute}$ the matrix / quad. form $M := \rezprod{W^\ast \tens \id}{\sum_{k = 2}^{d} T_k}$ }
			\STATE{$\mathbf{Compute}$ an eigenvalue decomposition of $M$ with corresponding eigenvectors of unit length.}
			\STATE{$\mathbf{Let }$ $u$ denote the computed eigenvector corresponding to the largest eigenvalue $\mu$.}
			\STATE{$\mathbf{Output }$ $c_i := \sqrt{\mu} u$} as the component and $\rho_i := \frac{1}{W^\ast(c_i)}$ as the weight.
			\STATE{$\mathbf{Add }$ the linear constraint $W(c_i) = 0$} to the SOS optimisation problem $(\ast)$. Increment $i$
		\UNTIL{the problem $(\ast)$ becomes infeasible}
	\end{algorithmic}
\end{algorithm}
\noindent
\begin{proof} (of \thmref{thm:valgoex}) Let us show first that the functions and constraints are indeed the same as before. Suppose the $T_k$ admit a simultaneous decomposition \[
T_k = \sum_{i = 1}^m \lambda_i \lfp{a_i}^k	
\] 
Let us fix this decomposition. By the results of \secref{sec:rezstuff}, in particular by Prop. \ref{pro:rezeval} and Cor. \ref{kor:moment_reduction}, we have that 
\begin{align}
\rezprod{W}{\sum_{k = 0}^{d-2} T_k} &= \sum_{i = 1}^m W(a_i) \nonumber \\
\rezprod{W \tens \id}{\sum_{k = 2}^{d} T_k} &= \sum_{i = 1}^m W(a_i) \sca{a_i, X}^2 \quad (= M) \label{eq:correctness_of_functions} \\
\text{and} \qquad \rezprod{W \lfp{v}}{\sum_{k = 2}^d T_k} &= \sum_{i = 1}^m W(a_i) \sca{a_i, v}\nonumber
\end{align}

Note that all of these (the objective function, the left side of the constraint and $M$) are linear functions in the variable $W$ and $W$ is the only variable, since $v$ is fixed and the $T_k$ are all given constant tensors. In particular, $(\ast) $ is indeed an SOS programme.

The equations \eqref{eq:correctness_of_functions} show that the first loop of the algorithm will indeed yield a component of the $T_k$. In fact, it'll yield the component $a_i$ which maximises $\sca{a_i, v}$. Note that the constraint $W(c_i) = \sum_{|\al| \leq d} W_\al c_i^\al = 0$ is linear in $W$, too (while $c_i$ is known at that time), and can thus be expressed by computing the powers $c_i^\al$ of $c_i$. \\
	
To see that we can recover all components, we claim that:
\begin{enumerate}[(i)]
	\item in each step we get a component that we didn't have before.
	\item the algorithm terminates after precisely $m$ steps.
\end{enumerate}
(i) is actually quite clear:  In the first step, since there are no restrictions of the kind $W(c_i) = 0$ yet, $W^\ast$ will concentrate on the component maximizing $\sca{a_i, v}$ over all $i\in[m]$ as we've seen before. 
Suppose that in the first $k$ steps we've got each time a distinct vector with corresponding weight, say $\rho_1 c_1,\ldots, \rho_k c_k$. By induction suppose that we have a matching:
\[
c_1 = a_{\sigma(1)},\:\ldots\:, c_k = a_{\sigma(k)}
\]
for some permutation $\sigma$ of $[m]$ (and likewise for the weights). Let $c_{k+1}$ denote the output of the k+1-st iteration which was subject to the constraints $W(c_1) = \ldots = W(c_k) = 0$. Now some of the terms in the estimation \eqref{eq:vupperbound} will vanish, allowing us to refine our  bound to: 
\begin{align}
\sum_{i = 1}^m \lambda_i W(a_i) \sca{a_i, v} 
= \sum_{\substack{i = 1,\\ a_i \notin \{\alist[k]{c}\}}}^m \lambda_i W(a_i) \sca{a_i, v} 
\leq \big(\underbrace{\sum_{i = 1}^m \lambda_i W(a_i)}_{= 1}\big) \max_{\substack{i \text{ s.t.}\\ a_i\notin \{\alist[k]{c}\} } } \sca{a_i, v}
\label{eq:vupperbound2}
\end{align}
which is again attained by the interpolation polynomial corresponding to the maximising $a_i\notin \{\alist[k]{c}\}$. By the same arguments  that were used in the proof of \thmref{thm:correctness_ex} we  can then again  reduce  our tensor  to a rank $1$ matrix  and see that we get a new component out of it (which is precisely the maximising $a_i\notin \{\alist[k]{c}\}$ and equal to $c_{k+1}$) together with the corresponding weight.\\ 
(ii): After  $m  $ iterations,   by our previous considerations,  we will have  added the constraints $W(a_1) = \ldots = W(a_m) = 0$.  But then the constraint $\sum_{i = 1}^m \lambda_i W(a_i) = 1$ of the  SOS problem becomes infeasible. 
\end{proof}

\begin{remark}
	Note that when choosing the decomposition $\alist[m]{a}$, it might seem like we had a freedom of choice. However, we get that the resulting output $\{c_1,\ldots, c_m\}$ is equal to $\{\alist[m]{a}\}$ independently of the choice of the $a_i$. This shows that there is in fact only \emph{one possible choice} for the $a_i$ (up to renumbering). Thus the algorithm provides us with a proof of uniqueness in this setting. 
\end{remark}

\begin{remark}
	Note that the algorithm would also work without resampling the random vector $v $ in every iteration (as long as $v$ is not orthogonal to any of the  components which is the case with probability 1).  However,  then the optimal value might get very low on the last iterations which is bad for several reasons that will become  slightly more obvious from section \secref{sec:effdecomp}.  
	Also, if one of the  components is the zero vector,  it will be recovered if all other remaining components attain a negative value of $\sca{a_i,v}$. In this case, $M $ will be the zero matrix.
\end{remark}

\newpage
\subsection{v-Algorithm for Tensor Decomposition}\label{sec:valgotd}

The algorithm developed in \secref{sec:valgoqf} can be adapted to work for tensor decomposition as well, that is, when only \emph{one} tensor $T$ of degree $d\in 2\N_{>2}$ (and rank $m\leq d/2$) is given. This is possible even though the $v$-algorithm makes explicit use of all the lower-degree-moments to write down the constraints and the objective function. 

To achieve this, we would need to \emph{generate} something alike lower degree moments of the $a_i$. But generating the actual  lower degree moments from the input $T$ has to be impossible, particularly since the $a_i$ will never be unique in the setting of even degree tensor decomposition. 
But we have a chance in generating some ``fake moments'' by abusing the input  ambiguity which we described in \ref{rem:input_ambiguity}: 
Suppose again for this section that \[
T = \tensdecomp[d]{a}
\]
is a tensor decomposition of $T$. Pick some random $w\in \Sn$  and compute the tensors 
\begin{align}
	T_k := \sum_{i = 1}^m \sca{a_i, w}^{d-k} \lfp{a_i}^k,\qquad k \in \{0,\ldots, d\}
\end{align}
by applying the linear maps $\linred{\lfp{w}^{d-k}}$ to $T$. Note that $T = T_d$.\\
Now, the $T_k$  aren't necessarily proper moments of the $ a_i $.  But in some sense,  they are \emph{fake moments}: 
Indeed,  write  $\lambda_i := \sca{a_i, w}^d $ and $b_i := \frac{a_i }  {\sca{a_i, w}}$.  Then we have
\begin{align}
	T_k = \sum_{i = 1}^m \lambda_i \sca{a_i, w}^{-k} \lfp{a_i}^k  = \sum_{i = 1}^m \lambda_i \lfp{b_i}^k 
\end{align}
From that we see that the tensors $T_k $ admit a simultaneous weighted  decomposition  with weights  $\lambda_i $ and components $b_i$. Once we have both the weights $\lambda_i$ and the components $b_i $,  we might get the true components $a_i $ out of them.  Let's see what will happen if we plug those fake moments into the $v$–algorithm.  

Algorithm \ref{alg:exactTensorDecomp}  shows a version of the $v$–algorithm  that has been ``hacked''  in such a way that whenever it gets fed  with the fake moments described above it outputs the correct components  up to a factor of  $\pm 1$.  We do not have to prove much here  since most of the work has already been done in the previous section.  In particular,  we know that $(\lambda_i, b_i)$ in line 8 will be one component-weight pair of the unique solution to the moment decomposition problem 
\begin{align}
T_k = \sum_{i = 1}^m \lambda_i \lfp{b_i}^k 
\end{align}
that we described above. Since we know uniqueness, there is no other possibility than \[
\lambda_i = \sca{a_j,w}^d 
\]
and $b_i = \frac{a_j}{\sca{a_j,v}}$ for some $a_j$. Furthermore, this $a_j$ will be exactly the one maximising the ratio $\frac{\sca{a_j,v}}{\sca{a_j,w}}$, since the objective function satisfies \[
\sum_{i=1}^m \lambda_i W(b_i) \sca{b_i,v} = \sum_{i = 1}^m \lambda_i W(b_i) \frac{\sca{a_i,v}}{\sca{a_i,w}}  \leq \left(\sum_{i = 1}^m \lambda_i W(b_i)\right)\frac{\sca{a_j,v}}{\sca{a_j,w}} = \frac{\sca{a_j,v}}{\sca{a_j,w}}
\] 
for $j = \argmax_{i\in [m]} \frac{\sca{a_j,v}}{\sca{a_j,w}}$. Note that we used again that the values $\frac{\sca{a_j,v}}{\sca{a_j,w}}$ are distinct and well-defined due to randomness of $w$. Actually, $v$ does not necessarily need to be random any more, since $w$ already is: The only thing we have to ensure is that $v\neq w$. (However, low correlation between $v$ and $w$ is of course preferable, e.g. $v\perp w$).

\begin{algorithm}[H]
	\caption{ $v$-algorithm for tensor decomposition} 
	\label{alg:exactTensorDecomp}
	\textbf{Input: } A single tensor $T \in \hompolyspace{d}$ \\
	\textbf{Output: } Vectors $\alist[m]{c}$ satisfying $T = \sum_{i = 1}^m \lfp{c_i}^d$\\ 
	\textbf{Require: } $d$ even and there should exist a decomposition of $T$  with $m$ components for some $m$ satisfying $d\geq 2m$.\\
	\textbf{Procedure: }
	\begin{algorithmic}[1]
		\REPEAT
		\STATE{$\mathbf{Choose}$ $w \in \Sn$ uniformly at random and then $v\in \Sn$ from the uniform distribution conditioned on $v\perp w$.}
		\STATE{$\mathbf{Generate }$ the ``fake moments'' \[
		T_k = \linred{\lfp{w}^{d-k}}(T)	
		\] for $k \in\{0,\ldots, d\}$.}
		\STATE{$\mathbf{Solve}$ the SOS optimisation problem}\[
			\max \rezprod{W \lfp{v}}{\sum_{k = 1}^d T_k} \quad (\ast)
			\]
		over all polynomials $W \succeq 0$ satisfying $\deg(W) \leq d-2$ and $\rezprod{W }{\sum_{k = 0}^{d-2} T_k} = 1$ and $W(\frac{c_i}{\sca{c_i,w}}) = 0$ for all $c_i$ that have already been recovered.
		\STATE{$\mathbf{Let }$ $W^\ast $ denote the output of this optimisation problem.} 
		\STATE{$\mathbf{Compute}$ the matrix / quad. form $M := \rezprod{W^\ast \tens \id}{\sum_{k = 2}^{d} T_k}$ }
		\STATE{$\mathbf{Compute}$ an eigenvalue decomposition of $M$ with corresponding eigenvectors of unit length.}
		\STATE{$\mathbf{Let }$ $u$ denote the computed eigenvector corresponding to the largest eigenvalue~$\mu$.}
		\STATE{$\mathbf{Let }$ $b_i := \sqrt{\mu} u$} and $\lambda_i := \frac{1}{W^\ast(b_i)}$. 
		\STATE{$\mathbf{Output }$ $c_i := \sqrt[d]{\lambda_i} b_i$} as one of the components (up to $\pm 1$).
		\UNTIL{the problem $(\ast)$ becomes infeasible}
	\end{algorithmic}
\end{algorithm}

\begin{remark} (Success probability in floating point systems)
	Note that the  $v$–algorithm has success probability $1$  even though it contains a randomised step which is the  choice of  $v $.  Therefore we can say that it is just a  ``mildly randomised'' algorithm contrary to the procedures in \cite{BKS15} and \cite{GM15},	where even a single  component recovery step  can have a very low success probability  $\ll 1$. However, the choice of $v$ will start to matter in the approximate setting of \secref{sec:effdecomp}. Though even here there is an issue  when performing this algorithm in a finite-accuracy floating point system:  We have used the argument that the values of the scalar products  $\sca{a_i,v}$ are all pairwise distinct.  Now  with a certain positive probability it could happen that  these values are indistinguishable up to the numerical accuracy. In this case,  our  recovery step could fail  and we would have to repeat it.  Now what is the probability for this event? Let us consider the set \[
	\mathcal{E}_{ij} = \{v\in \Sn \mid |\sca{a_i-a_j, v}| > \epsilon\} 
	\]
	which corresponds to the good event that two  values of the scalar products  are distinguishable up to $\epsilon$.
	Geometrically spoken,   this set can be imagined as  the intersection of the  two half spaces  $\{v\in \R^n \mid \pm \sca{a_i-a_j, v} > \epsilon\}$ 
	with the  $n $– dimensional unit sphere. It can be shown that the probability to hit  $\mathcal{E}_{ij}$ with the choice of  $v$ is $\frac{2}{\pi}\arccos(\frac{\epsilon}{\norm{a_i-a_j}})$ 
	and that $\arccos(x) \to~\frac{\pi}{2} \quad (x\to~0)$ with convergence order $\landaueq{x}$. The probability that not all the values are pairwise distinct is then by the union bound $\leq m \big(1 - \frac{2}{\pi} \arccos(\frac{\epsilon}{\min_{i\neq j} \norm{a_i-a_j}})\big)$. Thus recovery works if $\epsilon$ is several orders of magnitude smaller than $\norm{a_i-a_j}$ and $m$ is reasonable. 
\end{remark}

\subsection{Qualitative Aspects}\label{sec:qualitatives}
This  short section is dedicated to  the discussion of special cases as well as some ideas,  optimisations and alternative approaches that we will not elaborate in full detail. Let us start with  some special situations where exact recovery is possible from less data  than $ d =  2 m $. In particular,  we want to point out the connection between the  $v $–algorithm and Jennrich's algorithm.

\paragraph{Orthogonal Components and Connection to Jennrich's Algorithm}\label{sec:orthsos}

The above algorithm implies that there is only one quadrature formula fulfilling the requirements and thus gives a proof of uniqueness in this setting. However, note that knowing $d \geq 2m$ moments is quite a heavy requirement. 

In special instances, the $v$-algorithm might achieve exactness for way smaller values of $d$. As a friendly example, consider the tensor \[
T = \sum_{i = 1}^m \sca{e_i, X}^d
\] 
where $e_i\in\R^n$ denotes the $i-$th standard basis vector. On the set $\{e_1, \ldots, e_n\}$, the polynomial $X_i$ evaluates the same way as the $0/1$ indicator function $\mathbb{1}_{e_i}$. Well, $X_i$ is no SOS polynomial, but $X_i^2$ is and from that we get that the $v-$algorithm will recover the exact solution even for $d = 4$. A more detailed analysis shows that the case $d = 3$ can also be made to work: The objective function
\begin{algorithmic}
	\STATE{ \[
		\max {\sum_{i = 1}^m W(a_i) \sca{a_i, v}}
		\]}
\end{algorithmic}
can of course be written down  with a degree-2+1 polynomial $W\lfp{v}$ on the left hand side of Reznick's inner product. However, since $W$ consumes two degrees, we will not have enough data to compute the matrix $M$. But we can still do the ``vector reduction'' outlined in Equation \eqref{eq:vector_reduction}. In the case $a_i = e_i$, we may now use $W = X_i^2$ as a testimony to show that this adapted  $v-$algorithm can recover the components of $T$ exactly even for $d = 3$. More generally, if the $a_i$ are linearly independent, then there exists an invertible matrix $S$ such that $Sa_i = e_i$. Hence $\tilde{W} := W(SX)$ testimonies that the $v$-algorithm succeeds for the linearly independent case.

Its theoretical properties in this setting correspond precisely to those of Jennrich's algorithm, where we needed the moments of degree $2$ and $3$ (the second moment was used to transform the components to orthogonal ones). These moments are precisely the ones that the $v-$algorithm needs in this setting: Note that our testimony $W$ is homogeneous of degree $2$, whence we do not need the moments of degree $0$ and $1$. Also, it seems that the lower-order moment generation procedure isn't able to generate the second order moment: 
This is due to the fact that the weights $\lambda_i = \sca{a_i,v}^d$ have to be positive for which we get a guarantee only if $d $ is even.  This analysis demonstrates that the $v$-algorithm can indeed be seen  as a generalisation of Jennrich's algorithm to the overcomplete setting.

\paragraph{Using Symmetry} Another friendly example are the vertices of a regular simplex  with edge length  $r$ in  $n$ dimensions.  Here $\norm{a_j-a_i}$ attains only two values:   $0$ and $r$.  Therefore we can choose a univariate SOS  interpolation polynomial\footnote{E.g. $p = (1-\frac{1}{r}\Lambda)^2$} $p\in\R[\Lambda] $ of degree  $2$ with $p(0) = 1, p(r) = 0$ and we see that $p(\norm{X~-~a_j}^2)$ testimonies that exact recovery is possible already from the moments of degree less or equal $4$, independent of $n$. It's one of the beautiful  properties of this algorithm that it can exploit such symmetry with seemingly no additional effort.\footnote{Note that tensor decomposition over the simplex is relatively easy:  Since there are only  $n+1$ points $\alist[n+1]{a}$ in question,  we can make the ansatz $T_k = \sum_{i=1}^m \lambda_i \lfp{a_i}^k$ and just search for the weights via linear least squares. This shall thus just be seen as an illustrative example. Note further that the latter algorithm needs a priori knowledge on where the components are situated whereas the $v$-algorithm does not.}

\paragraph{Components with Differing Lengths and Discriminatory Polynomials}

An obvious downside that comes with  the randomness is that we will not get the components in any particular order.  Depending on the choice of $v $,  they could get recovered in any  order,  with the  odds slightly favouring  components  of high norm  $\|a_i\| $ to come first.\footnote{Note that the actual chance  for a component  to become $\argmax_i \sca{a_i, v}$  depends not only on the norm, but also on the geometry  of the $a_i$:  Isolated  components are more likely to be picked than components  which  have a ``neighbour'' being highly correlated with them.} 

However,  there are deterministic variants of the algorithm working for certain special cases, of which we want to present one. Suppose the  $a_i $ all have different lengths. Now, what we could  try  to do in the case of moment decomposition is to replace the objective function $(\ast)$ by
\begin{align}
\rezprod{W (X_1^2 + \ldots + X_n^2)}{\sum_{k = 2}^d T_k} = \sum_{i = 1}^m W(a_i) \|a_i\|^2 
\end{align}

Then we can again characterise the optimizing  weight functions like we did before: Indeed,  they will concentrate on the highest-norm component. Be careful: The requirement of distinct norms $\norm{a_i}$ is inherently necessary. Otherwise in general we will not get a rank $1$ matrix! (Note further that this will recover the \emph{highest norm} components first and not necessarily the \emph{highest magnitude} components). 

In the \emph{general} setting of tensor decomposition, it's possible (and it may or may not be useful) to replace  $\sca{b_i,v}$ by $\norm{b_i}^2 = \frac{\norm{a_i}^2}{\sca{a_i,w}^2}$,  since  sufficient randomness is already given by the denominator of $\frac{\norm{a_i}^2}{\sca{a_i,w}^2}$.  More generally, we can actually plug in any non-constant polynomial $ f $ of degree at most $2$ in place of $\sca{v, X}^2$,  since the differences $f(a_j) - f(a_i)$ will be non-zero with probability one (for moment decompositions this is also true as long as the polynomial $f$ is sampled from a continuous distribution. If we have some prior knowledge about the domain of the $a_i$, it might be possible to choose $f$ deterministically in a way that guarantees that the differences are always non-zero).   

The analysis  suggests that it makes sense to introduce the notion of \emph{discriminatory polynomials}, which are polynomials achieving high values of $\min_{i\neq j} f(a_j)-f(a_i)$ for the component $a_j$ maximising $f$.  Once we move to the setting of approximate recovery,  the choice of  $ f $ becomes crucial.  We will see that we need to make a choice which provides us  with a sufficiently big gap $\min_{i\neq j} f(a_j)-f(a_i)$.  While it can be shown that random choices achieve this gap with substantial probability,  finding good choices of  $ f $ could be an interesting subject of further research. 

\paragraph{Canonical Procedures for Lower-Degree Moment Generation}\label{sec:can_proc_momgen}
In the setting of tensor decomposition, where we do not have  the lower degree moments,   a common constraint is to require that all components have unit length, i.e.  $\norm{a_i} = 1$.
The  lengths are then either put into some weights $\lambda_i$,  which  yields an equivalent formulation,  or the weights are still assumed  to be one,  which  results in a proper special case and a slightly simplified problem. (For some problems  related to tensor decomposition this is justified,   see  for example \cite{BKS15}).  This implies that for each \emph{even} $ k $ we can generate
\begin{align}
T_k = \sum_{i = 1}^m \lfp{a_i}^k = \sum_{i = 1}^m \|a_i\|^{d-k} \lfp{a_i}^k = \rezprod{(X_1^2 + \ldots + X_n^2)^{(d-k)/2} \tens \id}{T}
\end{align}
which is  then the true moment of the  $a_i$. In the general case,  we can still compute the  even degree fake moments
\begin{align}
T_k = \sum_{i = 1}^m \|a_i\|^{d-k} \lfp{a_i}^k = \rezprod{(X_1^2 + \ldots + X_n^2)^{(d-k)/2} \tens \id}{T}\label{eq:normscaledfakemoments}
\end{align}
and feed them as input to the algorithm.  This will force a decision on the input ambiguity \ref{rem:input_ambiguity}: It will automatically ensure that the algorithm interprets the components as points of  unit length $\frac{a_i}{\norm{a_i}}$ and  the weights as $\lambda_i = \norm{a_i}^d$. 

Unfortunately, it's not clear what we can do in order to obtain the corresponding $ \|a_i\|^k$-scaled moments when $k$ is odd. 
It's very plausible though that there is no canonical choice here to make.  Note that having these fake moments for odd degree would automatically enable us to distinguish between $a_i$ and $-a_i$.  But these are indistinguishable in an even degree tensor!   Here, the moment generation procedure essentially makes a random,    noncanonical decision. One could ask though whether it's possible to generate \[
T_k = \sum_{i = 1}^m s_i^k \|a_i\|^{d-k} \lfp{a_i}^k
\]
for \emph{some} $s_i \in\{\pm 1\}$, which is a question we have to leave open. This  corresponds to the question  whether we  could potentially plug in $\sqrt{X_1^2 + \ldots + X_n^2}^k $ into \eqref{eq:normscaledfakemoments} in a way that makes sense. \\

Of course we could still try to run the algorithm  with only the even degree inputs.  As we've seen in section \ref{sec:rezstuff},  this means that we'd  essentially optimise over weight polynomials  with vanishing odd degree homogeneous parts.  If  in that way we get a solution vector $c $,  it might be an inferior approximation compared to the  vector  we'd have gotten  out of the $v $–algorithm.

However, note that there exist reasonable testimonies for this case, e.g. we could  concatenate some univariate interpolation polynomial $p$ with $\norm{a_j}^2\norm{X}^2 - \lfp{a_j}^2$: All monomials occurring in $p(\norm{a_j}^2\norm{X}^2 - \lfp{a_j}^2)$ have even degree.  The values of $\norm{a_j}^2\norm{a_i}^2 - \sca{a_j, a_i}^2$ for $i\neq j$ will be distinct from $0$ by Cauchy-Schwarz (as long as no other component $a_i$ is a multiple of $a_j$) and therefore we can  choose an interpolation polynomial $p$ such that $p(\norm{a_j}^2\norm{a_i}^2 - \sca{a_j, a_i}^2) = \delta_{ij}$. This shows that the approach of this paragraph is a valid alternative to the generation of randomised lower degree fake moments in \secref{sec:valgotd}.

\newpage
\section{Efficient Tensor Decomposition via SOS} \label{sec:effdecomp}

Throughout this  section,  let again 
\begin{align}
T_k = \weightdecomp[k]
\end{align} a weighted decomposition of the tensors  $T_0,\alist[d]{T} $. Let  also $v \in\Sn$ be random and reconsider the SOS optimisation problem
\begin{align}
\max &\rezprod{W \lfp{v}}{\sum_{k = 1}^d T_k} \quad  \nonumber \\
\text{over all } &W\in\polyspace{d-2},\: W \succeq 0 \text{ satisfying} \label{eq:sosprog}\\
&\rezprod{W }{\sum_{k = 0}^{d-2} T_k} = 1\nonumber
\end{align}

We've seen that in the setting $ d \geq  2 m $ the algorithm allows for a very clean analysis.  However,  having access to moments of that  high order is a very heavy requirement and results in very  bad running time:   Solving the  degree  $d $ SOS programmes involved will take our black box $n^{2m}$ time steps,  which is exponential  in  $m $.  The obvious next question to ask is what  guarantees we will get if we restrict to SOS programmes  which are polynomial–sized   in both $ m $ and $ n $. (This would be the case  if for instance  $ d $ depended logarithmically on $ m $).  In this   regime,  we can't of course  hope for  $W $ to vanish on all but one $ a_i $.  However,  we can try to search for a  weight polynomial which  attains \emph{very small   values} on most of the components  while being  of \emph{relatively high magnitude} on one particular component.

The  matrix  $M $ will then be only  very close to a rank  $1 $ matrix.  But as long as this approximation is good enough,  we will be able to show that the top  eigenvectors of  $M $ are still very close to one of the components.  To this end,  we will make use of the following lemma: 

\begin{lem}\label{lem:gappedmatrices} (\emph{Eigenvectors of  gapped matrices} - Adapted from Lemma A.3 in \cite{HSS16})\\
	Let  $M \in\R^{n\times n}$ be  a symmetric matrix and let  $a \in\R^n$ ,  $\gamma  >  0$ such that \[
	\spec{M-a\transp{a}} \leq \Mspec - \gamma \|a\|^2
	\]
	Then for each  top  eigenvector $u$ of  $M$  we have  $\sca{u,\frac{a}{\|a\|}}^2 \geq \gamma$.
\end{lem}

\begin{proof} 
	Denote by $\mu_{\max}$ the  eigenvalue of $M$ of largest absolute value. Recall the characterisation\[
	\mu_{\max}^2 = \|M\|_{\text{spec}}^2 = \sca{Mu, Mu}
	\]
	for any  top eigenvector $u$ of $M$. Thus $\Mspec = \sup_{x\in \Sn} |\sca{x,Mx}|$. Let $u$ denote a top eigenvector of $M$. In particular, $\Mspec = |\sca{u,Mu}|$ and therefore
	\[
	\spec{M-a\transp{a}} \geq |\sca{u, (M-a\transp{a})u}|\geq |\sca{u, Mu}| - \sca{u,a}^2 = \Mspec - \sca{u,a}^2
	\]
	Hence $\Mspec - \sca{u,a}^2 \leq \Mspec - \gamma \|a\|^2$, implying $\gamma \|a\|^2 \leq \sca{u,a}^2 $.
\end{proof}

This  will give us a guarantee that the top eigenvector is close to the true  component $a_j$  as long as the $W$-weighted Matrix $M $  is sufficiently close to a rank $1$ matrix. Now we are interested  in the question how close the matrix $M $ can actually become to a rank 1 matrix for a  given value of  $d $. From the preceding discussion  it should be very clear that there are two essential  issues to answer  here:  
\begin{enumerate}
	\item How close can a degree $d~<~2 m $ polynomial come to one of the 0/1 Kronecker delta  functions $\delta_{ij}$ on the set $K~:=~\{ a_1, \ldots,  a_m\} $?
	\item How close will the optimiser $\Wst$ come to one of the $\delta_{ij}$?
\end{enumerate}
The first question asks for the best testimony in the feasible space,  while the second question asks how reliably the optimisation solves for a polynomial which has similar properties as the testimony.
In the exact case,  recall that the optimiser and the testimony (which was an interpolation polynomial) were equally good for our purposes. In the approximate setup,  this will not be the case anymore:  Our results will depend on the choice of $v $,  so we have to make a good choice. Let us first address though the testimony problem: 

\subsection{Finding a Testimony}\label{sec:testimony}

\paragraph*{Univariate Chebyshev Polynomials} 
During this paragraph,  let us  switch to the univariate case, where $\Lambda$  is a scalar algebraic unknown. In this case, there is already a lot of common knowledge about similar optimisation problems that will help us.
\begin{defi}
	The degree-$d$ polynomials $\mathcal{T}_d$ uniquely determined by \begin{align}
	\forall \vartheta\in\R :\quad \mathcal{T}_d(\cos(\vartheta)) = \cos(d \vartheta)
	\end{align}
	are well-known  as  the  Chebyshev polynomials.
\end{defi} 

These polynomials have several characterising  extremal properties one of which is shown in the following lemma.

\begin{lem} \label{lem:chebys_extrprop}
	Let $ x \in\R$  not  in the interval $[-1,1]$.  Consider the optimisation problem 
	\begin{align}
		&\max p(x) \nonumber\\
		\text{subject to }\quad |&p| \leq 1 \:\text{ on } [-1,1]\\
		&p\in \R[\Lambda]_{\leq d}\nonumber
	\end{align}
	The unique optimiser  to this problem is the Chebyshev polynomial  $\mathcal{T}_d$.  
\end{lem}

\begin{proof} (taken from \cite{FP10})
	For any  $d + 1 $ values $\alist[d+1]{t}$  we can express  any $p\in\polyspace{d}$ by the interpolation formula\[
	p = \sum_{j = 1}^{d+1} p(t_i ) \prod_{i\neq j} \frac{ (\Lambda-t_i)}{(t_i - t_j)}
	\]
	Evaluating this formula in  $x $  and using that $|p(t_i)|\leq 1$ yields the bound  
	\begin{align}
	p(x) &= \sum_{j = 1}^{d+1} p(t_i ) \prod_{i\neq j} \frac{ (x-t_i)}{(t_i - t_j)} \\
	&\leq \sum_{j = 1}^{d+1} |p(t_i )| \prod_{i\neq j} \frac{ |x-t_i|}{|t_i - t_j|} \nonumber \\
	&\leq \sum_{j = 1}^{d+1} \: \prod_{i\neq j} \frac{ |x-t_i|}{|t_i - t_j|}
	\end{align}
	This bound is attained if and only if $p(t_i) = \sgn(u_i) $  for all $i\in [d+1] $ where $u_i := \prod_{i\neq j} \frac{ |x-t_i|}{|t_i - t_j|}$. Since  $ x \notin [-1,1]$, the terms $| x -t_i| $  have the same sign for all  $i $.  Together with the fact that  $t_1  < \ldots  <  t_{d+1} $ were sorted in ascending order,  we get $\sgn( u_{i+1})  =  (-1)\sgn(u_i)$,  that is, the signs of $p$ have to alternate  between  $1 $ and  $-1 $ on $ d + 1 $ values.  For a polynomial of degree  $d$ bounded by $1$,  this is only possible  if  $t_2, \ldots, t_d$  are the $d-1$  local maximisers  of $p$ and $t_1, t_{d+1}\in\{\pm 1\}$ are boundary maximisers of $p$ on $[-1,1]$.\\
	All $y\in [-1,1]$  may be represented as $y = \cos(\vartheta)$ for some $\vartheta\in[0, \pi)$.  Using the identity\[
	\mathcal{T}_d(\cos(\vartheta)) = \cos(d\vartheta)
	\]
	and the properties of the cosine function,  we see that over the course of one semi-period of $\cos(\vartheta)$ (that is, as $\vartheta$ increases from $0$ to $\pi$) the values of $\mathcal{T}_d$ perform $d$ semi-periods, attaining $d+1$ times a value of $\pm 1$. Hence the $d$–th Chebyshev polynomial has indeed this oscillating property and is thus the solution to this optimisation problem. 
\end{proof}

\begin{defi} 
	Let $\mathcal{I} = [a,b]$ an interval (where $a < b$). Consider the linear polynomial
	\[
	\psi := \frac{2}{b-a} X - \frac{b + a}{b-a}  
	\]
	mapping $[a,b]$ to  $[-1, 1] $ monotonically.  Then
	\begin{align}
		\chebyinterval := \mathcal{T}_d(\psi(X))
	\end{align}
	is called the $d$-th \emph{Chebyshev polynomial on $\mathcal{I}$}.
\end{defi}

\begin{kor}
	On the interval $\mathcal{I} = [a,b]$, where $b > a > 0$, 
	the Chebyshev polynomial $\chebyinterval$ on $\mathcal{I}$ 
	is the optimiser of 
	\begin{align}
	&\max p(0) \nonumber\\
	\text{subject to }\quad |&p(x)| \leq 1 \:\text{ on } \I\\
	&p\in \R[\Lambda]_{\leq d}\nonumber
	\end{align}
	The optimum value is \begin{align}
	\chebyinterval(0) = \chebystandard\left(\frac{b + a}{b-a}\right) = \frac{1}{2} \left(\kappa + \sqrt{\kappa^2 - 1}\right)^d + \frac{1}{2}\left(\kappa - \sqrt{\kappa^2 - 1}\right)^d \in\landauO{\kappa^d}
	\end{align}
	where we used the abbreviation $\kappa := \frac{b + a}{b-a}$. The optimum value grows exponentially with $d$. 
\end{kor}

\begin{proof} The optimality property follows trivially from Lemma \ref{lem:chebys_extrprop}. For the explicit value of the optimum we used the well-known fact that \[\chebystandard(x) = \frac{1}{2} \left(x + \sqrt{x^2 - 1}\right)^d + \frac{1}{2}\left(x - \sqrt{x^2 - 1}\right)^d\] for $|x| > 1$ and that $\kappa = \frac{b + a}{b-a} > 1$ since $0 < a < b$.
\end{proof}

\paragraph{Multivariate Chebyshev Polynomials}  Now, from univariate Chebyshev polynomials it's just  a small step to  multivariate but rotation symmetric polynomials: 
\begin{pro}  
	Let $d\equiv 2 \mod 4$, $x\in\R^n $ and $q  : = \|X-x\|^2$. We have that 
	\begin{align}
	\chebyinterval[(d-2)/2](q) + 1
	\end{align}
	is  a rotation symmetric\footnote{that is, w.r.t rotations around $x$} SOS polynomial. 
\end{pro}

\begin{proof}  Since $ d \equiv 2\mod 4 $ we know that $\chebyinterval[(d-2)/2]$ is an even degree polynomial. At any of its local minima it attains the value $-1$. Therefore, $\chebyinterval[(d-2)/2]+1$ is a nonnegative polynomial in one variable – hence a sum of squares. The concatenation with another SOS polynomial $q$ yields of course again an SOS polynomial.  Symmetry under rotations around  $x $ is clear.
\end{proof}

Note that this polynomial has a local maximum at $X = x$. It attains high values  on $y$ when  $\norm{y-x}$ is either very small or very big. For moderately sized $\norm{y-x}^2\in [a,b] = \I$ it will attain low values. Therefore, such a polynomial  does just what we want if $ x = a_j $ and $ \norm{a_j-a_i}\in[a,b]$ for all $i\neq j$. We could thus take \[
[a,b] = [\kmin,\kmax]
\]
(with $\kmin,\kmax$ as defined in \ref{not:cond_params}), but it might be wise to be cautious and leave some extra space,  e.g. to take \[
[a,b] = [\frac{1}{2}\kmin - \delta, 2\kmax + \delta]
\]
for some small $\delta \geq 0$.\footnote{We will need this $ \delta$-wide safety margin in section \secref{sec:all_comps}. The factor of 2 is needed due to an unfortunate suboptimality of the SOS proof we will give there.} Let us therefor 
write \[
W_{j, \delta}  := \mathcal{T}_{[\kmin - \delta,\: \kmax + \delta],(d-2)/2}(\norm{X - a_j}^2) + 1
\]
These SOS polynomials aren't all members of the feasible space,   since in general they will not fulfill\[
\sum_{i=1}^m W_{j, \delta}(a_i) =1
\]
Of course this can be fixed by a simple rescaling,  but we do not know the exact scaling factor. Let $W_{j, \delta,\text{test}}$  denote the correctly rescaled version of $W_{j, \delta}$.  What we're actually interested in is \[
\wminmaxtest := \wminmaxtest^{j,\delta} := \supnormI{W_{j, \delta,\text{test}}} := \sup_{\norm{x-a_j}^2\in \mathcal{I}} W_{j, \delta,\text{test}}(x)
\]
The following proposition deals with the unknown scaling factor by estimating $\wminmaxtest$. Let us in the following suppress $j $ and $\delta$ (and sometimes also $\I$) in the notation.

\begin{pro} \label{pro:guarantees_testimony}
	Let $d\equiv 2 \mod 4$. Fix some $j\in[m]$. For the polynomials $q  : = \|X-a_j\|^2$ and $\Chebtest := \chebyinterval[(d-2)/2] + 1 $ we have that 
	\begin{align}
		\Chebtest(q) = \chebyinterval[(d-2)/2](q) + 1 \label{eq:shiftedcheby}
	\end{align}
	is  a rotation symmetric\footnote{that is, w.r.t rotations around $a_j$} member of the feasible space of  $\eqref{eq:sosprog}$ after rescaling. Let us denote by $\wtest$ the unique multiple of $\Chebtest(q)$ such that
	\begin{align}
	\sum_{i = 1}^m \lambda_i \wtest(a_i) = 1
	\end{align}
	and write 
	\begin{align}
	\wminmaxtest := \supnormI{\wtest}
	\end{align}
	We claim that 
	\begin{align}
	&\qquad \bestratio := \frac{\Chebtest(0)}{2} = \frac{\wtest(a_j)}{\wminmaxtest} \quad \text{and} \\
	\wminmaxtest \leq \frac{1}{\bestratio \lambda_j} 
	\leq  &\frac{2}{\lambda_j} \left(\frac{1}{2}\left(\kappa + \sqrt{\kappa^2 - 1}\right)^d + \frac{1}{2}\left(\kappa - \sqrt{\kappa^2 - 1}\right)^d  + 1\right)^{-1} \in \landauO{\kappa^{-d}} \label{eq:wminmaxestimate}  
	\end{align}
	where $\kappa = \frac{2\kmax + \frac{1}{2} \kmin}{2\kmax - \frac{1}{2} \kmin + 2\delta}$
\end{pro}

\begin{proof}  Note that the ratio $\bestratio = \frac{\Chebtest(0)}{\supnormI{\Chebtest}}$ is invariant under rescaling,  whence we clearly have $\bestratio = \frac{\wtest(a_j)}{\wminmaxtest} $. Unlike $\wminmaxtest $ though,  we can explicitly compute $\bestratio$. For the estimate,  note that 
\begin{align*}
\bestratio &= \frac{\wtest(a_j)}{\wminmaxtest} \leq \frac{1 }{\lambda_j\wminmaxtest} \\
\Longleftrightarrow  \wminmaxtest &\leq (\bestratio \lambda_j)^{-1} =  \frac{2 }{\lambda_j\Chebtest(0)}  
\end{align*}
This shows the first inequality in \eqref{eq:wminmaxestimate}. The second inequality follows by plugging in  the known explicit value of $\bestratio$ or $\Chebtest(0)$, respectively.
\end{proof}

We will usually not work with the monstrosity that is Equation \eqref{eq:wminmaxestimate},  but use  instead the qualitative  fact that $\wminmax \leq \frac{1}{\bestratio \lambda_j}$ drops exponentially with $d$.

\subsection{Guarantees for the Optimiser} Now that we know that there is a  testimony in the feasible space which has the properties we want,  we still have to show that the optimiser $\Wst$  will concentrate on precisely one of the $a_i$ (like the testimony did).  Alas,    this problem becomes much more difficult in the approximate setting:  When we had  the perfect testimony satisfying \eqref{eq:deltaproperty} in the feasible space,  we could easily deduce \eqref{eq:deltapropertyforW}.  This is due to the fact that if $\I_j$ does not waste any concentration on a sub-optimal component,  then neither can $\Wst$.  Now however,  our testimony will likely waste some concentration   on components  with a very low value of $\sca{a_i,v}$ and the optimiser could theoretically  beat the testimony  by putting  e.g.  all of that concentration on the  component achieving the second largest value of $\sca{a_i,v}$,  provided that $\Delta_{j,i,v} := \sca{a_j,v} - \sca{a_i,v}$ is very small. 

The latter event is unlikely to happen, if $\norm{a_j}\geq \norm{a_i}$ and the components are not too correlated, but the choice of  $v $,  whether it's done randomly or  deterministically,  is still a non-trivial  issue which   we will address in   its own section.  For the moment,  let us just see what guarantees we get for a fixed value of $\Delta_{j,i,v}$.  To be slightly more general,  let us further replace  the discriminatory  polynomial $\lfp{v}$ by an arbitrary polynomial $f$ of degree at most two.  Likewise,  $ \Delta_{j,i,v}$ is replaced by $\dvobfun$.

\begin{lem} \label{lem:guarantees_concentration} Let $\Wst$ denote the optimiser of
\[
\max \sum_{i = 1}^m W(a_i) \obfun \quad (\ast)
\]
where $f$ is some polynomial in $X$ of degree less or equal $2$ (e.g. $f = \lfp{v}$ or $f = \lfp{v}^2$ or $f = (X_1^2 + \ldots + X_n^2)$).
Then 
\begin{align}
	\sum_{i\neq j} \lambda_i \Wst(a_i) \leq \wminmax\cdot \norm[1]{\lambda} \frac{\max_{i\neq j} \Delta_{j, i, f}}{\min_{i\neq j} \Delta_{j,i,f}}
\end{align}
\end{lem}

\begin{proof} Let us write again for brevity \[
	\Delta_{j,i,f} := f(a_j) - f(a_i) \geq 0
	\]
	where we suppose again that $a_j$ is the maximiser of $f$ among all $i$.
	Since $\Wst$ is the optimiser of $(\ast)$, we have 
	\begin{align*}
		&\sum_{i = 1}^m \lambda_i \Wst(a_i) \obfun - \sum_{i = 1}^m \lambda_i \wtest(a_i) \obfun \geq 0 \\
		\Longleftrightarrow &\left(f(a_j) - \sum_{i\neq j} \lambda_i \Wst(a_i) \dvobfun \right)-\left(f(a_j) - \sum_{i\neq j} \lambda_i \wtest(a_i) \dvobfun \right)  \geq 0 \\
		\Longleftrightarrow &\sum_{i\neq j} \lambda_i (\wtest(a_i)-\Wst(a_i)) \dvobfun   \geq 0 
	\end{align*}
	If $\wtest(a_i) \geq \Wst(a_i)$ for some $i\neq j$, then $\wtest(a_i)-\Wst(a_i) \leq \wminmax$. Otherwise we use $\wtest(a_i)~-~\Wst(a_i)~\leq~\wminmax~-~\Wst(a_i)$. 
	Combining these, we can upper-bound the left hand side and get
	\begin{align*}
	\qquad \sum_{\substack{\tiny i\neq j\\ \wtest(a_i) \geq \Wst(a_i) }} \:\: \lambda_i \wminmax \dvobfun &-  \sum_{\substack{i\neq j\\ \wtest(a_i) < \Wst(a_i) }} \lambda_i  (\Wst(a_i) - \wtest(a_i))  \dvobfun \geq 0\\
	\implies \sum_{\substack{i\neq j\\ \wtest(a_i) \geq \Wst(a_i) }} \lambda_i \wminmax \dvobfun    &\geq  \sum_{\substack{i\neq j\\ \wtest(a_i) < \Wst(a_i) }} \lambda_i  (\Wst(a_i) - \wtest(a_i))  \dvobfun \geq 0 \\
	\implies  \sum_{\substack{i\neq j\\ \wtest(a_i) \geq  \Wst(a_i) }} \lambda_i \wminmax \dvobfun    &\geq  \sum_{\substack{i\neq j\\ \wtest(a_i) < \Wst(a_i) }} \lambda_i  (\Wst(a_i) - \wminmax)  \dvobfun \\
	\implies \qquad \sum_{i\neq j} \:\:\qquad \lambda_i \wminmax \dvobfun  &\geq  \sum_{\substack{i\neq j\\ \wtest(a_i) < \Wst(a_i) }} \lambda_i  \Wst(a_i)  \dvobfun 
	\end{align*}
	Hence \begin{align}
	\sum_{i\neq j} \lambda_i  \Wst(a_i)  \leq (\min_{i\neq j}\dvobfun)^{-1} \sum_{i\neq j} \lambda_i \wminmax \dvobfun \leq (\min_{i\neq j}\dvobfun)^{-1} \norm[1]{\lambda} \wminmax \max_{i\neq j} \dvobfun
	\end{align}
	or, equivalently
	\begin{align}
	\lambda_j W(a_j) \geq 1 - \norm[1]{\lambda} \wminmax \frac{\max_{i\neq j} \dvobfun}{\min_{i\neq j}\dvobfun}
	\end{align}
	
\end{proof}

Note that for $f  = \lfp{v}^2$ or 
$ f  = \lfp{v}$, the maximum delta value can easily be bounded by $\max_{i\neq j}\norm{a_i-a_j}\norm{a_i+a_j}$  or $\max_{i\neq j}\norm{a_i-a_j}$,  respectively.  The problem is thus to find a polynomial $f$ or a choice of  $v $ such that the minimum delta value can be lower bounded. This seems to be a very hard task  to make deterministically,  but the chance of getting a good choice of $v $ at random is not too bad,  as we will see. In certain special cases,  there might be  systematic choices,  especially when prior knowledge to the domain  of the $a_i$ is given (e.g. if they lie on the  $\{\pm 1\}^n$-hypercube). 

\subsection{Approximate Recovery of One Component} \label{sec:one_component}
Now we've got everything we need in order to formulate the recovery step for a single component,  supposed we are  given the correct choice of  $v $ as a parameter.  To be precise,  the parameter is actually some discriminatory polynomial $f$ which can of course be chosen as $\lfp{v}$. This kernel step is formulated in Algorithm \ref{alg:approxDecomp_onestep}.  It will return some component  up to an error term that can be made infinitely small as $ d $ grows.

\begin{algorithm}[H]
	\caption{Approximate $v$-algorithm, one component, $f$ given as parameter} 
	\label{alg:approxDecomp_onestep}
	\textbf{Input: }Tensors $T_0 \in \hompolyspace{0}, \ldots, T_d\in \hompolyspace{d}$ \\
	\textbf{Parameters: } Some $f\in\polyspace{2}$.\\
	\textbf{Output: } One vector $c$ satisfying $\sca{c, a_j} \geq  \big(1 - 2 \wminmax\cdot \norm[1]{\lambda}\cdot \frac{\speccorrel}{ \|a_j\|^2} \cdot \frac{\max_{i\neq j} \Delta_{j,i,f}}{\min_{i\neq j} \Delta_{j,i,f}}\big)^{1/2}$ 
	and $|\norm{c}^2 - \norm{a_j}^2| \leq \wminmax\cdot \norm[1]{\lambda}\cdot \frac{\speccorrel}{ \|a_j\|^2} \cdot \frac{\max_{i\neq j} \Delta_{j,i,f}}{\min_{i\neq j}\Delta_{j,i,f}}$ for some $j\in[m]$.\\
	\textbf{Require: }$d \equiv 2 \mod 4$ and there should exist a simultaneous decomposition $T_k = \weightdecomp[k]$ of the input with $m$ distinct components. 
	\\
	\textbf{Procedure: }
	\begin{algorithmic}[1]
		\STATE{$\mathbf{Solve}$ the SOS optimisation problem} \[
			\max \rezprod{W \lfp{v}}{\sum_{k = 1}^d T_k} \quad (\ast)
			\]
		over all polynomials $W \succeq 0$ satisfying $\deg(W) \leq d-2$ and $\rezprod{W }{\sum_{k = 0}^{d-2} T_k} = 1$.
		\STATE{$\mathbf{Let }$ $W^\ast $ denote the output of this optimisation problem.} 
		\STATE{$\mathbf{Compute}$ the matrix / quad. form $M := \rezprod{W^\ast \tens \id}{\sum_{k = 2}^{d} T_k}$ }
		\STATE{$\mathbf{Compute}$ an eigendecomposition of $M$}  with eigenvalues and corresponding eigenvectors of unit length.
		\STATE{$\mathbf{Let }$ $u$ denote the eigenvector corresponding to the largest eigenvalue $\mu$.}
		\STATE{$\mathbf{Decide}$ the factor of $\pm 1$: Let $L := \sum_{i= 1}^m \lambda_i W(a_i) a_i$. If $\sca{u,L} > 0$, do not change $u$. Otherwise, replace $u:= -u$. }
		\STATE{$\mathbf{Output }$ $c := \sqrt{\mu} u$} as the component and $\rho := \frac{1}{W(c)}$ as the weight.
	\end{algorithmic}
\end{algorithm}

The procedure itself did not alter that much. We will start again by proving  that the first round of the algorithm  will indeed  yield an approximation to  some  component   of the  $T_k$. To this end,   we will prove the following theorem. 

\begin{lem}\label{lem:guarantees_recovery}
	Let $j\in[m]$ and $W$ an SOS polynomial such that $\sum_{i = 1}^{m} W(a_i) = 1 $ and $\sum_{i\neq j} W(a_i)\leq C \in (0,1)$. Then each top eigenvector $u$ of $M := \sum_{i = 1}^m \lambda_iW(a_i) a_i\transp{a_i}$ satisfies 
	\begin{align*}
	\sca{\frac{a_j}{\|a_j\|}, u}^2 \geq \gamma 
	\end{align*}
	for $\gamma = 1 - \frac{2 C \speccorrel}{\|a_j\|^2}$. Here, $\speccorrel$ is one of the condition parameters from \ref{not:cond_params}. In particular, if $\sca{a_i, u} \geq 0$ then
	\[
	\|u-a_j\|_2^2 = \|u\|_2^2 +  \|a_j\|_2^2 - 2\sca{a_i, u} \leq 1 + (1-2\frac{\sqrt{\gamma}}{\norm{a_j}}) \|a_j\|_2^2 
	\]
\end{lem}

\begin{proof} 
	Denote by $\mu_{\max}$ the  eigenvalue of $M$ of largest absolute value. Recall that $\Mspec$ is the absolute value of the largest eigenvalue of $M$. We have thus\footnote{$M$ is psd. Hence the $|\cdot|$ is actually  unnecessary.} \[
	\|M\|_{\text{spec}} = |\mu_{\max}| = \sup_{x\in \Sn} |M\sca{x, x}| 
	\]
	We want to use Lemma \ref{lem:gappedmatrices}. To this end, we need a lower bound for $\Mspec$ and an upper  bound  for  $\Mdiffspec$:
	\begin{align*}
		\norm{a_j}^2\|M\|_{\text{spec}} &\geq M\sca{a_j, a_j} = \sum_{i = 1}^m \lambda_i W(a_i) \sca{a_i, a_j}^2 \\
		&\geq (1 - \sum_{i \neq j} \lambda_i W(a_i)) \sca{a_j, a_j}^2 + \sum_{i \neq j} \lambda_i W(a_i) \sca{a_i, a_j}^2 \\
		&\geq \norm{a_j}^4 - \sum_{i \neq j} \lambda_i W(a_i) \left(\sca{a_j, a_j}^2 - \sca{a_i, a_j}^2\right) \\
		&\geq \norm{a_j}^4 - C \norm{a_j}^2 \max_{i\neq j} \left(\sca{a_j, \frac{a_j}{\norm{a_j}}}^2 - \sca{a_i, \frac{a_j}{\norm{a_j}}}^2\right)\\
		&\geq \norm{a_j}^4 - C \norm{a_j}^2 \max_{i\neq j} \sup_{x\in \Sn} \left(\sca{a_j, x}^2 - \sca{a_i, x}^2\right)\\
		&\geq \norm{a_j}^4 - C \norm{a_j}^2 \max_{i\neq j} \spec{a_i\transp{a_i} - a_j\transp{a_j}}\\
		&= \norm{a_j}^4 - C \norm{a_j}^2 \speccorrel
	\end{align*}
	Dividing by $\norm{a_j}^2$,  we get the bound
	\begin{align}
		\Mspec \geq \norm{a_j}^2 -  C \speccorrel \label{eq:Mspec_lowerbound}
	\end{align}
	Likewise we get, using $\lambda_j W(a_j) - 1 = -\sum_{i \neq j} \lambda_i W(a_i) $
	\begin{align*}
		\Mdiffspec &= \sup_{x\in \Sn} \big|\sum_{i = 1}^m \lambda_i W(a_i) \sca{a_i, x}^2 - \sca{a_j, x}^2\big|\\
		&= \sup_{x\in \Sn} \big|\sum_{i \neq j} \lambda_i W(a_i) \sca{a_i, x}^2 + (\lambda_j W(a_j) - 1) \sca{a_j, x}^2\big|\\
		&= \sup_{x\in \Sn} \big|\sum_{i \neq j} \lambda_i W(a_i) \left(\sca{a_i, x}^2 - \sca{a_j, x}^2\right)\big|\\
		&\leq \sup_{x\in \Sn}\:\: \sum_{i \neq j} \lambda_i W(a_i) \left|\sca{a_i, x}^2 - \sca{a_j, x}^2\right|\\
		&\leq \:\:\sum_{i \neq j} \lambda_i W(a_i)  \sup_{x\in \Sn} \left|\sca{a_i, x}^2 - \sca{a_j, x}^2\right|\\
		&\leq C \speccorrel
	\end{align*}
	Therefore \[
	\Mspec - \Mdiffspec \geq  \|a_j\|^2 - 2C \speccorrel =  \|a_j\|^2 \left(1 - \frac{2 C \speccorrel}{\|a_j\|^2}\right)
	\]and thus we may apply lemma \ref{lem:gappedmatrices} with \[
	\gamma := 1 - \frac{2 C \speccorrel}{\|a_j\|^2}
	\]
\end{proof}

It remains to put all of these bounds together to prove our guarantees for the  output.

\begin{sat} Let $j\in[m]$ such that $f(a_j)$ is maximal among all $a_i$. 
	Then each top eigenvector $u$ of $M := \sum_{i = 1}^m \lambda_i\Wst(a_i) a_i\transp{a_i}$ satisfies 	
	\begin{align*}
	\sca{\frac{a_j}{\|a_j\|}, u}^2 \geq 1 - 2 \wminmax\cdot \norm[1]{\lambda}\cdot \frac{\speccorrel}{ \|a_j\|^2} \cdot \frac{\max_{i\neq j} \Delta_{j,i,f}}{\min_{i\neq j} \Delta_{j,i,f}}
	\end{align*}
\end{sat}

\begin{proof} By Lemma \ref{lem:guarantees_concentration}, we have $\sum_{i\neq j} \Wst(a_i)\leq \wminmax\cdot \norm[1]{\lambda} \frac{\max_{i\neq j} \Delta_{j, i, f}}{\min_{i\neq j} \Delta_{j,i,f}}$. Therefore we may choose $C := \wminmax\cdot \norm[1]{\lambda} \frac{\max_{i\neq j} \Delta_{j, i, f}}{\min_{i\neq j} \Delta_{j,i,f}}$ in Lemma \ref{lem:guarantees_recovery} and get thus a vector $u$ satisfying
	\begin{align*}
	\sca{\frac{a_j}{\|a_j\|}, u}^2 \geq 1 - 2 \wminmax\cdot \norm[1]{\lambda}\cdot \frac{\speccorrel}{ \|a_j\|^2} \cdot \frac{\max_{i\neq j} \Delta_{j,i,f}}{\min_{i\neq j} \Delta_{j,i,f}}
	\end{align*}
	out of $M$.
\end{proof}

Let us quickly discuss the meaning of the  constants occurring in our  bound of the error term:   $\wminmax$ is the friendly approximation constant that we can make as small as we want as we let $d\to\infty$ – or perhaps it's better to say $d\to 2m$.  Note that $\wminmax$ will decrease exponentially with $ d $ and that it depends only on $ d $ and the condition parameters $\kmin$ and $\kmax$.\footnote{As long as we are  recovering only one component,  we do not yet need the safety margin $\varepsilon$,  but this will come in soon.} As we saw  already in the preliminaries,  
$\speccorrel$ is part of the condition parameters as well and we would have expected  something related to the weights like $\norm[1]{\lambda}$ to occur  in there anyway. The term  $\|a_j\|^2$ in the denominator is slightly weirder: It tells us that we can't recover components of small norm. First, we'd actually expect  the weight $\lambda_j$ to occur there, too, since it should be ok to recover a vector of small  length as long as  the weight is sufficiently big. We can indeed pull $\lambda_j$ off the hat by using the estimate $\wminmaxtest \leq (\lambda_j \bestratio)^{-1}$. 

But there is still a problem with vectors close to the zero vector. 
This seems to be an issue with our recovery technique:  These eigenvector methods all rely on vectors of unit length and it's probably due to the  side effects of rescaling  that everything gets more unstable for small length components.  I'm not entirely sure with this,  though.  It would thus be preferable if we could without loss of generality assume that the components  were of unit length and rescale the weights accordingly. However,  to be able to do so  while maintaining  full generality,  we would need to be able to generate the fake moments \[
T_k = \sum_{i = 1}^m s_i^k \|a_i\|^{k} \lfp{a_i}^k
\]
for some $s_i \in\{\pm 1\}$, as outlined in \secref{sec:qualitatives}.

Of course the most problematic  term is $\frac{1}{\min_{i\neq j} \Delta_{j,i,f}}$.  If we chose the wrong discriminatory polynomial,  then this term might actually cast a well-conditioned problem into an ill-conditioned one. Therefore,  we have to take care that at least we do not make  one of the worst possible choices. We dedicate the next section to this problem.

But first we have to show that  the length and  the factor of  $ \pm 1$ are correct, too. Note that from Eq. \eqref{eq:Mspec_lowerbound} in the proof of Lemma \ref{lem:guarantees_concentration} we already know that the top  eigenvalue  $\mu_1$ satisfies
\begin{align}
 \mu_1 = \Mspec \geq \norm{a_j}^2 - C\speccorrel = \norm{a_j}^2 - \wminmax\cdot\speccorrel \norm[1]{\lambda} \frac{\max_{i\neq j} \Delta_{j, i, f}}{\min_{i\neq j} \Delta_{j,i,f}}
\end{align}
By a very similar calculation,  we also get the upper bound
\begin{align*}
\Mspec &= \sup_{x\in \Sn} \big|\sum_{i = 1}^m \lambda_i W(a_i) \sca{a_i, x}^2 \big|\\
&= \sup_{x\in \Sn} \big|\sca{a_j, x}^2 + \sum_{i \neq j} \lambda_i W(a_i) \left(\sca{a_i, x}^2 - \sca{a_j, x}^2\right)\big|\\
&\leq \norm{a_j}^2 + \sup_{x\in \Sn}\:\: \sum_{i \neq j} \lambda_i W(a_i) \left|\sca{a_i, x}^2 - \sca{a_j, x}^2\right|\\
&\leq \norm{a_j}^2 + C \speccorrel
\end{align*}
Hence \begin{align}
|\Mspec - \norm{a_j}^2| \leq  C \speccorrel = \wminmax\cdot \speccorrel \norm[1]{\lambda} \frac{\max_{i\neq j} \Delta_{j, i, f}}{\min_{i\neq j} \Delta_{j,i,f}}
\end{align}
Therefore $|\norm{c}^2 - \norm{a_j}^2| = |\mu_1 - \norm{a_j}^2| = |\Mspec - \norm{a_j}^2| \leq \wminmax\cdot \speccorrel \norm[1]{\lambda} \frac{\max_{i\neq j} \Delta_{j, i, f}}{\min_{i\neq j} \Delta_{j,i,f}}$ as claimed in Alg. \ref{alg:approxDecomp_onestep}, which shows approximate correctness of the length. For the sign, let again
\[
L = \sum_{i= 1}^m \lambda_i W(a_i) a_i
\]
and note that
\begin{align}
\norm{L - a_j} &= \norm{\sum_{i \neq j} \lambda_i W(a_i) (a_i - a_j)} \leq |\sum_{i \neq j} \lambda_i W(a_i)| \max_{i\neq j} \norm{a_i - a_j} \nonumber \\
&\leq C \cdot \sqrt{\kmax} = \sqrt{\kmax} \cdot \wminmax\cdot \norm[1]{\lambda} \frac{\max_{i\neq j} \Delta_{j, i, f}}{\min_{i\neq j} \label{eq:L-aj} \Delta_{j,i,f}}
\end{align}
It's clear that the scalar product $\sca{L, u}$ will be positive if and only if $L $ is closer  to  $u$ than to $-u$.  Equivalently, we can of course replace $L $ by $L\cdot\norm{a_j}$. Set $\eta := \frac{a_j}{\norm{a_j}}$. The correct  sign choice for  $u $ achieves $\norm{u-\eta} \leq \norm{u - (-\eta)}$. 
By Lemma \ref{lem:guarantees_recovery}  we know that there exists a sign choice $s\in\{\pm 1\}$ such that 
\begin{align}
\norm{su - \eta}^2 =  2 - 2\sca{\eta,su} \leq 2- 2 \sqrt{1 - \frac{2C\speccorrel}{\norm{a_j}^2}} =: \delta\in \landauO{C} \label{eq:su-eta}
\end{align}
Now by the triangle inequality, \eqref{eq:L-aj} and \eqref{eq:su-eta}, 
we have\[
\norm{su-L/\norm{a_j}} \leq \norm{su - \eta} + \norm{\eta - L/\norm{a_j}} \leq \sqrt{\delta} + C \cdot \norm{a_j}\sqrt{\kmax}\approx \sqrt{C} + C  \ll 1
\]
whereas for the wrong sign choice
\begin{align*}
\norm{su + L/\norm{a_j}} &= \norm{(su-L/\norm{a_j}) + 2L/\norm{a_j}} 
 \geq |2\norm{L}/\norm{a_j} - \norm{su-L/\norm{a_j}}| \\
 &\geq 2\frac{\norm{L}}{\norm{a_j}} - (\sqrt{\delta} + C \cdot \norm{a_j}\sqrt{\kmax}) \geq \norm{L}\\
 &\geq 2- (\underbrace{2C \frac{\sqrt{\kmax}}{\norm{a_j}} + \sqrt{\delta} + C \cdot \norm{a_j}\sqrt{\kmax}}_{=: \tilde{C}} )
\end{align*}
where the last step holds due to $\frac{\norm{L}}{\norm{a_j}} \geq 1 - \frac{\norm{L - a_j}}{\norm{a_j}} \geq 1-C \frac{\sqrt{\kmax}}{\norm{a_j}}$. If $C$ is small enough such that $\tilde{C} \leq 1$ clearly we can distinguish the correct and the wrong $s$ since for the correct one $0 \approx \norm{su-L/\norm{a_j}}$ and $\norm{su+L/\norm{a_j}} \gg 1$. 
We apologise for the sloppiness in this one,  but since this is about  making a binary decision, we do not have to care about how the exact error looks like as long as it will decrease fast with $d $.\\

Note that we didn't  give any guarantee  regarding correctness of the weight. This is very much intentional:  $W$ will  usually be a (moderately) high degree polynomial and  thus if $c $ is only slightly off from the true component $a_j $,  $W(c)$ might already be a terrible approximation for $W(a_j)$. Hence,  trusting the continuity of $W $ would be madness!  A much better idea is to  recover all components $c_i$ first and then get the weights by solving a linear least squares programme of the kind
\[
\min_{\lambda \in \R_{> 0}^n} \norm[F]{T_d - \sum_{i = 1}^m \lambda_{i} \lfp{c_i}^d}^2
\]

It's of course also possible to boost  the accuracy of the components by running  some local searching procedure afterwards. As Hopkins et al. pointed out in \cite{HSS16},  there is a lot of potential for synergy here,  since these SOS based procedures output exactly what a local searching procedure needs:  Some approximate solution near the global  minimum of \[
\sum_{k = 0}^d \norm[F]{T_k - \sum_{i = 1}^m \lambda_{i} \lfp{c_i}^k}^2
\]

\subsection{Choice of $v$} \label{sec:v_choice}

One thing we had to leave open up to now  is the choice of $v $,  or,  more generally,  the choice of a suitable discriminatory polynomial $ f $. Up to now,  we've been always trying to keep full generality. However,  we currently do not know of any choice of discriminatory polynomials that is universally good in the fully general case. Let us switch therefore to the case $\norm{a_i} = 1$ for all $i\in[m]$. Also, it turns out that the choice $f = \lfp{v}^2$ instead of $\lfp{v}$ will make our life easier.  Let us therefore redefine \[
\Delta_{j,i,v} := \lfv{j}^2 - \lfv{i}^2
\] 
Now, we should first  estimate the chances to find a good $v \sim \mathcal{U}$ at random. Precisely, we're interested in the conditional probability 
\begin{align}
\prob_{v\sim \mathcal{U}}\left[\lfv{j}^2 - \lfv{i}^2 \geq r \mid \forall k \neq i,j : \lfv{j}^2 \geq \lfv{i}^2 \geq \lfv{k}^2\right]
\end{align}
that $\Delta_{j,i,v}$ attains a value greater or equal to some $r\in \R_{\geq 0}$ subject to the event $\sca{a_j,v}^2 \geq \sca{a_i,v}^2 \geq \lfv{k}^2$ for all $k\neq i,j$. We will  brute-force-estimate this probability by just considering the event
\begin{align}
\mathcal{E} := \{v\in\Sn \mid\lfv{j}^2\geq (1-\gamma)\norm{a_j}^2\}
\end{align}
that $a_j$ is $(1-\gamma)$-square-correlated with $v$ for some small $\gamma \in(0,1)$. This event has probability \begin{align}
\prob_{\mathcal{U}}\left[\mathcal{E}\right] = \frac{2}{\pi}\arccos(\sqrt{1-\gamma}) \in (0,\: 1)
\end{align}
which is constant in $n$ and $m$. From that,  we will indeed be able to show that whenever $\mathcal{E}$ happens,  then $\lfv{j}^2 - \lfv{i}^2 \geq r$ for some $r $ (here we use $\norm{a_i}^2 = \norm{a_j}^2$). This is precisely formulated in the following lemma. Note that this estimation is really brute,  since we didn't even condition on the event $\forall k \neq i,j : \sca{a_j,v}^2 \geq \sca{a_i,v}^2 \geq \lfv{k}^2$. 

\begin{lem} \label{lem:delta_bound}
	Suppose all $\norm{a_i} = 1$ and $v\in \mathcal{E} := \{v\in\Sn \mid\lfv{j}^2\geq (1-\gamma) \}$ for some fixed $j\in[m]$ and $\gamma \in(0,1)$. Let $\minspeccorrel := \min_{i,j\in[m], i\neq j} \projdist{a_i,a_j}^2$. Then   
	\begin{align*}
	\forall i\neq j : \quad \Delta_{j,i,v} \geq (\frac{1}{2}\minspeccorrel- 2\gamma) 
	\end{align*}
\end{lem}

\begin{proof}
	By the ``square'' triangular inequality for $\projdist{\cdot,\cdot}^2$:\[
	\forall x,y,z\in\Sn : \projdist{x,y}^2 \leq  2\projdist{x,z}^2 + 2 \projdist{z,y}^2
	\] we have that 
	\begin{align*}
	&\lfv{j}^2 -\lfv{i}^2\\
	=\: &\projdist{a_i,v}^2 - \projdist{a_j,v}^2 \\
	\geq\: &\frac{1}{2}\projdist{a_i,a_j}^2 - 2\projdist{a_j,v}^2 \\
	\geq\: &\frac{1}{2}\minspeccorrel - 2\projdist{a_j,v}^2 \\
	\geq\: &\frac{1}{2}\minspeccorrel - 2\gamma
	\end{align*}
	where for the fist estimate we choose $ x = a_i,  y  = a_j, z = v$.
\end{proof}
This shows $\lfv{j}^2 -\lfv{i}^2 > 0$ for sufficiently small $\gamma$. We get a bound  just in terms of the correlation. 

It's an interesting question to ask to what extent we need prior knowledge on the domain of the components in order to construct a good discriminatory polynomial. We'll see  that we can use the bound above to search for moment decompositions  where  the components  lie on the unit sphere. This suggests that prior knowledge  on the domain can do a lot,  but recall that we didn't need it in the exact case.

\paragraph{Conditioning} Now, of course we could try  to run the SOS programme with a random choice of  $v $ and see afterwards whether or not it worked. However,  SOS programmes are expensive to solve and therefore we will precondition the choice of  $v $ such  that we get a success probability of 1 again,  provided certain conditions are met. To this end,  let us observe that if $k$ is even, any candidate for  $v $ that is highly correlated with some $a_i$ (i.e. $\sca{a_i,v}^2 \geq (1-\delta)\norm{a_i}^2$ for some small $\delta \in (0,1)$) will attain a high value of  $T_k(v)$. Indeed, this is due to the estimate that for any component $a_i $ and any even  $k $,  we have \[
T_k(v) \geq \lambda_i \sca{a_i, v}^k = \lambda_i (1-\delta)^{k/2}\norm{a_i}^k
\] Hence  an obvious precondition would be to filter out all choices that  attain a low value of  $T_k(v) $. Once we're left with  a unit vector $v $ on which $T_k $ attains a high value,  we can do the following, quantitatively weaker, converse implication: 

By a simple averaging / convexity argument that we adapted from \cite{BKS15},  we know that 
\begin{align}
\exists j\in[m] : \sca{a_j, v}^k \geq \frac{T_k(v)}{\norm[1]{\lambda}} 
\end{align}
This is true  since $T_0^{-1} T_k(v) = \sum_{i= 1}^m \frac{\lambda_{i}}{\norm[1]{\lambda}} \sca{a_i,v}^k$ is a convex combination of the $\sca{a_i,v}^k$. It  implies that $|\sca{a_j, v}| \geq \sqrt[k]{\frac{T_k(v)}{\norm[1]{\lambda}}}$. We have thus seen that it is feasible to condition on \[
\exists i\in[m] :  |\sca{a_i,v}| \geq \sqrt[k]{\frac{1}{\norm[1]{\lambda}}}(1-\delta)^{1/2} R
\] for any fixed $\delta \in (0,1)$ and any $R\in[\min_{i\in[m]} \sqrt[k]{\lambda_{i}} \norm{a_i},\: \max_{i\in[m]} \sqrt[k]{\lambda_{i}} \norm{a_i}]$ with constant success probability at least $\frac{2}{\pi} \arccos(\sqrt{1-\delta})$. Recall that we had the unit sphere assumption. Thus we can choose $R$ s.t. $R^k = \min_{i\in[m]} \lambda_{i} =: \lambda_{\min}$.\footnote{It might be surprising that we do not take the maximum. The reason for this is that in the final algorithm, we will need that any $i\in[m]$ can beat the threshold.}

Of course the process of sampling and discarding takes time,  but first of all it's constant time and secondly it is likely not  near as much time as solving an SOS  programme.  For that reason,  we will ignore the effect this conditioning has  on the running time. We will choose $k$ as large as possible,  that is, $ k =  d $. We get that
\begin{align*}
\exists j\in[m] :  \sca{a_j,v}^2 &\geq (1-\delta) R^2 \sqrt[d/2]{\frac{1}{\norm[1]{\lambda}}}   
= (1-\delta) \sqrt[d/2]{\frac{\lambda_{i}}{\norm[1]{\lambda}}}    =: r_\delta
\end{align*}
is feasible and happens with probability at least $\frac{2}{\pi} \arccos(\sqrt{1-\delta})$.  Note that this is an asymptotically optimal bound,  since for $ d \to\infty$ we get $r_\delta  \to 1-\delta$. Choosing $\gamma$ such that  $1-\gamma = r_\delta$ and plugging this into Lemma \ref{lem:delta_bound}, we get\[
\Delta_{j,i,v} \geq  (\frac{1}{2}\minspeccorrel - 2(1-r_\delta))
\]
Let us now forget about the parameter $\delta $ by setting $r := r_{10^{-3}}$.\footnote{The $\delta$ of the next section will have a different meaning.}

\newpage
\subsection{Recovering all Components} \label{sec:all_comps}

Note that for any vectors of variables $X,Y,Z$
the \emph{SOS triangle inequality} 
\begin{align}\label{eq:sos_triangle_ineq}
\norm{X-Y}^2 \preceq 2\big(\norm{X-Z}^2 + \norm{Z-Y}^2\big)
\end{align}
holds. 
Indeed, by taking the triangle inequality \[
0 \leq \norm{x-z} + \norm{z-y} - \norm{x-y}
\]
to the power 2 it follows by some easy algebraic operations that \eqref{eq:sos_triangle_ineq} holds thus true for all $x,y,z\in \R^n$. Since the identity \eqref{eq:sos_triangle_ineq} is globally valid and of degree 2, it can be written as a sum of squares. 

\begin{lem}
	Let $\wtest$ be the testimony polynomial we chose in Prop. \ref{pro:guarantees_testimony} w.r.t. some $a_i \neq a_j$ and w.r.t. the interval $[a,b]$ (where $0 < a < \frac{1}{2} \kmin$ and $b > 2\kmax$). It satisfies the constraint \[
		(\wminmax - W) - g(\delta - \norm{X-c}^2) \succeq 0
	\]
	for some $g \geq 0$ with $\deg(g) = d-4$ 
	provided \[
	0\leq 2\delta \leq \min\{\norm{c-a_i}^2 - 2a,\: b - 2\norm{c-a_i}^2 \}
	\]
\end{lem}

\begin{proof} 
	Let us first search for a univariate identity of the kind
	\begin{align}
		(\wminmax - \Chebtest) \geq g\: (\Lambda - a)(b - \Lambda) \label{eq:unichebysos1}
	\end{align}
	for some $g\succeq 0$
	Note that such an identity would certify\[
	\wminmax \geq \wminmax - g(x)\: (x - a)(b - x) \geq \Chebtest(x) 
	\]
	for all $x\in [a,b]$. Let us strengthen \eqref{eq:unichebysos1} to
	\begin{align*}
	(\wminmax - \Chebtest) = g (\Lambda - a)(b - \Lambda) 
	\end{align*} 
	This is satisfied by the rational function \[
	g = \frac{(\wminmax - \Chebtest)}{(\Lambda - a)(b - \Lambda)} 
	\]
	which is actually a polynomial:  Indeed, recall that $a $ and $b $  are the boundary maximisers of   the Chebyshev polynomial with respect to $\mathcal{I}$. Hence, they are roots of $\wminmax - \Chebtest $. It remains to show that $g $ is a sum of squares,  which is the equivalent to being non-negative in the case of one variable. For example by  plotting $\wminmax - \Chebtest $ (or by an easy argumentation) it is clear that  this polynomial attains non-negative values precisely on $[a,b]$ (recall $d\equiv 2 \mod 4$). The same is true for $(\Lambda - a)(b - \Lambda)$. Hence, their quotient is globally non-negative, i.e. $g\succeq 0$. 
	Concatenated with the squared norm,  we get\[
	(\wminmax - \Chebtest(\norm{X-a_i}^2)) \succeq \tilde{g} \cdot (\norm{X-a_i}^2 - a)(b - \norm{X-a_i}^2) 
	\]
	where $\tilde{g} :=  g(\norm{X-a_i}^2)$
	Now, let us lower bound the right hand side. By the SOS triangle inequality we have 	\begin{align*}
	(\norm{c-a_i}^2 - 2\norm{X-c}^2) &\preceq 2\norm{X-a_i}^2 \qquad \text{and}\\
	(\norm{X-a_i}^2) &\preceq 2\norm{X-c}^2 + 2\norm{c-a_i}^2
	\end{align*}
	Therefore, 
	\begin{align*}
	I:= \quad &(\norm{X-a_i}^2 - a) \succeq (\frac{1}{2} \norm{c-a_i}^2 - a - \norm{X-c}^2)\quad \text{and}\\
	II:= \quad &(b - \norm{X-a_i}^2) \succeq (b - 2\norm{c-a_i}^2 - 2\norm{X-c}^2)
	\end{align*}
	Let us consider any $\delta > 0$ such that $\norm{c-a_i}^2 - 2a\geq 2\delta$ and $b - 2\norm{c-a_i}^2 \geq 2\delta$.\footnote{The factor of 2  is ugly and much likely due to the fact that the norm is not a polynomial. It infers that we can't take a Chebyshev polynomial with respect to the  interval $[\kmin-\delta,\kmax+\delta]$,  but that we have to leave a bigger safety margin $[\frac{1}{2}\kmin-\delta,2\kmax+\delta]$.} Then we can further estimate: 
	\begin{align*}
	I&\succeq ((\underbrace{\frac{1}{2}\norm{c-a_i}^2 - a -  \delta}_{\geq 0}) + (\delta - \norm{X-c}^2)) \succeq (\delta - \norm{X-c}^2) \quad \text{and}\\
	II&\succeq ((b - 2\norm{c-a_i}^2 - 2\delta) + (2\delta - 2\norm{X-c}^2)) \succeq 2(\delta - \norm{X-c}^2)
	\end{align*}
	Thus by multiplying two SOS polynomials and by expanding:
	\begin{align*}
	0 &\preceq (I - (\delta - \norm{X-c}^2))\cdot (II - 2(\delta - \norm{X-c}^2)) \\
	&= I\cdot II - (2\cdot I+II) (\delta - \norm{X-c}^2) + 2(\delta - \norm{X-c}^2)^2\\
	&= I\cdot II - ((2I+II) - 2(\delta - \norm{X-c}^2)) (\delta - \norm{X-c}^2)	\\
	&= I\cdot II - ((2b-a) - 2\delta + 2\norm{X-c}^2) (\delta - \norm{X-c}^2)\\
	&\preceq I\cdot II - \underbrace{\left(b+3a+4\delta + 2\norm{X-c}^2\right)}_{=:\:\: h\: \succeq \:0} (\delta - \norm{X-c}^2)
	\end{align*}
	The last step uses $6\delta \leq b-4a$, which we get by adding up the two constraints on $\delta$ we imposed in such a way that $\norm{c-a_i}^2$ cancels. We get \[
	I\cdot II \succeq h\cdot (\delta - \norm{X-c}^2)
	\]
	and therefore\[
	\wminmax - \Chebtest(\norm{X-a_i}^2) \succeq \tilde{g} \cdot \left(I\cdot II\right) \succeq \tilde{g} h\cdot (\delta - \norm{X-c}^2)
	\]
	\[
	\implies \exists g \succeq 0 : (\wminmax - \wtest) - g(\delta - \norm{X-c}^2) \succeq 0
	\]
\end{proof}

Thus, we have to show that we can choose $\delta $ greater or equal  to the distance $\norm{a_j-c}^2$ such that any feasible polynomial of the second round has to attain a small value at  the already recovered $a_j$. Now  we are left with the problem that the approximation accuracy depends on the interval size $b-a$ which depends on  $\delta $ which again depends on the approximation accuracy. So we need to get rid of all the cross-dependencies! We'll do so by bounding everything with respect to a constant $\minspeccorrel$ chosen in such a way that\footnote{The factor of $2$ is technical convenience.} \[
\min_{\substack{i,j \in[m]\\ i\neq j}} \projdist{a_i,a_j}^2 \geq 2\minspeccorrel
\]
Also,  we will  sometimes need to convert  distances  $x$ w.r.t. $d_{\mathbb{P}^{n-1}}^2$ to distances with respect to $d_{\Sn}^2$.\footnote{Note that $\norm{\cdot}^2 = 2d_{\Sn}^2$ on $\Sn$, which we will also use.} This is done via the  function $F(x) := 1-\sqrt{1-x}$. It might be difficult to read,  but there is no way around it,  so please do not be confused.

This turns all out to be very technical,  but it's done in Algorithm \ref{alg:approxDecomp_allsteps}. 
The parameters  shall be seen as follows:  For each  choice of parameters $(\minspeccorrel, \lambda_{\min})$ we get a different algorithm where lenient parameter choices  can decompose  a bigger class of tensors  but require a higher value of $d$  to achieve the same approximation accuracy. E.g. for the maximally restrictive choice $2\minspeccorrel = 1$ we'd get an algorithm for orthogonal tensor decomposition.

The parameter $\minspeccorrel$ is needed just for the  ``distinct recovery constraint'' which  ensures that we get  each time a different component.  Precisely, we need $\minspeccorrel$ and $\lambda_{\min}$ to compute the constant $\wminmax = (\bestratio \lambda_{\min})^{-1}$, which is in this algorithm the same for all components,  and an estimate $\tilde{\varepsilon}$ for the approximation error ($\minspeccorrel$ is needed for nothing else!).  Thus, it would be interesting to know if we could  work completely without $\minspeccorrel$. On the other hand,  $\lambda_{\min}$ is also used to have an ``accepting threshold'' $T_d(v) \geq (1-1/1000)^{d/2} \lambda_{\min}$ such that we can condition on
\begin{align*}
\exists j\in[m] :  \sca{a_j,v}^2 \geq (1-1/1000) \sqrt[d/2]{\frac{\lambda_{\min}}{\norm[1]{\lambda}}} =: r
\end{align*}
Now, let us deal with the cross-dependencies: For fixed $\minspeccorrel$ define for brevity $\rho := \frac{1}{2}F(\minspeccorrel) = \frac{1}{2}(1-\sqrt{1-\minspeccorrel})\in \landauO{\minspeccorrel} $.\footnote{The factor of $\frac{1}{2}$ is again technical convenience.}\\
For fixed $m,\minspeccorrel, \lambda_{\min}, T_0\: (= \norm[1]{\lambda})$ we say that $d$ satisfies the \emph{three useful constraints}, if:
\begin{enumerate}[(I)]
	\item $0.25\minspeccorrel \geq 2(1-r)$, where $r = (1-1/1000) \sqrt[d/2]{\frac{\lambda_{i}}{\norm[1]{\lambda}}}$
	\item $\rho \geq 4 F(\tilde{\varepsilon})$,\quad where $\tilde{\varepsilon}:= \frac{16T_0}{\minspeccorrel \cdot \bestratio \cdot \lambda_{\min}}$ \quad and \\$\bestratio := \frac{1}{2}\Chebtest(\frac{8 + 2\rho}{8 - \rho}) \in \landauO{(1 + \frac{3\rho}{8-\rho})^d} = \landauO{(1 + \minspeccorrel)^d} $
	\item $0.25\minspeccorrel \geq 4m \sqrt{2- 2(1-\tilde{\varepsilon})^{d/2}}$
\end{enumerate}

\begin{algorithm}[H]
	\caption{Approximate $v$-algorithm on $\Sn$, all components} 
	\label{alg:approxDecomp_allsteps}
	\textbf{Input: }Tensors $T_0 \in \hompolyspace{0}, \ldots, T_d\in \hompolyspace{d}$ \\
	\textbf{Parameters: } Minimum weight parameter $\lambda_{\min} > 0$. Minimum correlation parameter $\minspeccorrel > 0$.\\
	\textbf{Require: }$d \equiv 2 \mod 4$ and there should exist a simultaneous decomposition $T_k =~\weightdecomp[k]$ of the input with $m$  components on the unit sphere such that $\lambda_{\min} \leq \min_{i} \lambda_{i}$ and $\min_{\substack{i,j = 1,\ldots,m\\ i\neq j}} \projdist{a_i,a_j}^2 \geq 2\minspeccorrel$. \\
	Furthermore, $d$ should satisfy the \emph{three useful constraints} stated in this section. \\
	\textbf{Abbreviations:} 
	$r := (1-1/1000) \sqrt[d/2]{\frac{\lambda_{\min}}{\norm[1]{\lambda}}}$, $\rho := \frac{1}{2}(1-\sqrt{1-\minspeccorrel})$,\\ $\bestratio := \frac{1}{2}\Chebtest(\frac{8 + 2\rho}{8 - \rho}), \wminmax := (\bestratio \lambda_{\min})^{-1}, \tilde{\varepsilon} := \frac{16T_0}{\minspeccorrel \cdot \bestratio \cdot \lambda_{\min}}$ \\
	\textbf{Output: } Output a set of vectors $\{\alist[m]{c}\} \subseteq \Sn$ such that for each component $a_i$ there is one vector $c_l$ satisfying $\projdist{c_l, a_i}^2 \leq  \tilde{\varepsilon}$ and the factor of $\pm 1$ is correct.
	\\
	\textbf{Procedure: }
	\begin{algorithmic}[1]
		\REPEAT
		\STATE{$\mathbf{Set}$ $S := T_d$ and $i:=1$ }
		\STATE{$\mathbf{Choose}$ $v \in \Sn$ from the uniform distribution conditioned on} \[
		S(v) \geq r - 4(i-1) \cdot \sqrt{1- (1-\tilde{\varepsilon})^{d/2}}\]
		\STATE{$\mathbf{Solve}$ the SOS optimisation problem} \[
		\max \rezprod{W \lfp{v}^2}{\sum_{k = 2}^d T_k} \quad (\ast)
		\]
		over all $W\succeq 0$ satisfying $\deg(W)\leq d-2$ and $\rezprod{W }{\sum_{k = 0}^{d-2} T_k} =~1$.
		\STATE{$\mathbf{Let }$ $W^\ast $ denote the output of this optimisation problem.} 
		\STATE{$\mathbf{Compute}$ the matrix $M := \rezprod{W^\ast \tens \id}{\sum_{k = 2}^{d} T_k}$ }
		\STATE{$\mathbf{Compute}$ an eigenvalue decomposition of $M$ with corresponding eigenvectors of unit length.}
		\STATE{$\mathbf{Let }$ $u$ denote the computed eigenvector corresponding to the largest eigenvalue $\mu$.}
		\STATE{$\mathbf{Switch}$ $u := -u$} \textbf{if} $\sca{L,u} < 0$ for $L := \sum_{i= 1}^m \lambda_i \Wst(a_i) a_i$
		\STATE{$\mathbf{Output }$ $c_i := u$} as the component (and $\rho_i := \frac{1}{W^\ast(c_i)}$ as the weight).
		\STATE{$\mathbf{Add }$ the constraint \[
			\exists g\succeq 0 \quad \wminmax - W  \succeq g\cdot (2 - 2\sqrt{1 - \tilde{\varepsilon}} - \|X - c_i\|^2)
			\]} to the SOS optimisation problem $(\ast)$.
		\STATE{$\mathbf{Set}$ $S := S - \lfp{u}^d$ and \textbf{increment} $i$. }
		\UNTIL{the problem $(\ast)$ becomes infeasible}
	\end{algorithmic}
\end{algorithm}
The constraints are needed to ensure that the approximation error is  small enough such that the algorithm stays stable over $m $ rounds. All of these constraints are feasible for sufficiently high $d $ (supposed that  $0.25\minspeccorrel > 2/1000$ for the first one. Note that $(1-(1-\tilde{\varepsilon})^d) \in \landauO{\tilde{\varepsilon}}$ for the last one and that $\tilde{\varepsilon}$ decreases exponentially with $d$). But  $d $ does not even need to be very high (up to constants):  We get a logarithmic dependency of $d $  on the parameters.

Note that the ratio $\bestratio$ is correctly estimated:  With the testimony from Section \secref{sec:testimony} we get that the interval $\I = [a,b] $ can actually be chosen as \[
[\frac{1}{2}\kmin - 2F(\tilde{\varepsilon}), 2\kmax + 2F(\tilde{\varepsilon})]
\] as long as the approximation accuracy  w.r.t. $d_{\mathbb{P}^{n-1}}^2$ is less or equal $\tilde{\varepsilon}$. Using the estimates $\kmax = \max_{i\neq j} \norm{a_i-a_j}^2 \leq 2^2 = 4$ and $\frac{1}{2}\kmin = \frac{1}{2}\min_{i\neq j} \norm{a_i-a_j}^2 = \min_{i\neq j} \spheredist{a_i,a_j}^2 \geq (1-\sqrt{1-\minspeccorrel}) = 2\rho \in \landauO{\minspeccorrel}$,  we see that \[
\bestratio_{\text{true}} := \frac{1}{2}\Chebtest\left(\frac{8 + 2\rho}{8 - 2\rho + 4 F(\tilde{\varepsilon})}\right)
\] would actually be valid\footnote{This shall mean: It is  attained for a testimony in  the feasible space.} if the true approximation error   $\varepsilon$ (w.r.t. $d_{\mathbb{P}^{n-1}}^2$) was less or equal $\tilde{\varepsilon}$. We simplified this to \[
\bestratio := \frac{1}{2}\Chebtest\left(\frac{8 + 2\rho}{8 - \rho }\right)
\]
which will be valid  if and only if we can show that it is possible to  choose $d $ so high that the ``converted'' approximation error $4 F(\varepsilon)$ is  less or equal  to $\rho$. The second useful constraint states $\rho \geq 4 F(\tilde{\varepsilon})$, so it remains to show that the approximation error gets smaller than $\tilde{\varepsilon}$ for sufficiently high $d $.

Now,  let's show this: Suppose the parameters $\minspeccorrel$ and $\lambda_{\min}$ were correctly specified for given input tensors $T_0,\ldots T_d$  with exact decomposition $T_k = \weightdecomp[k]$.  Furthermore,  suppose that $d $ satisfies the 3 useful constraints with respect to the given input and parameter choices. 

Choose the first round's $v\in\Sn$ such that $T_d(v) \geq (1-\frac{1}{1000})^{d/2}\lambda_{\min}$. By  the results of \secref{sec:one_component},  we can recover a proxy $c_1 = u$ to the component $a_j$ maximising $\lfv{j}^2$ such that\[
\sca{c_1, a_j}^2 \geq  \big(1 - 2 \wminmax\cdot \norm[1]{\lambda}\cdot \frac{\speccorrel}{ \|a_j\|^2} \cdot \frac{\max_{i\neq j} \Delta_{j,i,v}}{\min_{i\neq j} \Delta_{j,i,v}}\big)
\]
Note that $\max_{i\neq j} \Delta_{j,i,v} \leq \max_{i\neq l} \lfv{l}^2 - \lfv{i}^2 \leq \speccorrel$ by definition and since we are on the unit sphere even $\speccorrel \leq \max_{i\neq l} (\spec{a_i\transp{a_i}}  + \spec{a_l\transp{a_l}}) = 2$. Since $\min_{i\neq j} \Delta_{j,i,v}\geq \minspeccorrel - 2(1-r)$, we get \[
\sca{c_1, a_j}^2 \geq  \big(1 - \frac{8 \wminmax\cdot \norm[1]{\lambda}}{\minspeccorrel - 2(1-r)}\big) \geq \big(1 - \underbrace{\frac{8\cdot \norm[1]{\lambda}}{\bestratio \cdot \lambda_{\min} \cdot (\minspeccorrel - 2(1-r))}}_{ = \varepsilon}\big)
\]
and hence $\projdist{c_1, a_j}^2 \leq \varepsilon$
with accuracy at least $\varepsilon =  \frac{8\cdot \norm[1]{\lambda}}{\bestratio \cdot \lambda_{\min} \cdot (\minspeccorrel - 2(1-r))}$.    From the first useful constraint we get that 
\begin{align}
\minspeccorrel - 2(1-r)\geq 0.75 \minspeccorrel \geq 0.5\minspeccorrel \label{eq:1-r_estimate}
\end{align} and therefore\footnote{Note that in  this estimation we've left some extra space ``$0.75 \minspeccorrel \geq 0.5\minspeccorrel$'' which we will need for the subsequent rounds to work.} \[
\varepsilon \leq \tilde{\varepsilon} = \frac{16\norm[1]{\lambda}}{\minspeccorrel \cdot \bestratio \cdot \lambda_{\min}}
\]
will indeed be sufficiently small (at least in the first round). We have already argued before  that the  decision between $\pm u$ is not the most critical issue here,  so let us just assume that it is correct for sufficiently small $\tilde{\varepsilon}$.\footnote{It would be a problem though if the moment decomposition could have  distinct components satisfying $a_i = -a_l$ for $i\neq l$. (Note that this case is excluded due to $\min_{ i\neq j} \projdist{a_i,a_j}^2 \geq \minspeccorrel$). Then $\Wst$ could concentrate on both of them equally,  since the discriminatory polynomial $\lfp{v}^2$ would satisfy $\lfv{j}^2 = \lfv{l}^2$.  This problem is a suboptimality   that stems from the fact that we weren't able to prove in \secref{sec:v_choice} that we could require  $\exists j : \sca{a_j,v} \geq \sqrt{r}$ instead of $\exists j : \sca{a_j,v}^2 \geq r$.} But then we have\[
\norm{c_1 - a_j}^2 = 2(1 - \sca{c_1, a_j}) \leq 2(1 - \sqrt{1 - \tilde{\varepsilon}})
\]
Knowing  $\wminmax$ and $\tilde{\varepsilon}$, we can write down the constraint
\[
\exists g\succeq 0 \quad \wminmax - W  \succeq g\cdot (2 - 2\sqrt{1 - \tilde{\varepsilon}} - \|X-c_i\|^2)
\]
by adding a new ``slack variable'' $g $ to the SOS programme. This  constraint  ensures indeed that the optimiser of the 2nd round satisfies $\Wst(a_j) \leq \wminmax$ (such that we will not recover the same component twice),  since the distance of $c_1$ and  $a_j$ is less or equal $2-~2\sqrt{1 - \tilde{\varepsilon}}$.  As we saw in \secref{sec:all_comps},  we needed a safety margin  of at least  $2F(\varepsilon)~=~2 -~2\sqrt{1 -\varepsilon}$ in the interval of the Chebyshev polynomial to ensure that the testimony satisfies  the distinct recovery  constraint. This is satisfied since we saw $\varepsilon \leq \tilde{\varepsilon}$. 

It is left to show that this argumentation does not break until the $m$-th round.
The only thing that changes from round 2 and onwards is that we can't just take some $v $ which achieves just a large value of $T_d(v) $:  We need to ensure that $v $ is highly correlated with some component $a_i\neq a_j $.  This would be easy if we had access to the tensor $S_{\text{correct}} := \sum_{i\neq j} \lambda_{i} \lfp{a_i}^d$. To this end, we need   the 3rd useful constraint which essentially guarantees us that $S - \lfp{u}^d = S - \lfp{c_1}^d$ is a sufficiently good approximation to $\sum_{i\neq j} \lambda_{i} \lfp{a_i}^d$. We can show that  they have to attain similar values on $v $ by using the  evaluation property of Reznick's scalar product (Prop. \ref{pro:rezeval}):
\begin{align*}
\norm[F]{S_{\text{correct}}  - S}^2 &= \norm[F]{\lfp{a_j}^d - \lfp{u}^d}^2 \\
&\leq \norm[F]{\lfp{a_j}^d}^2 + \norm[F]{\lfp{u}^d}^2 - 2\rezprod{\lfp{a_j}^d}{\lfp{u}^d}\\
&= \sca{a_j,a_j}^d + \sca{u,u}^d - 2\sca{a_j,u}^d \\
&= 2 - 2\sca{a_j,u}^d \geq 2- 2(1-\tilde{\varepsilon})^{d/2} \in \landauO{\tilde{\varepsilon}}
\end{align*}
In particular, since \[
|(S-S_{\text{correct}}) (v)| = |\rezprod{S - S_{\text{correct}}}{\lfp{v}^d}| \overset{\text{Cauchy-Schwarz}}{\leq} \norm[F]{S_{\text{correct}}  - S} \quad\in \landauO{\sqrt{\tilde{\varepsilon}}}
\]
we get that in the 2nd round we can sample with the same success probability as in round 1 some $v\in\Sn$ such that $S(v) \geq r - \landauO{\sqrt{\tilde{\varepsilon}}}$ and any such $v$ satisfies $S_{\text{correct}} (v) \geq r - 2 \landauO{\sqrt{\tilde{\varepsilon}}}$.

This shows that in the end, what has changed is that we have to start our estimations for $\tilde{\varepsilon}$ with a slightly lower value of $r $.  We'll show that our estimations for $\tilde{\varepsilon}$ were conservative enough such that they do not have to be changed. For the sake of simplicity,  we demonstrated the above for round 2.  Generally, in round $i $ we are left with some $v\in\Sn$ such that \[
S_{\text{correct}}(v) \geq r - 2(i-1) \sqrt{2- 2(1-\tilde{\varepsilon})^{d/2}} \geq r - 2m \sqrt{2- 2(1-\tilde{\varepsilon})^{d/2}} = r - m\landauO{\sqrt{\tilde{\varepsilon}}}
\]
Set $r' := r - 2m \sqrt{2- 2(1-\tilde{\varepsilon})^{d/2}}$. 
This is where the 3rd useful constraint comes in: It guarantees  that $0.25\minspeccorrel \geq 4m \sqrt{2- 2(1-\tilde{\varepsilon})^{d/2}} = 2(r - r')$, whence in each round  we can still deduce the estimate 
\begin{align*}
\minspeccorrel - 2(1-r') &= \minspeccorrel - 2(1-r) + 2(r' - r)\\ 
&\geq 0.75 \minspeccorrel + 2(r' - r) \geq 0.5\minspeccorrel 
\end{align*} from Equation \eqref{eq:1-r_estimate}.
Hence, the approximation guarantee $\tilde{\varepsilon}$ will not get lower over the course of the $m$ rounds of the algorithm.

\newpage
\subsection{Noise Stability and Tensor Decomposition}
In the current state, the algorithm is not very stable to noise: For instance,  a single component with negative weight could make the optimisation problem unbounded. This can be changed,  though: We just need to cap  the Frobenius norm $\norm[F]{W} \leq D$ for some constant $D$.\footnote{Then we also have $\norm[F]{{W\sca{V,X}^2}}\leq \tilde{D}$ for some polynomial $\tilde{D}\in\R[V]_{\leq 2}$.  The cap  needs  of course to be chosen such that the testimonies are still members of the feasible space.  Capping can be realised  for instance by adding linear constraints on the coefficients of $W $ or by constraining the trace of the Gram matrices in the SDP.}  Indeed,  if the  correct input polynomial $\sum_{k = 0}^d T_k$ is disturbed by some  noise polynomial  $E $ where $\norm[F]{E}$ is several orders of magnitude smaller than the Frobenius norm of the  input polynomial,  then we have by Cauchy-Schwarz,  that\begin{align*}
\rezprod{W}{E} \leq \norm[F]{W} \norm[F]{E} \leq D\norm[F]{E}
\end{align*}

Of course the  Lemmas giving guarantees on the optimiser would have to be  reproven  while carrying the additional noise term with us. A full-fledged noise analysis would go beyond the scope of this thesis. Philosophically spoken,  capping the Frobenius norm limits the damage that  a  wrong summand can do:  In the special case that the   noise term is, for instance, a single component  $a_i $ of unit length with small negative weight $-1 \ll \lambda_{i}  <  0$,  we get that \[
W(a_i) \leq \norm[F]{W} \cdot \norm[F]{\sum_{k = 0}^{d-2} \lfp{a_i}^k}  \leq D \sum_{k = 0}^{d-2} \norm{a}^{2k} \leq D(d-2) \ll \frac{1}{|\lambda_{i}|}
\]
Hence $\lambda_{i}W(a_i)$ will be very small.\\

Regarding tensor decomposition,  it should be said that the fake moment generation procedure we gave  in \secref{sec:valgotd} is not the most ideal choice for the  setting of approximate recovery.  The problem is again  that the random choice of $w $ can cast a well-conditioned problem into an ill-conditioned one. Therefore it's better to use the canonical norm-scaled fake moments of even degree that we presented in \secref{sec:can_proc_momgen}.

This also allows us to assume without loss of generality that the components have unit length,  which we have seen to be very practical.
However,  this means that we'd have to do some more work for the case of tensor decomposition:  First of all, we need a new testimony having only even degree coefficients. We can  employ essentially the same idea as in \secref{sec:can_proc_momgen}:  Take a univariate Chebyshev polynomial $\chebyinterval + 1$ with respect to the  interval $\I := [0, \max_{i\neq j}~\norm{a_i}^2~\norm{a_j}^2~-~\sca{a_i,a_j}^2]$ and concatenate it with the ``Cauchy-Schwarz gap'' $\norm{X}^2 \norm{a_j}^2 - \sca{X,a_j}^2$. This will separate the components up to their length and a factor of $\pm 1 $. Of course we'd still have to prove  that a similar sums of squares constraint as in \secref{sec:all_comps} exists such that we can avoid to recover the same component twice.\\

\noindent
These topics may be addressed in detail in some future research.

\newpage
\section{Conclusions and Open Problems}

Sums of squares optimisation has  shown to be a valuable tool  for moment decomposition when sufficient data is given. The algorithmic approximation scheme we  developed in Section \secref{sec:effdecomp} can  compute the decomposition  to arbitrary  high accuracy,  as long as we have access to all the moments of arbitrary high degree. If only a finite number of moments is  accessible,   then the approximation accuracy is fixed,  but can eventually be increased  by running some local searching procedure afterwards. The  error depends on the condition of the problem,  whence decompositions with  well-separated  components are easier to recover. This is a very much expected phenomenon,  even though it does not occur in more  elementary procedures such as Jennrich's,  where the working  requirement is just a relation between $m$ and $n$.\\

On the other hand,  on many instances the algorithm might produce results way better than the worst case bounds suggest. In particular,  we've seen that  for certain  instances (e.g.  highly symmetric ones such  as  the components lying on a simplex) exact recovery is possible already for very low degree.\\

If this thesis got them interested, the reader is encouraged to  dedicate further research to this topic. The question which is probably the most interesting is  whether  some  variant of the $v$-algorithm could be combined with a procedure for  the generation of higher order moments to further improve on the approximation accuracy -- at the expense of larger computation time. In the following,  we've collected some  opportunities related  directly to the $v$-algorithm.
\newpage
\paragraph{Related to the $v$-Algorithm}
Although the algorithm works,  there is still a lot of  potential for optimisation.  We would like to collect the most notable suboptimalities of the algorithm to motivate further research:

\begin{enumerate}
	\item The  issue likely to be the most important is to find  strategies for  the choice of $v $ other than  just sampling   it at random. To this end, it can be helpful to require some   prior knowledge  on the domain of the $a_i $.  We would particularly be interested in deterministic procedures for the choice of discriminatory polynomials.
	\item  In hindsight,  some of  our techniques seem to be  way more suitable for even degree tensor decomposition  than for moment decomposition:  This is particularly true for the  conditioning we performed. It seems also that the  radially symmetric testimonies we chose do not work too well with most of the other correlation-based estimates. While the arguments could be made to work,  many technical estimates were  needed.
	
	Everything could potentially become simpler when considering tensor decomposition. We've seen that the unit sphere assumption is not that much of a restriction in this case due to the possibility to generate the even degree norm-scaled  fake moments. Furthermore,  then we can  concatenate our univariate Chebyshev  polynomials with the Cauchy-Schwarz gap, which has the advantage that it  depends directly on the square  correlation. This has the potential to make the analysis much easier.
	\item  The choice of radially symmetric Chebyshev polynomials  seems weirdly heuristic.  On the one hand,  restricting to radially symmetric polynomials is a potential waste of the  valuable and  computationally expensive degree:  After the substitution with the squared norm  we can only plug-in polynomials of degree less or equal $d-2$.  On the other hand,  it's likely that  multivariate polynomials will achieve better separation results. 
	\item  It also comes with the seemingly unnecessary restriction $ d \equiv 2 \mod 4 $:  If we would have designed the algorithm from the very beginning for moment decomposition on the unit sphere,  then we could have used that on the unit sphere the metric $\spheredist{x, a_j}^2$ can be represented  by the linear polynomial $ 1-\sca{X, a_j}$. If we had  concatenated  the Chebyshev polynomials with these linear polynomials instead of the norm,  our life would have been much easier in the setting of our main result. 
\end{enumerate}

\newpage
\thispagestyle{empty}
\section*{Declaration of Authorship}\vspace{15pt}
I declare that the submitted thesis \vspace{18pt}\\ 
	\begin{center}
		{\Large \textit{A new Algorithm for Overcomplete Tensor Decomposition based on Sums-of-Squares Optimisation}}\vspace{22pt}\\
	\end{center}
is my own unaided work. Direct or indirect sources are acknowledged as references. This thesis was not previously presented to another examination board and hasn't been published before.\vspace{15pt}\\

\noindent
Konstanz, 5 October 2018\vspace{15pt} \\ \hfill \sigline{\hfill Alexander Taveira Blomenhofer} 

\newpage
\section*{\textrm{Supplement} \textrm{I} - Existence of  Optimal Solutions}
\addcontentsline{toc}{section}{\textrm{Supplement I}} 

\textit{We did not yet give an argument why we can always assume the occuring SOS optimisation problems to  have an optimal solution. While in the exact case, we were able to give explicit examples (the interpolation polynomials), for the general case it is actually not clear whether or not optimal solutions exist. However, note that in all of our analysis we did merely use one property of the optimal solution $\Wst$, namely that the objective function achieves on $\Wst$ some value at least as high as the values on all of the testimony polynomials} $$\wtest~=~\chebyinterval[(d-2)/2](\norm{X-a_i}^2) + 1$$

\noindent
Let us reconsider the basic SOS optimisation problem from above 
\begin{align*}
(\ast)\qquad\qquad\maxi&\quad \rezprod{W \lfp{v}^2}{\sum_{k = 0}^d T_k} \quad  \nonumber \\
\text{over all}& \quad W\in\polyspace{d-2},\: W \succeq 0 \text{ satisfying} \\
&\quad \rezprod{W }{\sum_{k = 0}^{d} T_k} = 1\nonumber
\end{align*}
With the expectation operator $\E_{\mu}$ of $\mu = \sum_{i= 1}^m \lambda_{i} \delta_{a_i}$, this can be rewritten as
\begin{align*}
(P_{\mu, v})\qquad\qquad\maxi \quad & \E_{\mu}[W\cdot\lfp{v}^2] \quad  \\
\text{over all}\quad& W\in\polyspace{d-2},\: W \succeq 0 \text{ satisfying} \\
&\E_{\mu}[W] = 1\nonumber
\end{align*}
One problem is that the feasible space is in general not compact: Indeed, assume there is some sum of squares polynomial $W_0$ vanishing on every $a_i$. Then \[
\E_{\mu}[W_0] = \E_{\mu}[W_0\cdot\lfp{v}^2] = 0
\] 
Thus it could be possible to  traverse the feasible space  in a direction where the value of the objective function does not change. Doing so would of course yield an unnecessary increase of the optimisers ``complexity'' (complexity can be measured by an appropriate norm on $\polyspace{d-2}$). This can easily be prevented by adding a complexity  bound.\footnote{It is not clear whether such a complexity bound will affect the quality of the optimal solution. However, we are fine with an optimal solution that is ``at least as good'' as the testimonies.} The most basic approach in doing so would be to cap the  absolute values of the coefficients of  $W $ using linear constraints.  We will use a slightly different approach: With respect to any continuous measure $\nu $, $\E_{\nu}[W]$ will be strictly positive for any non-zero SOS polynomial $W$. Therefore we can take for instance  $\nu = \mathcal{U} $ as  the uniform probability measure on the set $B:= \conv \Sn  $ and cap the size of $\E_{\mathcal{U}}[W]$. We just have to verify that every testimony polynomial still lies in the feasible space. Hence, we should upper bound \[
\E_{\mathcal{U}}[\chebyinterval[(d-2)/2](\norm{X-a}^2)]
\]
for an arbitrary unit vector $a\in\Sn$ and any reasonable choice of the interval $I$. (Recall that for $\mu = \sum_{i= 1}^m \lambda_{i} \delta_{a_i}$ we used $I = [\min_{i\neq j} \norm{a_i-a_j}^2,\: \max_{i\neq j} \norm{a_i-a_j}^2] \subseteq (0, 4]$ if all $a_i\in \Sn$). W.l.o.g. we may assume $a = e_1$ since $\mathcal{U}$ is rotation-symmetric. $\chebyinterval[(d-2)/2]$ attains its maximum on $[0,4]$ on the points that have the largest distance to the midpoint of $\I$. Hence  $\max_{x\in [0,4]} \chebyinterval[(d-2)/2](x) \leq \max\{\chebyinterval[(d-2)/2](0), \chebyinterval[(d-2)/2](4)\} =: C_{\max}$.\\

But then we get the brute upper bound that for every $I \subseteq [0, 4]$ and $a\in \Sn$: \[
\E_{\mathcal{U}}[\chebyinterval[(d-2)/2](\norm{X-a}^2)] \leq C_{\max} \E_{\mathcal{U}}[1]  = C_{\max} 
\]
And therefore for $\wtest = \E_{\mu}[\chebyinterval[(d-2)/2](\norm{X-a_j}^2)+1]^{-1} \left(\chebyinterval[(d-2)/2](\norm{X-a_j}^2) + 1\right)$ we have:
\[
\E_{\mathcal{U}}[\wtest] \leq\frac{ C_{\max} + 1}{ \lambda_{\min} }  
\]
due to $$\E_{\mu}[\chebyinterval[(d-2)/2](\norm{X-a_j}^2) +1 ] \geq \lambda_{\min}\delta_{a_j}[\chebyinterval[(d-2)/2](\norm{X-a_j}^2) + 1] = \lambda_{\min} (\chebyinterval[(d-2)/2](0) + 1)\geq \lambda_{\min}$$
The last estimate holds since the Chebyshev polynomial is nonnegative outside of $\I$.\\

With this bound we get a new \emph{complexity-truncated} SOS optimisation programme
\begin{align*}
(P_{\mu, v, C})\qquad\qquad\maxi \quad & \E_{\mu}[W\cdot\lfp{v}^2] \quad  \\
\text{over all}\quad& W\in\polyspace{d-2},\: W \succeq 0 \text{ satisfying} \\
&\E_{\mu}[W] = 1\nonumber\\
&\E_{\mathcal{U}}[W] \leq\frac{ C_{\max} + 1}{ \lambda_{\min} }  
\end{align*}
By adding the new constraint, the feasible space gets compact with respect to the norm $\norm[\mathcal{U}]{W} := \E_{\mathcal{U}}[|W|]$ on $\polyspace{d-2}$ and therefore there exists an optimiser $\Wst$ of $(P_{\mu, v, C})$.

\end{document}